\documentclass[11pt]{amsart}
\usepackage{amssymb}
\usepackage{amsmath}
\usepackage{amsthm}
\usepackage{graphicx}
\usepackage{macros}
\usepackage{bm}
\usepackage[usenames,dvipsnames]{xcolor}
\usepackage{subcaption}
\usepackage{tikz}
\usepackage{pgfplots} 
\usetikzlibrary{3d,calc}
\usepackage{mwe}

\usepackage[T1]{fontenc}
\usepackage[english]{babel}
\pgfplotsset{compat=newest}

\usepackage{float}
\usepackage{color}
\usepackage{hyperref}
\hypersetup{
	colorlinks,
	citecolor= Brown,
	linkcolor= link,
	urlcolor= link
}

\usepackage{appendix}
\usepackage{array}
\usepackage{footnote}
\usepackage{booktabs}
\usepackage{hhline}
\makesavenoteenv{tabular}

\usepackage{algpseudocode}
\usepackage{algorithm}

\usepackage{caption}

\usepackage{bbm}

\usepackage{enumitem}

\usepackage{pgfplots} 
\usepackage{textcomp}

\usepackage{scrextend}
\deffootnote{0em}{1.6em}{\enskip}
\usepackage{neuralnetwork}
\graphicspath{{./figures/}}

\usetikzlibrary{positioning}	
\tikzset{%
	every neuron/.style={
		circle,
		draw,
		minimum size=1cm
	},
	neuron missing/.style={
		draw=none, 
		scale=4,
		text height=0.333cm,
		execute at begin node=\color{black}$\vdots$
	},
}

\usepackage[top=2cm,bottom=2cm,left=2.5cm,right=2.5cm]{geometry}
\definecolor{link}{rgb}{0.18,0.25,0.63}
\definecolor{myred}{rgb}{0.7,0.25,0.2}
\definecolor{mygray}{rgb}{0.8,0.8,0.8}

\numberwithin{equation}{section}

\newcommand{\cmmnt}[1]{}

\makeatletter
\g@addto@macro{\endabstract}{\@setabstract}
\newcommand{\authorfootnotes}{\renewcommand\thefootnote{\@fnsymbol\c@footnote}}%
\makeatother

\definecolor{myred}{rgb}{0.78,0.20,0.00}

\newtheorem{definition}{Definition}[section]
\newtheorem{theorem}[definition]{Theorem}
\newtheorem{proposition}[definition]{Proposition}
\newtheorem{lemma}[definition]{Lemma}
\newtheorem{remark}[definition]{Remark}

\newcommand{\norm}[1]{\left\| #1 \right\|}
\newcommand{\abs}[1]{ \left\vert #1 \right\vert}
\newcommand{\set}[1]{ \left\{ #1  \right\}}

\newcommand{\R}{ \mathbb{R}}
\newcommand{\N}{ \mathbb{N}}

\begin{document}

 \begin{center}
 \Large
   \textbf{\textsc{A descent algorithm for the optimal control of {ReLU} neural network informed {PDEs} based on approximate directional derivatives}} \par \bigskip \bigskip
   \normalsize
    \textsc{Guozhi  Dong}\textsuperscript{$\,1$},
  \textsc{Michael Hinterm\"uller}\textsuperscript{$\,2,$$3$}, \textsc{Kostas Papafitsoros}\textsuperscript{$\,4$}
\let\thefootnote\relax\footnote{
\textsuperscript{$1$}School of Mathematics and Statistics, HNP-LAMA, Central South University, Lushan South Road 932, 410083 Changsha, China
}
\let\thefootnote\relax\footnote{
\textsuperscript{$2$}Institute for Mathematics, Humboldt-Universit\"at zu Berlin, Unter den Linden 6, 10099 Berlin, Germany}

\let\thefootnote\relax\footnote{
\textsuperscript{$3$}Weierstrass Institute for Applied Analysis and Stochastics (WIAS), Mohrenstrasse 39, 10117 Berlin, Germany}

\let\thefootnote\relax\footnote{
\textsuperscript{$4$}School of Mathematical Sciences, Queen Mary University of London, Mile End Road, E1 4NS, UK}

\let\thefootnote\relax\footnote{
\vspace{-12pt}
\begin{tabbing}
\hspace{3.2pt}Emails: \= \href{mailto:guozhi.dong@csu.edu.cn}{\nolinkurl{guozhi.dong@csu.edu.cn}},
\href{mailto:Hintermueller@wias-berlin.de}{\nolinkurl{hintermueller@wias-berlin.de}},
 \href{mailto: k.papafitsoros@qmul.ac.uk}{\nolinkurl{k.papafitsoros@qmul.ac.uk}}
  \end{tabbing}
}
\end{center}
\vspace{-0.8cm}

	\begin{abstract}
 We propose and analyze a numerical algorithm for solving a class of  optimal control problems for  learning-informed semilinear partial differential equations. The latter  is a class of PDEs with constituents that are in principle unknown and are approximated by nonsmooth ReLU neural networks.  We first show that a direct smoothing of the ReLU network with the aim to make use of classical numerical solvers can have certain disadvantages, namely potentially introducing multiple solutions for the corresponding state equation. This motivates us to devise a numerical algorithm that treats directly the nonsmooth optimal control problem,  by employing a descent algorithm inspired by a bundle-free method. Several numerical examples are provided and the efficiency of the algorithm is shown.
		\vskip .3cm
		%{\bf AMS subject classifications.} \ { }
		%\vskip .3cm	
	\end{abstract}

\section{Introduction}

\subsection{Context and motivation}

%\vspace{1cm}

In this paper we study a numerical algorithm for the following artificial neural network based optimal control problem:

\begin{equation}\label{P_intro}\tag{$P_{\mathcal{N}}$}
\begin{aligned}
&\text{minimize }\quad J(y,u):= \frac{1}{2} \|y-g\|_{L^{2}(\Omega)}^{2} + \frac{\alpha}{2} \|u\|_{L^{2}(\om)}^{2},\quad \text{ over } (y,u)\in H_{0}^{1}(\om) \times L^{2}(\om),\\
&\text{subject to }\left \{
\begin{aligned}
-\Delta y + \mathcal{N}(\cdot, y)&=u, \;\; \text{ in }\Omega,\\
y&=0, \;\; \text{ on }\partial \Omega,
\end{aligned}
\right. \quad \text{ and }\quad  u\in \mathcal{C}_{ad}.
\end{aligned}
\end{equation}
Here $\Omega$ denotes an open, bounded, Lipschitz domain in $\RR^{d}$ with boundary $\partial \om$, $d\ge 2$, $g\in L^{2}(\om)$ is a given desired state, $\alpha>0$ is fixed, and $\mathcal{C}_{ad}$ is an admissible set for the control $u$, which is assumed to be a nonempty, closed and convex subset of $L^{p}(\om)$ for some $p\ge 2$. The state (variable) is $y$ which, given a control $u$, solves a semilinear elliptic partial differential equation (PDE), the state equation. The term that renders the above problem nonstandard is the function $\mathcal{N}: \RR^{d}\times \RR\to \RR$, a constituent of the PDE acting as a constraint for the minimization problem.  In fact, throughout we assume that $\mathcal{N}$ represents a ReLU (Rectified Linear Unit) artificial neural network, that is, a neural network that has the ReLU  $\sigma(t):=\max(t,0)$ as its activation function. We note that the ReLU is one of the most common and advantageous activation functions in deep learning \cite{bengio2013representation, glorot2011deep}, see Section \ref{sec:deep_learning} for more details and definitions. As a result, $\mathcal{N}$ is in general a nonlinear and nonsmooth function.  We mention that here we consider $\mathcal{N}$ to be monotonically increasing in the variable $y$ which guarantees the uniqueness of a solution to the state equation, resulting in a well-defined control-to-state map.

The semilinear PDE in \eqref{P_intro} is thus an instance of a \emph{learning-informed} PDE, a concept that was  introduced in \cite{DonHinPap20} and further explored  recently in other works \cite{aarset2022learning, kaltenbacher2021discretization}. We assume that it forms an approximating model to an unknown ground truth physical law expressed by
\begin{equation}\label{intro:state_f}
\left \{
\begin{aligned}
-\Delta y + f(\cdot, y)&=u, \;\; \text{ in }\Omega,\\
y&=0, \;\; \text{ on }\partial \Omega,
\end{aligned}
\right. 
\end{equation}
with the function $f$ being some unknown nonlinearity, which is approximated by the network $\mathcal{N}$.  This could be achieved for instance in a setting where we have at our disposal a dataset 
\[D:=\{(y_{i}, u_{i}):\; \text{$y_{i}$ (approximately) solves \eqref{intro:state_f} for $u_{i}$},\;i=1\ldots, n_{D}\},\]
which corresponds to some pre-specified controls and associated state responses, collected for example through measurements or computations. This dataset can be used towards evaluation instances of $f$ via $f(x_{j}, y_{i}(x_{j}))\simeq u_{i}(x_{j})+ \Delta y_{i} (x_{j})$ where $\{x_{j}\}_{j=1}^{\ell}$ is an appropriate discrete collection of points in the domain. Using these instances as a training set, a neural network $\mathcal{N}$ can be trained in the context of supervised learning in an \emph{offline phase}, and take the role of an approximating map for the unknown $f$. 
Applications of the above framework were considered in \cite{DonHinPap20} in order to learn the physical law that governs the separation of a fluid into two immiscible phases as well as to learn the physical law behind magnetic resonance imaging (MRI), where, instead of a PDE, a system of ordinary differential equations (ODEs) acts as a constraint \cite{DonHinPap19}.

Several theoretical aspects of the  optimal control problem \eqref{P_intro} were studied in detail in \cite{DonHinPapVoe22}. There, existence and uniqueness of solutions to the state equation were shown, as well as continuity and directional differentiability properties of the control-to-state map.  The main challenge here is the aforementioned nonsmoothness of $\mathcal{N}$ due to the ReLU. In fact, it can be shown that the set of functions represented by ReLU neural networks coincides with the family of piecewise affine maps. In general one does not expect the associated control-to-state map $S_{\mathcal{N}}$ to be G\^ateaux differentiable which poses difficulties in the derivation of first-order optimality conditions for the optimal control problem. Addressing this latter aspect, stationary conditions were derived in the companion work \cite{DonHinPapVoe22} based on generalized differentiability concepts. In this paper, we focus on establishing algorithms for the numerical solution of \eqref{P_intro} towards the approximation of so-called B-stationary points.

With the desire of making use of classical numerical solvers, an immediate approach to solving \eqref{P_intro} would be to regularize the problem by smoothing the nonsmooth component $\mathcal{N}$. As a consequence, the classical Karush-Kuhn-Tucker theory for stationarity (see, e.g., \cite{Zowe1979}) becomes available and solvers from (smooth) nonlinear programming, such as sequential quadratic programming \cite{NoceWrig06,PDEcon_Hinze}, may be employed. Indeed such an approach has also been for long used in order to derive limiting optimality conditions (under vanishing regularization) which unfortunately typically leads to stationarity systems containing less information when compared to the strong stationarity conditions as in \cite{barbu1984optimal, christof, mignot_puel}, obtained by using nonsmooth analysis techniques. In this work, we  show that in the case of ReLU learning-informed PDEs, additional issues can arise from a smoothing approach. In particular, due to a potentially large architecture of a network $\mathcal{N}$ (large number of layers and neurons), a natural and efficient way to smoothen $\mathcal{N}$  (after its training has been completed) would be via simply smoothing the ReLU function in $\mathcal{N}$, denoted now by $\sigma_{\epsilon}$, resulting in a smooth network $\mathcal{N}_{\epsilon}$ approximating $\mathcal{N}$. We refer to this technique as \emph{canonical smoothing} of $\mathcal{N}$. However we show with simple examples that this type of smoothing does not necessarily preserve monotonicity for deep enough networks, and in fact it does not even preserve it in a way  that monotonicity of the PDE operator could still be shown. This possibly renders the resulting control-to-state map $S_{\mathcal{N}_{\epsilon}}$ multi-valued, posing difficulties when resorting to classical algorithms for the solution of the regularized problem.
This is yet another motivation for us to devise numerical methods which are capable of directly solving \eqref{P_intro}. In this vein, we propose to adapt the bundle-free method from \cite{Hint_Suro}, originally developed for a class of mathematical programs with equilibrium constraints (MPECs). The proposed algorithm makes use of an auxiliary optimization problem as in \cite{Hint_Suro}, and we show that  by approximating the derivatives of the ReLU network (but not the ReLU itself!) via a smoothed $\max$-function, then a  descent direction for a reduced version of \eqref{P_intro} at a given control iterate is identified or (ideally) $B$-stationarity of that iterate can be diagnosed. 
We also mention that in \cite{christof}, an algorithm for solving a very specific nonsmooth semilinear PDE (in a first-discretize-then-optimize flavor) in the absence of control constraints has been proposed, where $\mathcal{N}(\cdot, y)=\max(0,y)$. However, as it was also noted by the authors of  \cite{christof} their algorithm cannot be applied to general nonsmooth semilinear PDEs, and an efficient algorithm for the general case calls for new ideas. The current paper aims to cover this gap.

\subsection{Structure of the paper} 
In Section \ref{sec:deep_learning} we focus on the structure of the functional form of ReLU networks. We are in particular interested in understanding how this structure changes after smoothing the ReLU network function via regularization of the associated activation function (canonical smoothing). Here our main focus is on how this kind of smoothing can break the monotonicity of the network. 
The implication of the latter concerning the emergence of nonuniqueness of solutions of the learning-informed state equation is discussed in Section \ref{sec:ReLU_PDE}. We also collect basic results concerning the general optimal control problem \eqref{P_intro} and in particular we recall  the  stationarity conditions derived in \cite{DonHinPapVoe22}. In Section \ref{sec:algorithm}, we introduce and analyze a descent algorithm that directly treats the nonsmooth optimal control problem. It is applied in Section \ref{sec:numerics} to several instances of an optimal control problem with a ReLU network-informed semilinear second-order elliptic PDE. In particular, also a nonmonotone setting is considered in order to challenge the solver.

\section{Smoothings of ReLU neural networks}\label{sec:deep_learning}

\subsection{Definition and basic properties}
We first fix some notation.
 For a set $A$, the characteristic and the indicator functions $\mathbbm{1}_{A}$ and $\mathcal{X}_{A}$, respectively, are defined as $\mathbbm{1}_{A}(x)=1$ if $x\in A$ and $\mathbbm{1}_{A}(x)=0$ otherwise, and $\mathcal{X}_{A}(x)=0$ if $x\in A$ and $\mathcal{X}_{A}(x)=+\infty$ otherwise. Unless otherwise stated $\langle \cdot, \cdot \rangle$ denotes the standard $L^{2}$ inner product.

\begin{definition}[Standard feedforward multilayer neural network]\label{def:neural_networks}
	Let $L\in \mathbb{N}$, network parameters $\theta=\left ((W_{1}, b_{1}), \ldots, (W_{L},b_{L} ) \right )$ with $W_{i}\in \mathbb{R}^{n_{i}\times n_{i-1}}$, $b_{i}\in \mathbb{R}^{n_{i}}$, for $i=1,\ldots, L$ and $n_{i}\in \mathbb{N}$ for $i=0,\ldots, L$. Furthermore let $\sigma:\mathbb{R}\to \mathbb{R}$ be an arbitrary function. We call a function $\mathcal{N}: \mathbb{R}^{n_{0}}\to \mathbb{R}^{n_{L}}$ a neural network with weight matrices $(W_{i})_{i=1}^{L}$, bias vectors $(b_{i})_{i=1}^{L}$ (the network parameters) and activation function $\sigma$ if $\mathcal{N}(x)$ can be defined through the following recursive relation for any $x\in \mathbb{R}^{n_{0}}$:
	\begin{align}
	z_{0}&= x,\label{defNN1}\\
	z_{\ell}&=\sigma \left (W_{\ell} z_{\ell-1} + b_{\ell} \right ), \quad \ell=1,\ldots,  L-1, \label{defNN2}\\
	\mathcal{N}(x)&= W_{L} z_{L-1} + b_{L}. \label{defNN3}
	\end{align}
	The action of the activation function $\sigma$ in \eqref{defNN2} is considered componentwise i.e.\ for a vector $y=(y^{1}, \ldots, y^{n})\in \mathbb{R}^{n}$ we set $\sigma(y):=(\sigma(y^{1}), \ldots, \sigma(y^{n}))$. More compactly, $\mathcal{N}$ can also be defined as 
	\begin{equation}\label{defNN1_comp}
	\mathcal{N}(x)= T_{L} \circ \sigma(T_{L-1}) \circ \cdots \circ \sigma(T_{2}) \circ \sigma (T_{1}(x)), \quad x\in \RR^{n_{0}},
	\end{equation}
	where for every $\ell=1,\ldots, L$, $T_{L}$ denotes the affine transformation $z\mapsto W_{\ell} z + b_{\ell}$.\\[0.5em]
	We call $\mathcal{N}$ a ReLU neural network if $\sigma$ is the ReLU (Rectified Linear Unit) activation function:
	\begin{equation}
	\sigma(t)=\max(t,0), \quad t\in \mathbb{R}.
	\end{equation}
	
\end{definition}

Following the standard neural network terminology, we say that a neural network defined as in \eqref{defNN1}--\eqref{defNN3}, has \emph{$L$ layers} and \emph{$L-1$ hidden layers}, with the latter denoting the operations in \eqref{defNN2}. The final operation \eqref{defNN3} is called the \emph{output layer}. Furthermore, $n_{i}$ is  the number of \emph{neurons} in the $i$-th layer, $i=1, \ldots, L$, that is, it is the number of rows of the weight matrix $W_{i}$. The number of neurons of a given layer is also called the \emph{width} of that layer, while the number of layers is called the \emph{depth} of the network.

We should note that a neural network as a function, does not necessarily admit a unique representation with respect to the weight matrices, the bias vectors and the activation functions.  Furthermore in the Definition \ref{def:neural_networks}, the input of the $\ell$-th layer consists only of the output $z_{\ell-1}$ of the previous layer. A more general neural network definition would allow the input for each layer to depend on the output of all the previous layers. In that case every $W_{\ell}$ would be a weight matrix of size $\mathbb{R}^{n_{i}\times (\sum_{k=0}^{\ell-1} n_{k})}$. However, since every network of the latter type can be realized by a network as in \eqref{def:neural_networks}, see  \cite{relu_Wsp}, we will stick to the more classical definition given above.

We are interested in the regularity of the functions that are realized by ReLU neural networks. It turns out that the latter class coincides with the class of \emph{continuous piecewise affine functions}.

\begin{definition}[Continuous piecewise affine functions]\label{def:cpwl}
	Let $n_{0}\in \mathbb{N}$. We say that a function $\mathcal{F}: \mathbb{R}^{n_{0}}\to \mathbb{R}$ is continuous piecewise affine (CPWA)  if 
	the following condition holds:
	\begin{itemize}
		\item $\mathcal{F}$ is continuous and there exist finitely many affine maps $f_{1}, \ldots, f_{p}: \mathbb{R}^{n_{0}}\to \mathbb{R}$ for some $p\in\mathbb{N}$ such that for every $x\in \mathbb{R}^{n_{0}}$, there exists an $i\in \{1,\ldots, p\}$ such that $\mathcal{F}(x)=f_{i}(x)$.
	\end{itemize}
\end{definition}

We refer to \cite{aliprantis_harris_tourky_2006, arora2018understanding,hinging_hyperplanes} for further equivalent characterizations of CPWA  functions.

\begin{theorem}[Characterization of ReLU neural networks, \cite{arora2018understanding}]
	A  function $\mathcal{N}:\mathbb{R}^{n_{0}}\to \mathbb{R}$ is a ReLU neural network if and only if it is a CPWA function.
\end{theorem}

From the definition \eqref{defNN1}--\eqref{defNN3} it is clear that $\mathcal{N}:\mathbb{R}^{n_{0}} \to \mathbb{R}^{n_{L}}$, $n_{L}\ge 1$, is an $\mathbb{R}^{n_{L}}$-valued ReLU neural network if and only if $\mathcal{N}=(\mathcal{N}_{1}, \ldots, \mathcal{N}_{L})$ with each $\mathcal{N}_{i}: \mathbb{R}^{n_{0}}\to \mathbb{R}$, $i=1, \ldots, L$ being a scalar-valued  ReLU neural network. Thus $\mathcal{N}$ is an $\mathbb{R}^{n_{L}}$-valued ReLU neural network if and only if it is an $\mathbb{R}^{n_{L}}$-valued CPWA  function, with the latter defined exactly as in Definition \ref{def:cpwl} with the only difference being that the affine maps $f_{i}$ are $\mathbb{R}^{n_{L}}$-valued.

To give an example, for $p\ge 2$ and $t_{1} \le \cdots \le t_{p-1}$, we consider the following one dimensional continuous piecewise affine  function $\mathcal{F}$ with 
\begin{equation}\label{cpwl_1d}
\mathcal{F}(t) =
\begin{cases}
a_{1} t + \gamma_{1} &\text{ if } t\le t_{1}, \\
a_{i} t + \gamma_{i} &\text{ if } t_{i-1} \le t \le  t_{i}, \quad  i=2, \ldots p-1,\\
a_{p} t + \gamma_{p} & \text{ if }  t \ge t_{p-1}.
\end{cases}
\end{equation}
Note that we assume that $(a_{i},\gamma_{i})_{i=1}^{p}$ satisfy the appropriate conditions such that $\mathcal{F}$ is continuous.
Then it can be checked, see for instance \cite[Corollary 3.5]{aliprantis_harris_tourky_2006}, that  $\mathcal{F}$ can be written as 
\begin{align}\label{cpwl_1d_relu}
\mathcal{F}(t)
&=a_{1}t + \gamma_{1} + \sum_{i=1}^{p-1} (a_{i+1}-a_{i}) \mathrm{max}(t-t_{i}, 0)\nonumber\\
&=a_{1}(\mathrm{max}(t,0)  -a_{1}\mathrm{max}(-t,0) +  \sum_{i=1}^{p-1} (a_{i+1}-a_{i}) \mathrm{max}(t-t_{i}, 0) + \gamma_{1}.
\end{align}
This means that $\mathcal{F}$ can be realized as a ReLU neural network with one hidden layer having $p+1$ neurons. In particular,  $\mathcal{F}= T_{2}\circ \sigma (T_{1})$, where $T_{1}(t)= W_{1}t + b_{1}$, $T_{2}(z)=W_{2}z +b_{2}$ with $W_{1}=(1,-1, 1, 1, \ldots, 1)^{T}\in \mathbb{R}^{(p+1)\times 1}$, $b_{1}=(0,0,-t_{1}\ldots, -t_{p-1})^{T}\in \mathbb{R}^{(p+1)\times 1}$, and $W_{2}=(a_{1}, -a_{1},a_{2}-a_{1}, \ldots, a_{p}-a_{p-1})\in \mathbb{R}^{1\times (p+1)}$, $b_{2}=\gamma_{1}\in \mathbb{R}$.

Another characteristic of ReLU neural networks are their approximation capabilities. In fact it can be easily checked that given a bounded domain $U\subset \mathbb{R}^{n_{0}}$ with Lipschitz boundary we have that for every $\epsilon>0$ and $f\in W^{1,\infty}(U)$ there exists a ReLU neural network $\mathcal{N}_{\epsilon}:\RR^{n_{0}} \to \RR$ such that $\|\mathcal{N}_{\epsilon}-f\|_{W^{1,\infty}(U)}<\epsilon$, see also \cite[Section 2.2]{DonHinPapVoe22}.
\vspace{1em}\noindent

\subsection{Smoothings of R{e}LU neural networks}\label{sec:smoothingReLU}
We are also interested in smoothing versions of ReLU networks. One canonical way to achieve smoothing is via appropriately smoothing the ReLU function $\sigma$ which is the constituent of the network that determines its regularity. In optimal control, typically specific approximating sequences are used \cite{christof, mignot_puel, Schiela} which we will also employ here.

\begin{definition}[Canonical smoothing of ReLU]
	We say that the family $\sigma_{\epsilon}: \mathbb{R}\to \mathbb{R}$ (or ReLU$_{\epsilon}$), $\epsilon>0$, is a canonical smoothing of the ReLU function if:
	\begin{enumerate}
		\item $\sigma_{\epsilon}$ is a positive, convex, monotonically increasing $C^{1}(\mathbb{R})$ function for all $\epsilon>0$.
		\item $\sigma_{\epsilon} \to \sigma $ uniformly and monotonically as $\epsilon \to 0$, i.e.,
			\[|\sigma_{\epsilon_{1}}(x)- \sigma(x)|\le |\sigma_{\epsilon_{2}}(x)- \sigma(x)|  \; \text{ for } \; 0< \epsilon_{1} \le \epsilon_{2}\;\text{ and for every } \; x\in \mathbb{R}.\]
	\end{enumerate}
	We say that a family of  networks $\mathcal{N}_{\epsilon}:\mathbb{R}^{n_{0}} \to \mathbb{R}^{n_{L}}  $, $\epsilon>0$, is a canonical smoothing of the ReLU network $\mathcal{N}:\mathbb{R}^{n_{0}} \to \mathbb{R}^{n_{L}} $ if it results from $\mathcal{N}$ by simply substituting the activation function $\sigma$ by $\sigma_{\epsilon}$.
\end{definition}

\begin{lemma}\label{lemma_can_smooth}
	Let $(\sigma_{\epsilon})_{\epsilon>0}$, be a canonical smoothing of the ReLU function. Then   the following two additional properties hold for $\epsilon>0$ small enough:
	\begin{enumerate}
		\item $0\le \sigma_{\epsilon}'(t)\le 1$, for all $t\in \mathbb{R}$.
		\item For every $\delta>0$, $\sigma_{\epsilon}'$ converges uniformly to $1$ on $[\delta, \infty)$ and uniformly to $0$ on $(-\infty, \delta]$ as $\epsilon\to 0$.
	\end{enumerate}
\end{lemma}

\begin{proof}
	Suppose that $(i)$ does not hold. Then because every $\sigma_{\epsilon}$ is convex and hence $\sigma_{\epsilon}'$ is incrceasing, there exists $\epsilon_{n}\to 0$ and $t_{n}\in \mathbb{R}$ such that $\sigma_{\epsilon_{n}}'(t)>1$ for every $t\in [t_{n}, \infty)$. But that means that for every $n\in\mathbb{N}$ there exists $t\in  [t_{n}, \infty)$ such that $\sigma_{\epsilon_{n}}(t)$ is arbitrary far away from $\sigma(t)$ contradicting the uniform convergence.
	
	For $(ii)$, we fix $\delta>0$. Using a similar argument as before we deduce that for every $\epsilon>0$ small enough it holds that $\lim_{t\to\infty} \sigma_{\epsilon}'(t)=1$. Hence, given the monotonicity of $\sigma_{\epsilon}'$ and $(i)$, it suffices to show that $\lim_{\epsilon\to 0}\sigma_{\epsilon}'(\delta)=1$. But if this is not the case it can easily be checked that there exists a subsequence $\sigma_{\epsilon_{n}}$ and $\eta>0$ such that $|\sigma_{\epsilon_{n}}(0)|>\eta$ contradicting the convergence $\lim_{\epsilon\to 0} \sigma_{\epsilon}(0)=0$. The uniform convergence of $\sigma_{\epsilon}'$ to $0$ on $(-\infty,\delta)$ is proved similarly.
\end{proof}

There are numerous options for a canonical smoothing of the ReLU function, see for instance Figure \ref{fig:canonical_smoothing}.
It is also clear that $\mathcal{N}_{\epsilon}\to \mathcal{N}$ uniformly but as we will show later with a counterexample the convergence does not have to be necessarily monotonic.

\begin{figure}[h!]
	%\hspace{-1.5em}
	\begin{minipage}[t]{0.4\textwidth}
		\centering
		\includegraphics[width=1.1\textwidth]{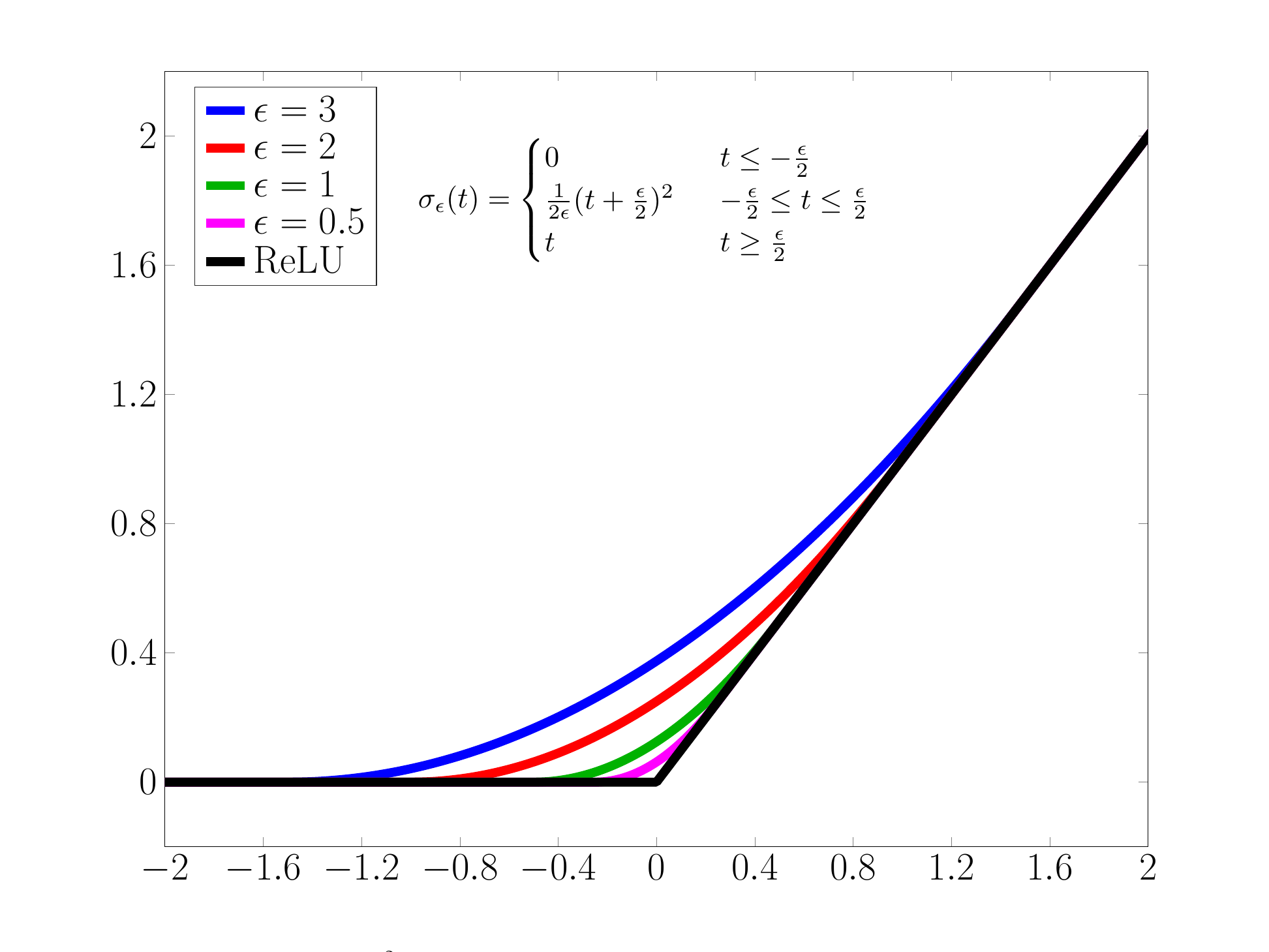}
	\end{minipage}
	\begin{minipage}[t]{0.4\textwidth}
		\centering
		\includegraphics[width=1.1\textwidth]{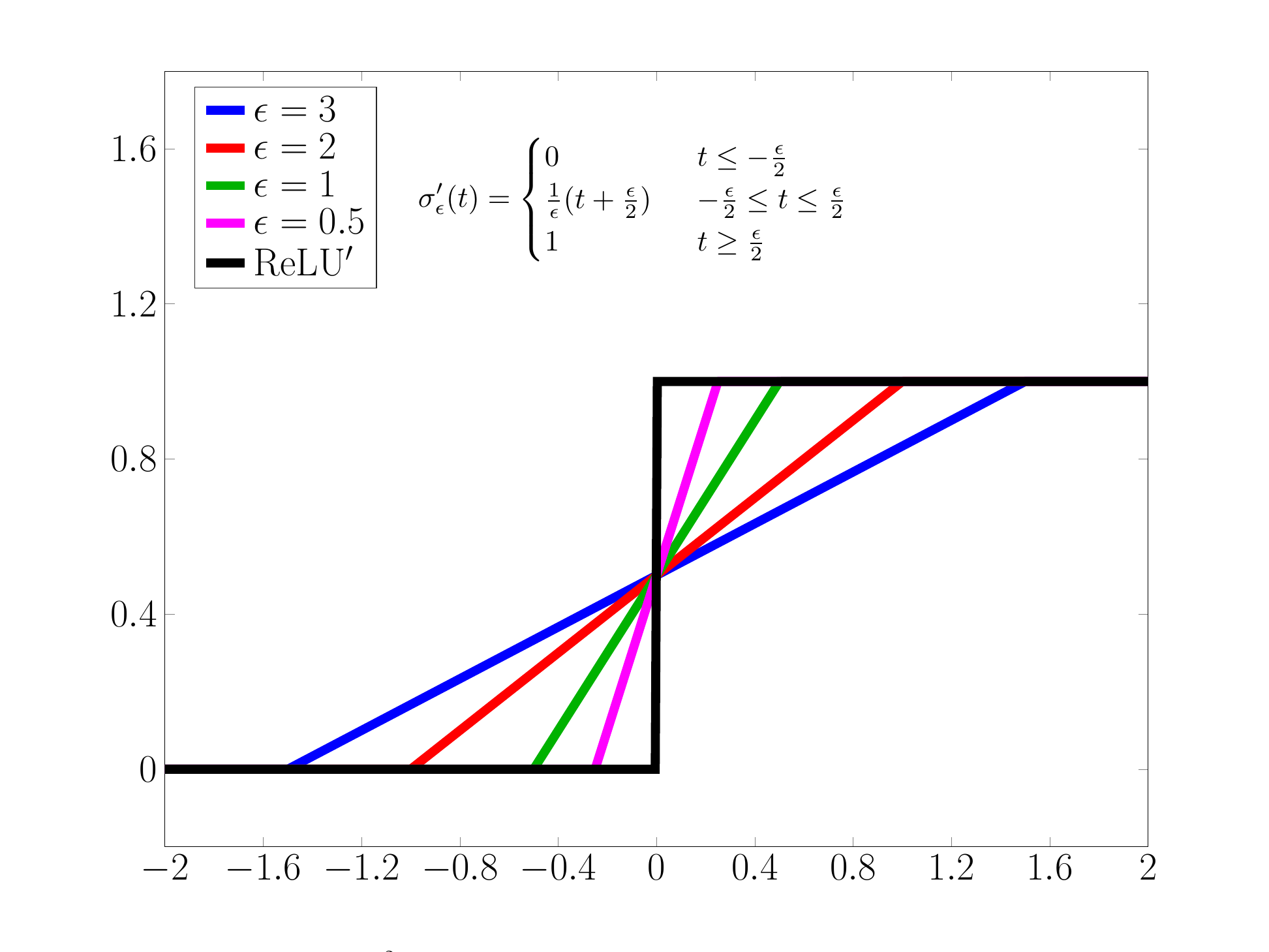}
	\end{minipage}
	%\hspace{1em}
	
	\begin{minipage}[t]{0.4\textwidth}
		\centering
		\includegraphics[width=1.1\textwidth]{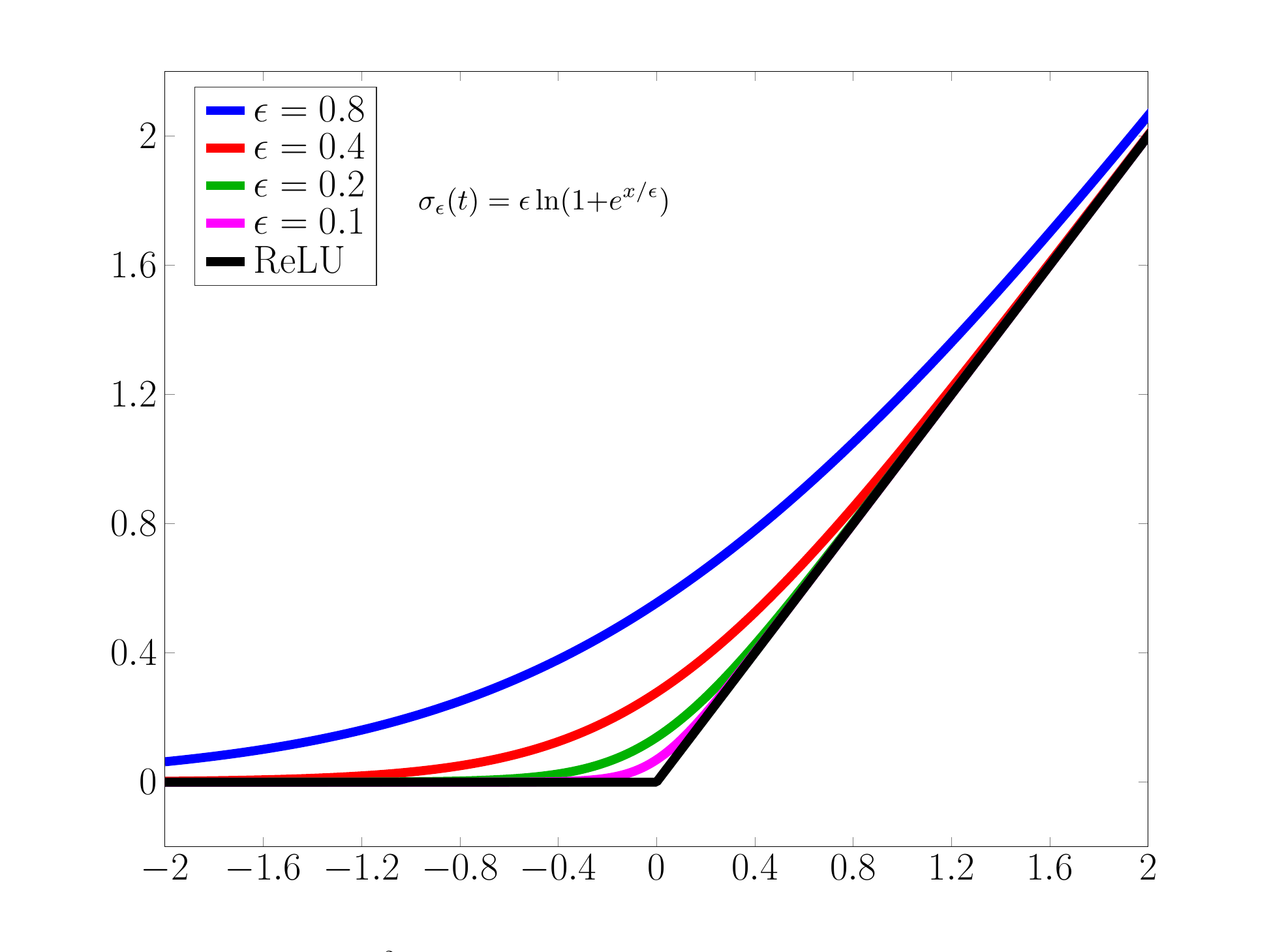}
	\end{minipage}
	\begin{minipage}[t]{0.4\textwidth}
		\centering
		\includegraphics[width=1.1\textwidth]{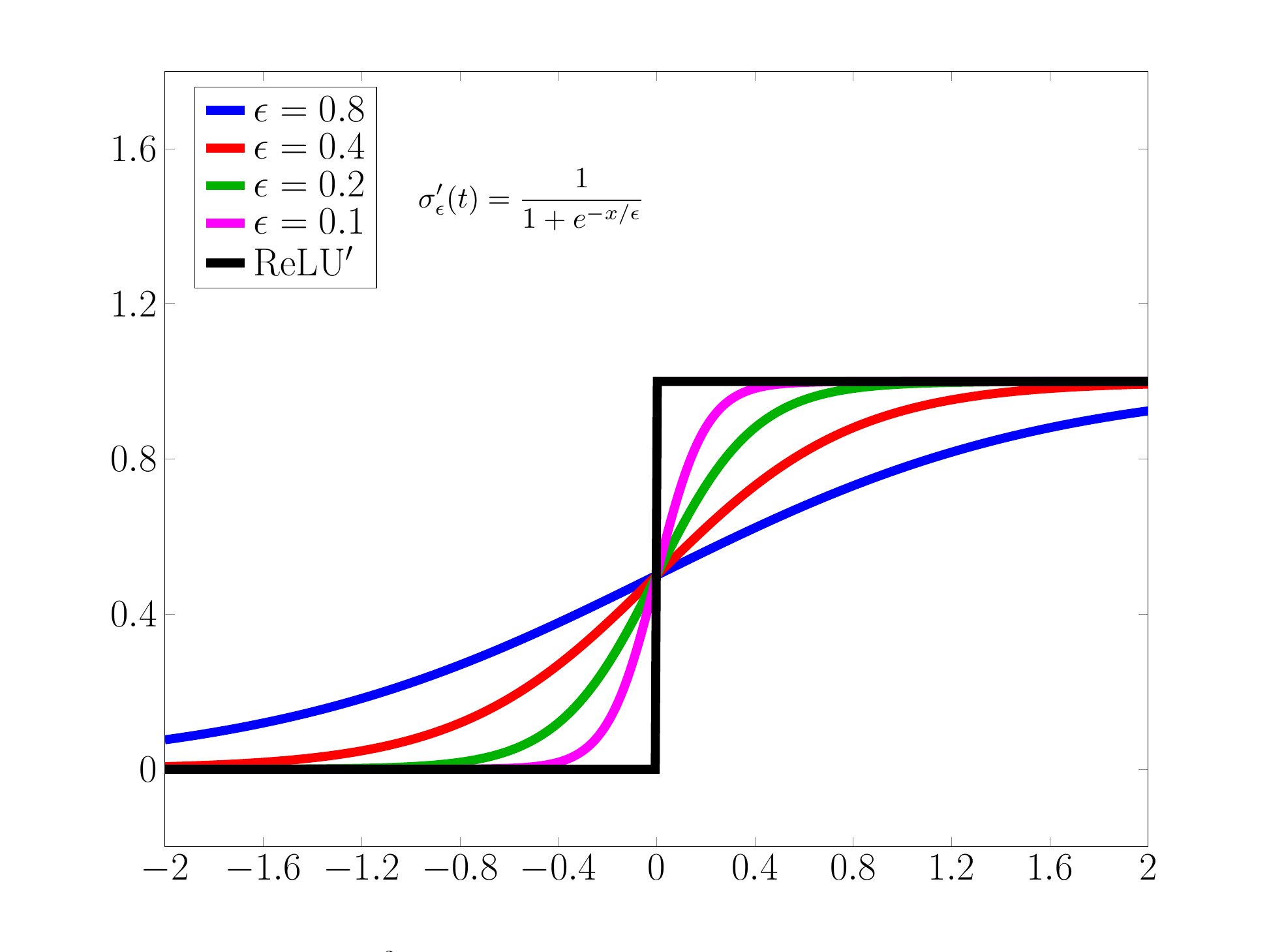}
	\end{minipage}		
	\caption{Examples of canonical smoothings of the ReLU functions together with depiction of the corresponding derivative approximations. In particular, the second one is the so-called \emph{Softplus} function, whose derivative is the logistic function - both extensively used in machine learning.}
	\label{fig:canonical_smoothing}
\end{figure}

Nevertheless the following holds:

\begin{proposition}\label{prop:approx_smoothing}
	Let $\mathcal{N}, (\mathcal{N}_{\epsilon})_{\epsilon>0}: \mathbb{R}^{n_{0}} \to \mathbb{R}^{n_{L}} $ be a ReLU network and a canonical smoothing of it. Then it holds:
	\begin{equation}\label{cs_uniform_values}
	\|\mathcal{N}_{\epsilon} - \mathcal{N}\|_{\infty} \le M \|\sigma_{\epsilon} - \sigma\|_{\infty},
	\end{equation}
	where the constant $M$ does not depend on $\epsilon$ but only on the parameters of $\mathcal{N}$. In particular,  $\mathcal{N}_{\epsilon}\to \mathcal{N}$ uniformly as $\epsilon \to 0$.
	
	Furthermore, for every $1\le p<\infty$ and for every open bounded $U\subset \mathbb{R}^{n_{0}}$ we have that
	\begin{equation}\label{cs_Lp_derivatives}
	\|\nabla \mathcal{N}_{\epsilon}-\nabla \mathcal{N}\|_{L^{p}(U)} \to 0, \quad \text{ as } \epsilon\to 0.
	\end{equation}
\end{proposition}

\begin{proof}
	In order to show \eqref{cs_uniform_values} we will show the result for networks with two hidden layers and then one can proceed via induction.
	Let $\mathcal{N}=T_{3}\Big ( \sigma \big ( T_{2} (\sigma (T_{1}))  \big) \Big )$ and $\mathcal{N}_{\epsilon}=T_{3}\Big ( \sigma_{\epsilon} \big ( T_{2} (\sigma_\epsilon (T_{1}))  \big) \Big )$ be a two hidden layer ReLU network and its corresponding canonical smoothing, in accordance to the formulation \eqref{defNN1_comp} (without loss of generality let $T_{3}$ be linear). Then, setting $N^{(2)}(x):=T_{2}(\sigma(T_{1}(x)))$ and  $N_{\epsilon}^{(2)}(x):=T_{2}(\sigma_{\epsilon}(T_{1}(x)))$, we estimate successively for  $x\in \mathbb{R}^{n_{0}}$ 
	\begin{align}
	|N_{\epsilon}^{(2)}(x)-N^{(2)}(x)|=
	|T_{2} (\sigma_{\epsilon} (T_{1}(x))) - T_{2} (\sigma (T_{1}(x)))|\le \|T_{2}\| |\sigma_{\epsilon} (T_{1}(x))-\sigma (T_{1}(x))|
	\le \|T_{2}\| \|\sigma_{\epsilon}-\sigma\|_{\infty}.\label{ind1}
	\end{align}
	We then further estimate
	\begin{align*}
	|\mathcal{N}_{\epsilon}(x)-\mathcal{N}(x)|
	&=\|T_{3}\| \left |(\sigma_{\epsilon}(N_{\epsilon}^{(2)}(x))- \sigma(N^{(2)}(x))) \right |\\
	&\le \|T_{3}\|\left ( \left | \sigma_{\epsilon} (N_{\epsilon}^{(2)}(x)) -\sigma_{\epsilon} (N^{(2)}(x)) \right | 
	+ \left| \sigma_{\epsilon} (N^{(2)}(x)) - \sigma (N^{(2)}(x)) \right | \right )\\
	&\le \|T_{3}\| (\|T_{2}\| \|\sigma_{\epsilon}-\sigma\|_{\infty} + \|\sigma_{\epsilon}-\sigma\|_{\infty})\\
	&\le M  \|\sigma_{\epsilon}-\sigma\|_{\infty},
	\end{align*}
	where we employed the mean value theorem for $\sigma_{\epsilon}$, using the fact that $0\le \sigma_{\epsilon}'\le 1$. The induction step follows similarly.
	
	For \eqref{cs_Lp_derivatives}, notice first that  $\mathcal{N}$ restricted to $U$ (a function that we still denote by $\mathcal{N}$) belongs to $W^{1,\infty}(U)$, being Lipschitz. In particular $\nabla \mathcal{N}:U\to \mathbb{R}^{n_{L}\times n_{0}}$ is a function in $L^{\infty}(U)$ and -- see \cite[Theorem III.1]{relu_chainrule} -- for almost every $x$ it is  equal to 
	\begin{equation}\label{relu_prime}
	\nabla \mathcal{N}(x)=W_{L}  \cdot \sigma'( N^{(L-1)}(x)) \cdot W_{L-1} \cdot \ldots \cdot \sigma'(N^{(1)}(x))\cdot W_{1},
	\end{equation}
	with $N^{(K)}$ defined as above and $\sigma':=\mathbbm{1}_{(0,\infty)}$ being applied pointwise. 
	Note that, while $\sigma'(N^{i}(x))$  is an $\RR^{n_{i}}$-vector, in \eqref{relu_prime}  using the same notation we denote the $n_{i} \times n_{i}$ diagonal matrix with the same vector in the diagonal.
	Analogously $\nabla \mathcal{N}_{\epsilon}\in C(\overline{U})$ where for every $x$
	\begin{equation}\label{relu_eps_prime}
	\nabla \mathcal{N}_{\epsilon}(x)=W_{L}  \cdot \sigma_{\epsilon}'( N_{\epsilon}^{(L-1)}(x)) \cdot W_{L-1} \cdot \ldots \cdot \sigma_{\epsilon}'(N_{\epsilon}^{(1)}(x))\cdot W_{1}.
	\end{equation}
	We check that $\nabla \mathcal{N}_{\epsilon}\to \nabla \mathcal{N}$ almost everywhere, and then \eqref{cs_Lp_derivatives} follows by employing the dominated convergence theorem using the fact that $\nabla \mathcal{N}_{\epsilon}$ is uniformly bounded in $L^{\infty}(U)$ since $0\le \sigma_{\epsilon}'\le 1$ for every $\epsilon>0$. In order to show the almost everywhere pointwise convergence of the gradients, in view of the recursive formulas \eqref{relu_prime} and \eqref{relu_eps_prime}, and considering an inductive argument it suffices to show that if $N:=(N_{1}, \ldots, N_{n}):\mathbb{R}^{m}\to \mathbb{R}^{n}$ is a ReLU network, $N_{\epsilon}$ is a canonical smoothing such that $\nabla N_{\epsilon}\to \nabla N$ almost everywhere as $\epsilon \to 0$ then also
	\begin{equation}\label{deriv_induct}
	\sigma_{\epsilon}' (N_{\epsilon}) \nabla N_{\epsilon} \to \sigma'(N) \nabla N, \quad \text{a.e. as } \epsilon \to 0.
	\end{equation}
	Let $\tilde{U}$ be the set of full measure where $\nabla N_{\epsilon}\to \nabla N$ converges pointwise.  Fixing an $1\le i \le n$, as a first case, let $x\in \tilde{U}$ be such that $N_{i}(x)\ne 0$. Then using $N_{i,\epsilon} (x) \to N_{i}(x)$ and Lemma \ref{lemma_can_smooth} $(ii)$, we get $\sigma'_{\epsilon}(N_{i,\epsilon}(x))\to \sigma'(N_{i}(x))$ and hence \eqref{deriv_induct} holds for that $x$ and the $i$-th row. Let now $x\in \tilde{U}$ such that $N_{i}(x)=0$. Since $N_{i}$ is Lipschitz then, see e.g. \cite[Theorem 3.3(i)]{EvansGariepy}, the set of such $x$ such that $\nabla_{i}N(x)\ne 0$ has a zero Lebesgue measure, so we can assume that $\nabla_{i} N(x)=0$. Then \eqref{deriv_induct} for the $i$-th row follows from the fact that $\nabla_{i} N_{\epsilon}(x)\to \nabla_{i} N(x)=0$ and the fact that $0\le \sigma'_{\epsilon} \le 1$.
\end{proof}

\begin{remark}\label{sigma_sigmaeps_le_eps}
	Looking at the examples of Figure \ref{fig:canonical_smoothing} one can see that $\sigma_{\epsilon}$ can be actually chosen such that
	\begin{equation}\label{sigma_uniform_values_stronger}
	\|\sigma_{\epsilon}-\sigma\|_{\infty}\le c\epsilon, \quad \text{for every }\epsilon>0,
	\end{equation}
	for some constant $c>0$. Then \eqref{cs_uniform_values} could be written in a stronger form as
	\begin{equation}\label{cs_uniform_values_stronger}
	\|\mathcal{N}_{\epsilon}-\mathcal{N}\|_{\infty}\le M\epsilon, \quad \text{for every }\epsilon>0.
	\end{equation}
	
\end{remark}

Since our focus here is on ReLU learning-informed PDEs, that is, PDEs that contain a ReLU neural network, we are particularly interested in monotonically increasing networks. As we will see later, they will guarantee uniqueness for the corresponding PDE. In particular we are interested in whether the monotonicity of the ReLU networks can be preserved under canonical smoothing, or not. If the latter is the case, then we study to which extent the resulting nonmonotone part can be controlled by the smoothing parameter $\epsilon>0$. In what follows we will always make a distinction between monotonically increasing and strictly monotonically increasing functions.
The following proposition sheds some light on this context.

\begin{proposition}\label{prop:relu_smoothing}
	The following are true:
	\begin{enumerate}
		\item There exists a canonical smoothing $\tilde{\sigma}_{\epsilon}$ of the ReLU function such that for every (strictly) monotonically increasing one-hidden layer ReLU network $\mathcal{N}:\mathbb{R}\to \mathbb{R}$, its corresponding canonical smoothing $\mathcal{N}_{\epsilon}$ under $\tilde{\sigma}_{\epsilon}$ is also (strictly) monotone for every $\epsilon>0$. However preservation of monotonicity of one-hidden layer ReLU networks does not necessarily hold for an arbitrary canonical smoothing.
		\item The property of the above canonical smoothing $\tilde{\sigma}_{\epsilon}$ does not hold for ReLU networks with more than one hidden layers. That is, there exists a monotone increasing two-hidden layer ReLU network such that its canonical smoothing $\mathcal{N}_{\epsilon}$ under $\tilde{\sigma}_{\epsilon}$ is not monotonically increasing for every $\epsilon>0$.
	\end{enumerate}
\end{proposition}
\begin{proof}
	For $(i)$ it suffices to define $\tilde{\sigma}_{\epsilon}=\rho_{\epsilon}\ast \sigma$, where $\rho_{\epsilon}(t)=\epsilon^{-1} \rho(t/\epsilon)$ and $\rho$ being the standard mollifier, 
	\[\rho(t)=ce^{-\frac{1}{1-t^{2}}}.\]
	Here $c>0$ is a constant such that $\int_{\mathbb{R}} \rho\,dx=1$. It is easy to check that $\sigma_{\epsilon}$ is a canonical smoothing. Let $\mathcal{N}$ be an one hidden layer ReLU network, that is
	\begin{align*}
	\mathcal{N}(t)
	&=b_{2} + W_{2} \sigma (W_{1}t + b_{1}) \\
	&=b_{2} + \sum_{i=1}^{n_{1}} w_{2}^{i} \sigma (w_{1}^{i} t + b_{1}^{i}), 
	\end{align*}
	where $W_{1}= (w_{1}^{1}, \ldots, w_{1}^{n_{1}})^{T}$, $W_{2}= (w_{2}^{1}, \ldots, w_{2}^{n_{1}})$, $b_{1}=(b_{1}, \ldots, b_{n_{1}})^{T}$ and $b_{2}\in \mathbb{R}$. Then from the linearity of convolution we have
	\begin{align*}
	\mathcal{N}_{\epsilon}(t)
	&:= b_{2} + \sum_{i=1}^{n_{1}} w_{2}^{i} \rho_{\epsilon}\ast \sigma (w_{1}^{i} t + b_{1}^{i})
	=  \rho_{\epsilon}\ast \left ( b_{2} + \sum_{i=1}^{n_{1}} w_{2}^{i} \sigma (w_{1}^{i} t + b_{1}^{i}) \right)=\rho_{\epsilon}\ast \mathcal{N}(t).
	\end{align*}
	Hence if $\mathcal{N}$ is (strictly) monotone then $\mathcal{N}_{\epsilon}$ is (strictly) monotone as well, since it is immediate to check that this convolution preserves (strict) monotonicity.
	
	In order to see that the above property does not hold for an arbitrary canonical smoothing, consider for instance the canonical smoothing of the first example of Figure \ref{fig:canonical_smoothing}. Let $\mathcal{N}$ be the ReLU neural network defined as
	\[\mathcal{N}(t)=\mathrm{max}(\lambda_{1}t,0) + \mathrm{max}(\lambda_{2}t,0) -\mathrm{max}(\lambda_{3}t,0) - \mathrm{max}(\lambda_{4}t,0), \]
	where $\lambda_{1}, \lambda_{2}, \lambda_{3}, \lambda_{4}>0$ and $\lambda_{1}+\lambda_{2}=\lambda_{3}+\lambda_{4}$ . This network has one hidden layer with 4 neurons and  obviously, $\mathcal{N}\equiv 0$ and is hence monotone. Given $\epsilon>0$, we have that for every $t\in [-\epsilon/2\lambda_{\max}, \epsilon/2\lambda_{\max}]$, with $\lambda_{\max}:=\mathrm{max}_{i} \lambda_{i}$ that
	\begin{equation}\label{sum_squares}
	\mathcal{N}_{\epsilon}'(t)=(\lambda_{1}^{2}+ \lambda_{2}^{2} - \lambda_{3}^{2}-\lambda_{4}^{2})\frac{t}{\epsilon}.
	\end{equation}
	Then by simply choosing $\lambda_{i}$ such that the specific linear combination of their squares in \eqref{sum_squares} is not zero, we get that the derivative of $\mathcal{N}_{\epsilon}$ changes sign in a small neighbourhood of the origin and thus implies nonmonotonicity.
	
	In order to produce a counterexample for $(ii)$ consider 
	\[\mathcal{M}(t)=\mathrm{max} (-\mathrm{max}(t,0),0)=\sigma(-\sigma(t)),\] 
	which is a two-hidden layer ReLU neural network realizing again the zero function, and let $\mathcal{M}_{\epsilon}(t)=\tilde{\sigma}_{\epsilon}(-\tilde{\sigma}_{\epsilon}(t))$ denote its canonical smoothing under $\tilde{\sigma}_{\epsilon}=\rho_{\epsilon}\ast \sigma$. Note that 
	\[
	( \rho_{\epsilon}\ast \sigma)(t)=
	\begin{cases}
	0 & \text{ if } t\le -\epsilon,\\
	\frac{1}{\epsilon} \int_{B(t,\epsilon)} \rho \left (\frac{t-s}{\epsilon} \right )\sigma (s)\,ds& \text{ if } -\epsilon<t<\epsilon,\\
	t & \text{ if } t\ge \epsilon.
	\end{cases}
	\]
	For the derivative of $\mathcal{M}_{\epsilon}$, it obviously holds that $\mathcal{M}_{\epsilon}'(t)=-\tilde{\sigma}_{\epsilon}'(-\tilde{\sigma}_{\epsilon}(t))\tilde{\sigma}_{\epsilon}'(t)\le 0$. Furthermore, for $t\ge \epsilon$, we have $\mathcal{M}_{\epsilon}=0$, while for $t\le -\epsilon$ we have $\mathcal{M}_{\epsilon}(t)=\sigma_{\epsilon}(0)>0$. Hence $\mathcal{M}_{\epsilon}$ is monotonically decreasing.
\end{proof}

\begin{figure}[h!]
	\hspace{-1.5em}
	\begin{minipage}[t]{0.32\textwidth}
		\centering
		\includegraphics[width=1.15\textwidth]{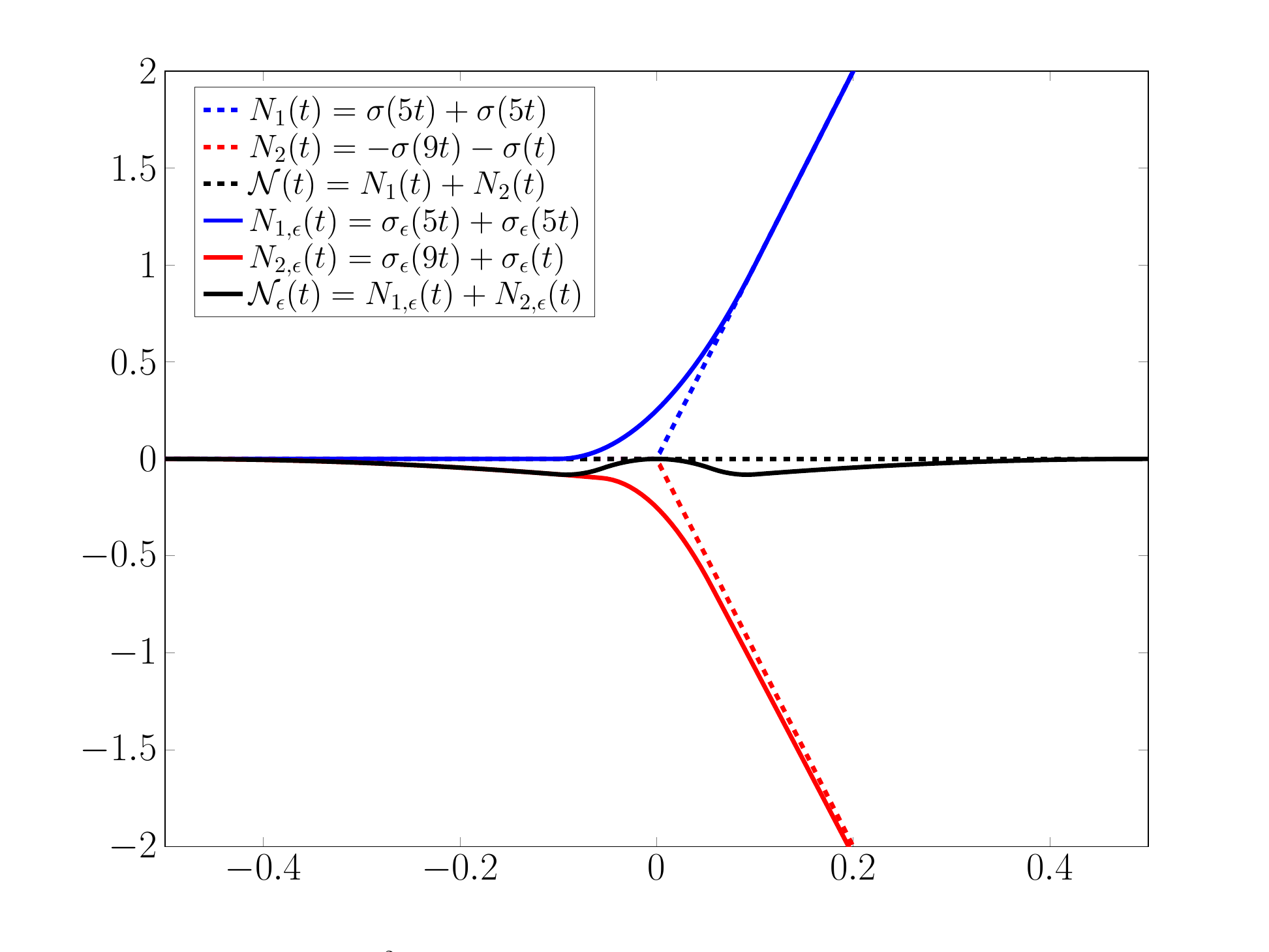}
	\end{minipage}
	\begin{minipage}[t]{0.32\textwidth}
		\centering
		\includegraphics[width=1.15\textwidth]{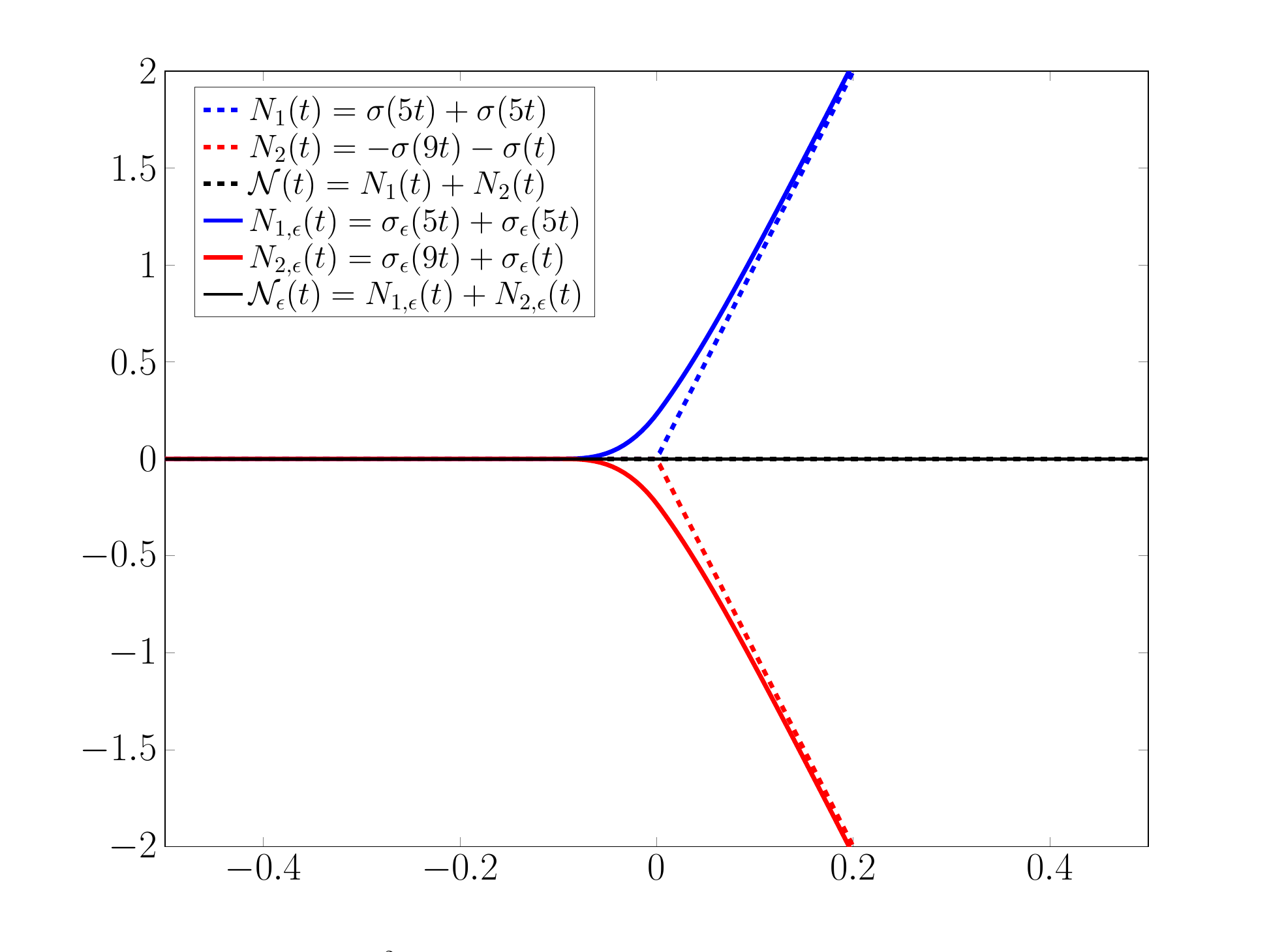}
	\end{minipage}
	
	\hspace{-3em}
	\begin{minipage}[t]{0.32\textwidth}
		\centering
		\includegraphics[width=1.15\textwidth]{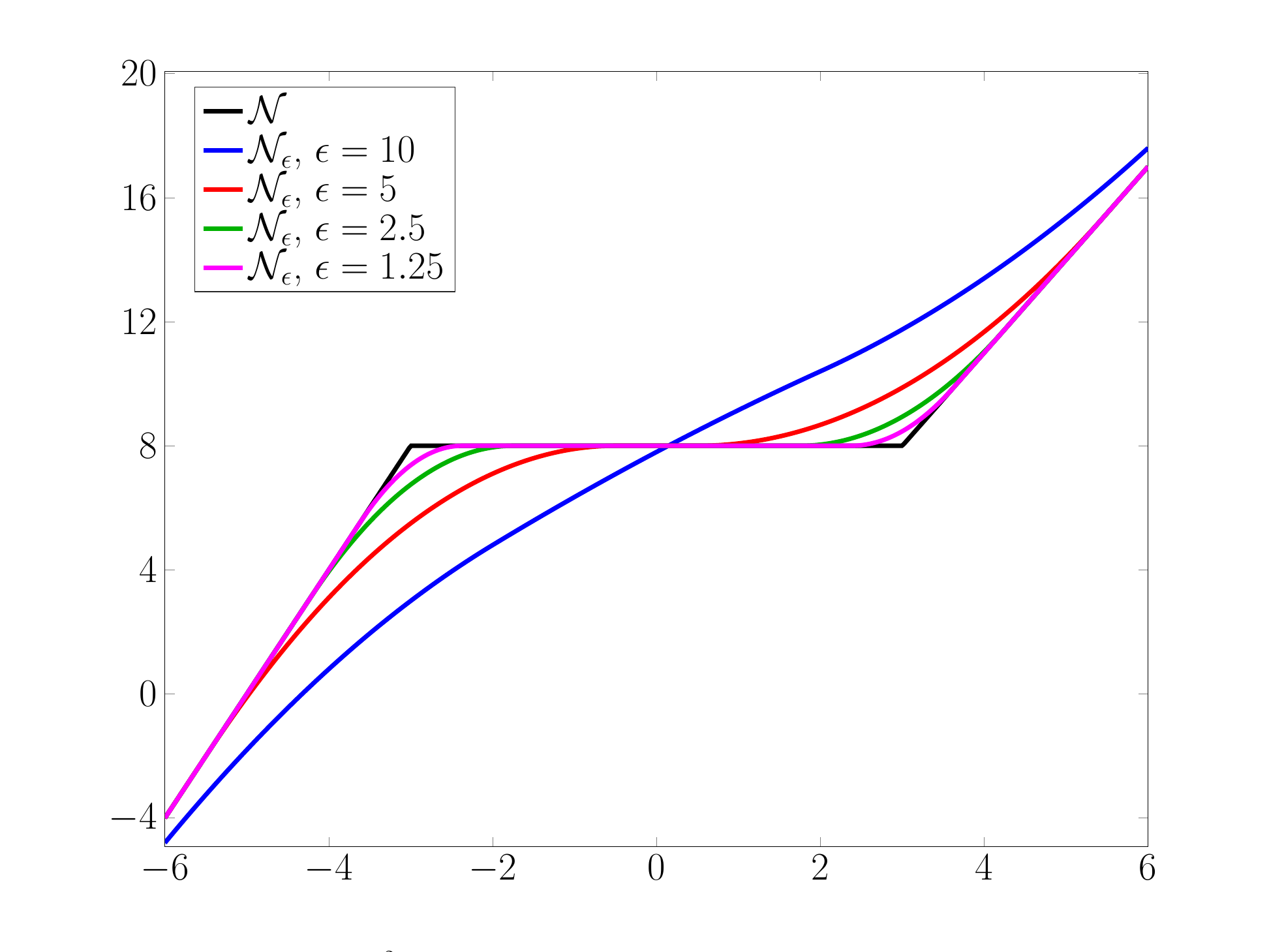}
	\end{minipage}
	\begin{minipage}[t]{0.32\textwidth}
		\centering
		\includegraphics[width=1.15\textwidth]{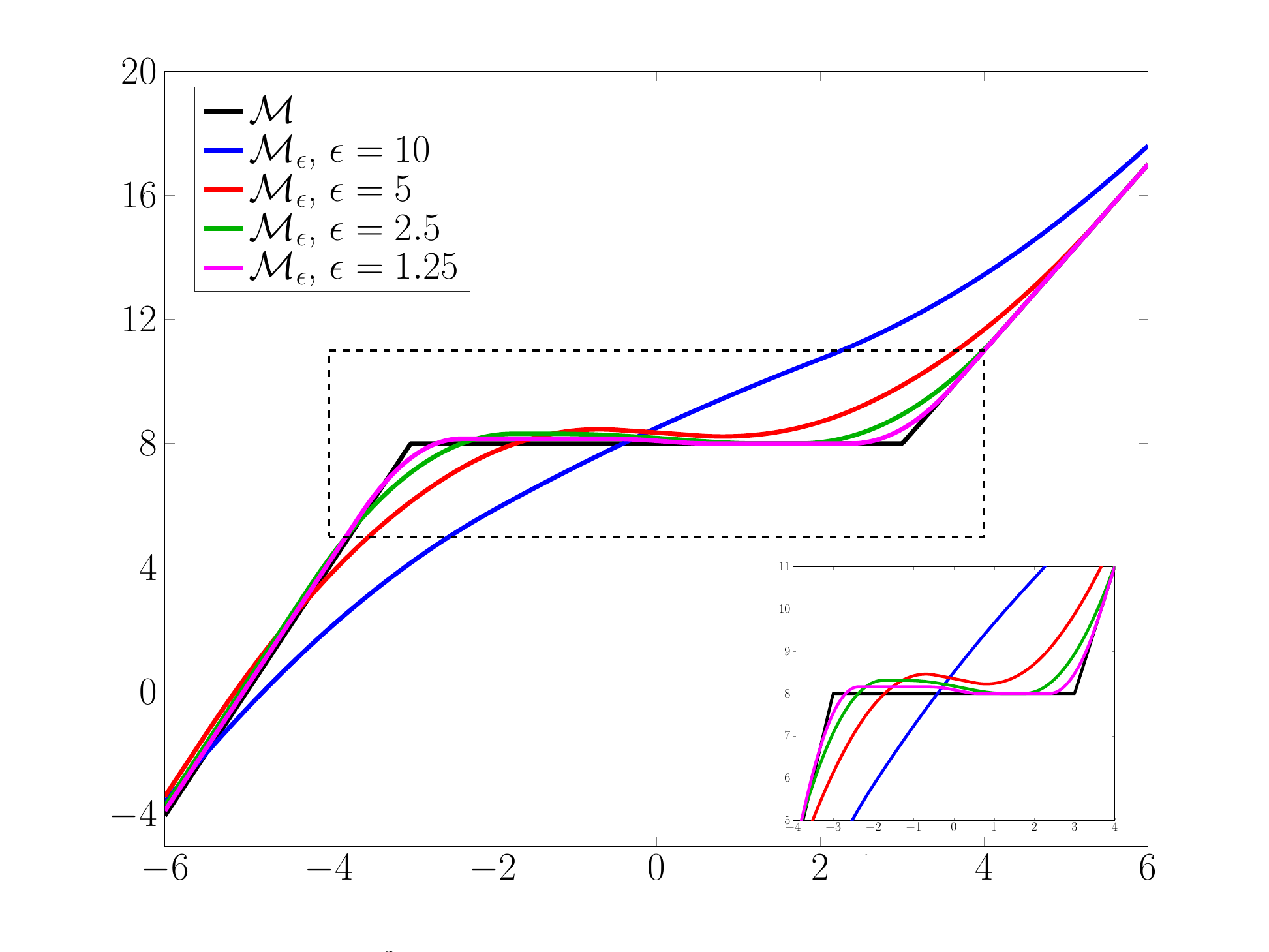}
	\end{minipage}
	\begin{minipage}[t]{0.32\textwidth}
		\centering
		\includegraphics[width=1.15\textwidth]{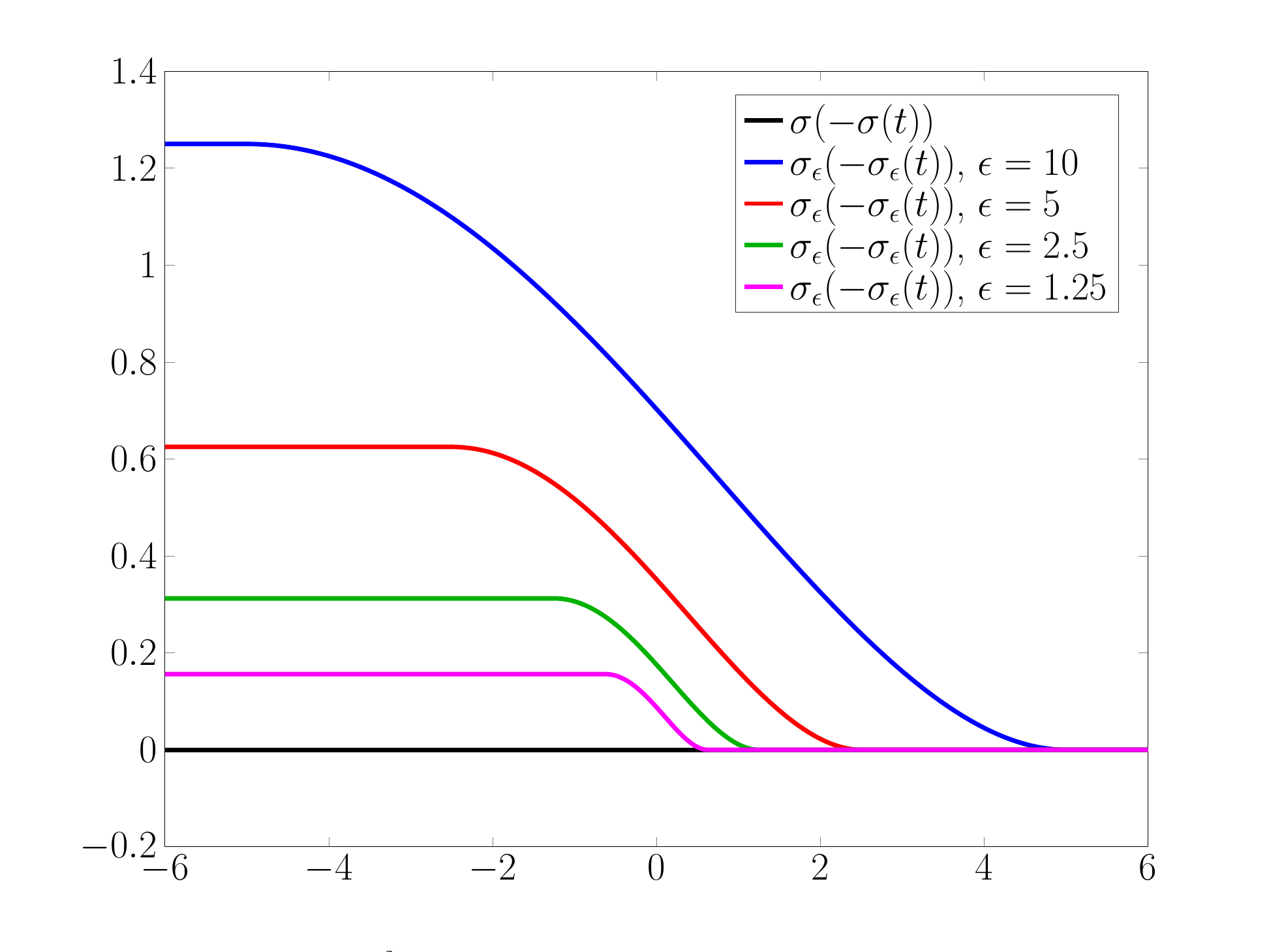}
	\end{minipage}
	\caption{Examples of non preservation of monotonicity of ReLU networks after canonical smoothing as these are dictated by Proposition \ref{prop:relu_smoothing}.}
	\label{fig:counterexamples_smoothing}
\end{figure}

In Figure \ref{fig:counterexamples_smoothing} we see a visualization of the examples given in Proposition \ref{prop:relu_smoothing}. At the top part of the figure,  we provide an example corresponding to $(i)$ of Proposition \ref{prop:relu_smoothing}. There, the zero function is written as a one hidden layer ReLU network, $\mathcal{N}(t)=\mathrm{max}(5t)+\mathrm{max}(5t)-\mathrm{max}(9t,0)-\mathrm{max}(t,0)$ (dashed black line). Using the canonical smoothing of the first example of Figure \ref{prop:relu_smoothing} the monotonicity is not preserved  (top left,  black solid line). On the other hand the linear smoothing $\tilde{\sigma}_{\epsilon}=\rho_{\epsilon}\ast \sigma$ preserves the monotonicity, see top right plot. At the  bottom part of Figure \ref{fig:counterexamples_smoothing} we have expressed the same monotone increasing CPWA function as a ReLU network of both one and two hidden layers, $\mathcal{N}$ and $\mathcal{M}$ respectively,
\begin{align*}
\mathcal{N}(t)&=4 \mathrm{max}(t,0)-4\mathrm{max}(-t,0) + 20 - 4\mathrm{max}(t+3,0) +3\mathrm{max}(t-3,0) ,\\
\mathcal{M}(t)&=\mathcal{N}(t)+ \mathrm{\max}(-\mathrm{\max}(t,0),0),
\end{align*}
where $\mathcal{M}$ trivially results by adding the zero function $\mathrm{\max}(-\mathrm{\max}(t,0),0)$ to $\mathcal{N}$. Nevertheless their canonical smoothings $\mathcal{N}_{\epsilon}$ and $\mathcal{M}_{\epsilon}$ under the same $\sigma_{\epsilon}$ look rather different. We note that here  we depict this for the canonical smoothing of the first example of Figure \ref{prop:relu_smoothing} but the differences are similar for  $\tilde{\sigma}_{\epsilon}$. The canonical smoothing  $\mathcal{N}_{\epsilon}$  of the one hidden layer network remains monotonically increasing ( bottom left), which is not the case for the two hidden layer network $\mathcal{M}_{\epsilon}$ (bottom middle). This is due to the term $ \mathrm{\max}(-\mathrm{\max}(t,0),0)$ whose canonical smoothing introduces a decreasing part near the origin, see bottom right plot of Figure \ref{fig:counterexamples_smoothing}.

\begin{remark}\label{rem:negative_grad}
	Note that even though canonical smoothings do not necessarily preserve monotonicity -- in particular as we saw,  if $\mathcal{N}$ is increasing, $\mathcal{N}_{\epsilon}$ does not have to be increasing as well -- nevertheless the negative part of the derivative $(\mathcal{N}_{\epsilon}')^{-}:= \mathrm{max}(- \mathcal{N}_{\epsilon}', 0)$ can be controlled. Specifically, according to Proposition \ref{prop:approx_smoothing}, if $\mathcal{N}:\mathbb{R}\to \mathbb{R}$ is a monotonically increasing ReLU network -- in particular $(\mathcal{N}')^{-}=0$, then given an open bounded $U\subset \mathbb{R}$, and $1\le p<\infty$,  we have for every canonical smoothing $\mathcal{N}_{\epsilon}$ that
	\begin{equation}\label{negative_part_grad}
	\|(\mathcal{N}_{\epsilon}')^{-}\|_{L^{p}(U)} \to 0, \quad \text{ as } \epsilon\to 0.
	\end{equation}
\end{remark}

\section{Basic facts  of the optimal control problem and implications of smoothing }\label{sec:ReLU_PDE}

Recall the main learning-informed optimal control problem:
\begin{equation}\label{P}\tag{$P_{\mathcal{N}}$}
\begin{aligned}
&\text{minimize }\quad J(y,u):= \frac{1}{2} \|y-g\|_{L^{2}(\Omega)}^{2} + \frac{\alpha}{2} \|u\|_{L^{2}(\om)}^{2},\quad \text{ over } (y,u)\in H_{0}^{1}(\om) \times L^{2}(\om),\\
&\text{subject to }\left \{
\begin{aligned}
-\Delta y + \mathcal{N}(\cdot, y)&=u, \;\; \text{ in }\Omega,\\
y&=0, \;\; \text{ on }\partial \Omega,
\end{aligned}
\right. \quad \text{ and }\quad  u\in \mathcal{C}_{ad},
\end{aligned}
\end{equation}
where the different constituents are defined in the introduction. We mention that in \cite{DonHinPapVoe22} a  more general setting was adopted by considering a function  $f$  instead of $\mathcal{N}$, belonging to a slightly larger family than the one defined by ReLU neural networks, with the main characteristic that  $y\mapsto f(x,y)$ is directionally differentiable.
Here, we note that $\mathcal{N}$ is additionally Hadamard directionally differentiable with respect to the second variable. Using the chain rule for Hadamard directionally differentiable functions we can state a recursive formula for $\mathcal{N}_{x}'(y;h)$, where for every $x\in\RR^{d}$
\[\mathcal{N}_{x}'(y;h):= \lim_{t_{n}\to 0^{+}} \frac{\mathcal{N}(x,y+t_{n}h)-\mathcal{N}(x,y)}{t_{n}}.\] 
Indeed, for $z:=(x,y)$, and for $N^{(2)}(z)=W_{2}\cdot \sigma(W_{1} z +b_{1}) + b_{2}$, $W_{2}\in \RR^{1\times n_{1}}$, $W_{1}\in \RR^{n_{1}\times (d+1)}$, $b_{1}\in \RR^{n_{1}}$, $b_{2}\in\RR$, we have that for any $y,h\in \RR$
\begin{align}\label{direct_deriv_N2}
(N^{(2)})_{x}'(y;h)=W_{2}\cdot \left ( \mathbbm{1}_{(0,\infty)}(W_{1}z+b_{1}) W_{1}(:,n_{0})h 
+  \mathbbm{1}_{\{0\}}(W_{1}z+b_{1})  \mathrm{max}(0,W_{1}(:,n_{0})h) \right ). 
\end{align}
Here $ W_{1}(:,n_{0})$ denotes the last column of $W_{1}$, and  $\mathbbm{1}_{(0,\infty)}(W_{1}z+b_{1})$ is a diagonal matrix, whose diagonal consists of the vector resulting from the componentwise action  of the function $\mathbbm{1}_{(0,\infty)}(\cdot)$ on the vector $W_{1}z+b_{1}$ -- similarly for the second summand in \eqref{direct_deriv_N2}. Recursively for $N^{(\ell)}=W_{\ell}\sigma(N^{\ell-1}(z)) + b_{\ell}$ we have
\begin{align}\label{direct_deriv_N_ell}
(N^{(\ell)})_{x}'(y;h)=W_{\ell} \cdot \left ( \mathbbm{1}_{(0,\infty)}(N^{(\ell-1)}(z)) (N^{(\ell-1)})_{x}'(y;h) 
+ \mathbbm{1}_{\{0\}} (N^{(\ell-1)}(z)) \max(0,(N^{(\ell-1)})_{x}'(y;h))  \right ).
\end{align}
Comparing the formulas \eqref{direct_deriv_N2}--\eqref{direct_deriv_N_ell} with the formula \eqref{relu_prime} for the weak gradient of $\mathcal{N}$, one notes that while \eqref{relu_prime} holds almost everywhere, the formulas for the directional derivatives hold at every point.

We will also make use of the function space 
\[Y:=\{y\in H_{0}^{1}(\om):\; \Delta y \in L^{2}(\om) \},\]
which is a separable Hilbert space equipped with the inner product $(y,v)_{Y}:= \int_{\om} \Delta y\Delta v + \nabla y \nabla v + yv\, dx$ and it is compactly embedded in $H_{0}^{1}(\om)$.
Let $N$ be the Nemytskii operator  $y\mapsto N(y)$, with $N(y)(x):=\mathcal{N}(x,y(x))$ for $y$ in some $L^{p}$ space. Note that  $N: L^{p}(\om)\to L^{p}(\om)$ is Lipschitz continuous for every fixed $1\le p\le \infty $. We also remark, see \cite[Section 3.2]{DonHinPapVoe22}, that $N:L^{p}(\om) \to L^{p}(\om)$ for $1\le p<\infty$ is Hadamard directionally differentiable with the directional derivative  $N'(y;h)\in L^{p}(\om)$ defined via $N'(y;h)(x)=\mathcal{N}_{x}'(y(x);h(x))$.
In the next theorem we briefly summarize the basic results from \cite{DonHinPapVoe22} concerning the optimal control problem \eqref{P_intro}.

\begin{theorem}[\cite{DonHinPapVoe22}]\label{thm:review}
The following hold for the learning-informed optimal control problem \eqref{P_intro} where we also assume that $p>\frac{d}{2}$ and $p\ge 2$:
\begin{enumerate}[leftmargin=*]
\item For every $u\in L^{p}(\om)$, there exists a unique solution $y\in Y\cap C^{0,a}(\overline{\om})$ for the state equation of   \eqref{P_intro}, where $a>0$ depends only on $p,d$ and $\om$. In particular, for every $M>0$ there exists a constant $c_{a}$ (that depends on $M$) such that
\begin{equation}\label{Holder_estimate}
\|y\|_{C^{0,a}(\overline{\om})} \le c_{a} \|u-\mathcal{N}(\cdot,0)\|_{L^{p}(\om)},\quad \text{for all } \|u\|_{L^{p}(\om)}\le M.
\end{equation}
\item The control-to-state map $S:L^{p}(\om)\to Y$ is Hadamard directionally differentiable, and given $u\in L^{p}(\om)$ and a direction $h\in L^{p}(\om)$, $S'(u;h):=z_{h}\in Y\cap C^{0,a}(\overline{\om})$ is the unique solution of 
	\begin{equation}\label{adjoint}\tag{$K$}
	\left \{
	\begin{aligned}
	-\Delta z_{h} + N'(y;z_{h})&=h, \;\; \text{ in }\Omega,\\
	z_{h}&=0, \;\; \text{ on }\partial \Omega,
	\end{aligned}
	\right. 
	\end{equation}
	where $y=S(u)$.
\item The optimal control problem \eqref{P_intro} has a solution.
\item ($B$-stationarity) If $\overline{u}\in L^{p}(\om)$ is a local minimizer for \eqref{P_intro}, $\overline{y}=S(\overline{u})$ is the associated state, and $\mathcal{J}(\cdot)=J(S(\cdot), \cdot)$ is the reduced objective for \eqref{P_intro}, then the pair $(\overline{u}, \overline{y})$ satisfies the following variational inequality:
	\begin{equation}\label{B_stationarity_VI}
	\mathcal{J}'(\overline{u};h)= \langle \overline{y}-g, S'(\overline{u};h)\rangle + \alpha \langle \overline{u},h\rangle \ge 0, \quad \text{for all } h\in T_{\mathcal{C}_{ad}}(\overline{u}).
	\end{equation}
Here $T_{\mathcal{C}_{ad}}(\overline{u})$ denotes the contingent cone of $\mathcal{C}_{ad}$ at $\overline{u}\in \mathcal{C}_{ad}$.
\item ($C$-stationarity)  If $\overline{u}\in L^{p}(\om)$ is a local minimizer for \eqref{P_intro}, and $\overline{y}=S(\overline{u})$ is the associated state, then the pair  $(\overline{u}, \overline{y})$ satisfies the following optimality system:
	\begin{equation}\label{eq:C_stationary}
	\left\{
	\begin{aligned}
	-  \Delta \bar{p}+   \bar{\zeta} \bar{p} = \bar{y}-g\;  &\;\; \text{ in } \Omega ,\quad
	\bar{p}=0\;   \text{ on } \partial \Omega,\\
	\bar{\zeta}(x)  \in \partial \mathcal{N}(x,\bar{y}(x))&\;\; \text{ for almost every } x\in\Omega ,\\
	(\bar{p}+  \alpha \bar{u}, u-\bar{u})\geq 0\; & \;\;\text{ for all }  u\in \mathcal{C}_{ad} ,
	\end{aligned}  \right.
	\end{equation}
for some nonnegative $\bar{\zeta}\in L^{\infty}(\om)$ and for some adjoint state $\bar{p}\in Y$. Here $\partial \mathcal{N}(x,\bar{y}(x))$ is the Clarke generalized gradient of $\mathcal{N}_{x}:=\mathcal{N}(x,\cdot): \RR\to \RR$ evaluated at $\bar{y}(x)$.
\item (Weak stationarity) We say that $\bar{u}\in L^{p}(\om)$ and $\bar{y}=S(\bar{u})$ satisfy the weak stationarity condition if  the first and the third conditions of \eqref{eq:C_stationary} are satisfied for some nonnegative $\bar{\zeta}\in L^{\infty}(\om)$ and for some adjoint state $\bar{p}\in Y$. Obviously any pair $(\bar{u}, \bar{y})$ of local minimizers for \eqref{P_intro} is weak stationary.
\end{enumerate}
\end{theorem}

We note that $T_{\mathcal{C}_{ad}}(\overline{u})$ is defined as 
\begin{equation*}\label{TCad}
T_{\mathcal{C}_{ad}}(\overline{u}):=\{h\in L^{p}(\om): \exists\, t_{n}\downarrow 0 \text{ and }h_{n}\to h\in L^{p}(\om), \text{ s.t. for all }n\in\NN, \; \overline{u}+t_{n}h_{n}\in \mathcal{C}_{ad} \}.
\end{equation*}
Note that it can be shown \cite[Lemma 6.34]{bonnans2013perturbation}, that if $\mathcal{C}_{ad}$ is of the form 
\begin{equation}\label{Cad_box}
\mathcal{C}_{ad}=\{u\in L^{p}(\om): u_a(x)\le u(x) \le u_b(x), \text{ for almost every }x\in\Omega \}
\end{equation}
with $u_{a},u_{b}\in L^{\infty}(\om)$, $u_{a}<u_{b}$ almost everywhere, then $T_{\mathcal{C}_{ad}}(\overline{u})$ can be characterized by

\begin{equation}\label{TCad_box}
T_{\mathcal{C}_{ad}}(\overline{u})=
\Bigg \{h\in L^{p}(\om):\;\;
\begin{aligned}
h(x)\ge 0,& \;\; \text{almost everywhere in } \{x\in \om: \overline{u}(x)=u_a(x)\}\\
h(x)\le 0,& \;\; \text{almost everywhere in } \{x\in \om: \overline{u}(x)=u_b(x)\}
\end{aligned}
\Bigg\}.
\end{equation}

Apart from the primal notion of $B$-stationarity, and the primal-dual notions of weak and $C$-stationarity also one more primal-dual stationarity concept was discussed in \cite{DonHinPapVoe22}, namely strong stationarity. There, the relationships between all these concepts were analyzed. 
Here we focus on $B$-stationarity, and in particular our developed algorithm studied in Section \ref{sec:algorithm} builds on that notion. 
We only mention that  the $C$-stationarity system (which is weaker than strong stationarity) is obtained as a limiting optimality system where $\mathcal{N}$ is substituted by some smooth version $\mathcal{N}_{\epsilon}$ and the smoothing parameter $\epsilon$ vanishes. In that case the smoothing of the network $\mathcal{N}$ does not need to be canonical as it is only used as a tool in order to get this stationarity system in the limit. 
 Next, we discuss the limitations that arise when this regularization is used not in order to study the limiting case, but in order to solve the corresponding regularized optimal control problem with a classical numerical solver, via smoothing the problem for a fixed $\epsilon>0$.

\subsection*{Implications of the ReLU smoothing on the uniqueness of the state equation}
There are, in general, two levels of approximation involved in the optimal control of learning-informed PDEs. The first level of approximation arises from the approximation of $f$ by a sequence of ReLU neural networks $\mathcal{N}_{n}$  and can be thought as the capability of the ReLU-informed PDE to approximate some ground truth nonsmooth physical model. This is studied in \cite[Proposition 3.3]{DonHinPapVoe22}. The second level of approximation -- as we mentioned above -- considers the approximating PDEs that arise after smoothing the ReLU network in order to  treat the problem algorithmically with classical solvers. 
As we have mentioned in the introduction, due to the potentially large architecture of a network $\mathcal{N}$ (large number of layers and neurons), a natural and efficient way to smoothen it (after its training has been completed) would via simply smoothing the ReLU function, with the canonical smoothing procedure described in the previous section. This would result in the following smoothed version of the ReLU learning-informed PDE
\begin{equation}\label{state_N_eps}\tag{$E_{\mathcal{N}_{\epsilon}}$}
\left \{
\begin{aligned}
-\Delta y + \mathcal{N}_{\epsilon}(\cdot, y)&=u, \;\; \text{ in }\Omega,\\
y&=0, \;\; \text{ on }\partial \Omega.
\end{aligned}
\right. 
\end{equation}
However, nonuniqueness issues for the solutions of \eqref{state_N_eps} can arise, as demonstrated above in Proposition \ref{prop:relu_smoothing}, since the resulting canonically smoothed network $\mathcal{N}_{\epsilon}$ is not necessarily monotonically increasing. Uniqueness for the solutions of \eqref{state_N_eps} could be derived by showing that the operator $A_{\epsilon}: H_{0}^{1}(\om)\to H^{-1}(\om)$ with 
	\[\langle A_{\epsilon}(y), z \rangle_{H^{-1}(\om), H_{0}^{1}(\om)}:= \int_{\om} \nabla y \nabla z\,dx + \int_{\om} \mathcal{N}_{\epsilon}(x,y) z\, dx,\]
is strongly monotone and then applying the Browder-Minty theorem. This is certainly the case if  $ \mathcal{N}_{\epsilon}$ was monotone in $y$, but it could also follow, at least for small $\epsilon>0$, see \cite[Proposition 3.3]{DonHinPapVoe22}, if  $\nabla \mathcal{N}_{\epsilon}\to\nabla \mathcal{N}$ uniformly. However in the case of a canonical smoothing $\mathcal{N}_{\epsilon}$, the convergence of $\nabla \mathcal{N}_{\epsilon}$ to $\nabla \mathcal{N}$ as $\epsilon\to 0$ can only be guaranteed to hold with respect to the $L^{p}$ norm, for every $1\le p<\infty$, see \eqref{cs_Lp_derivatives}. The potential nonuniform convergence of $\nabla \mathcal{N}_{\epsilon}$ to $\nabla \mathcal{N}$ makes the application of the Browder-Minty theorem problematic. In order to be more precise, it would suffice as in the proof of \cite[Proposition 3.3]{DonHinPapVoe22}, to show that for every $\eta>0$ there exists $\epsilon_{0}>0$ such that for every $0<\epsilon<\epsilon_{0}$
\begin{equation}\label{eps_monotone}
	\int_{\om} (\mathcal{N}_{\epsilon}(x,y_{1})-\mathcal{N}_{\epsilon}(x,y_{2}))(y_{1}-y_{2})\,dx \ge -\eta \|y_{1}-y_{2}\|_{L^{2}(\om)}^{2},
\end{equation}
for all $y_{1}, y_{2}\in H_{0}^{1}(\om)$. Indeed in that case, denoting by $c_{\om}$ the Poincar\'e inequality constant, we would have for every $y_{1}, y_{2}\in H_{0}^{1}(\om)$
\begin{align*}
\langle A_{\epsilon} (y_{1}) - A_{\epsilon}(y_{2}), y_{1}-y_{2} \rangle
&\ge \frac{1}{(c_{\om}+1)^{2}} \|y_{1}-y_{2}\|_{H_{0}^{1}}^{2} + \int_{\om} (\mathcal{N}_{\epsilon}(x,y_{1})-\mathcal{N}_{\epsilon}(x,y_{2}))(y_{1}-y_{2})\,dx\\
&\ge \left ( \frac{1}{(c_{\om}+1)^{2}}  -\eta\right )\|y_{1}-y_{2}\|_{H_{0}^{1}}^{2}, 
\end{align*}
and thus by choosing $0<\eta<1/(c_{\om}+1)^{2}$ we would get strong monotonicity for the operator $A_{\epsilon}$ for small enough $\epsilon>0$.
Consider now the example of Figure \ref{fig:counterexamples_smoothing}, where for $\mathcal{N}_{\epsilon}: \RR \to \RR$ it holds that there exists a $c>0$ such that for every $\epsilon>0$, there exists a $\delta>0$ such that
\begin{equation*}
	\nabla \mathcal{N}_{\epsilon}(t)=\mathcal{N}_{\epsilon}'(t)<-c, \quad \text{ for every }t\in (-\delta, \delta). 
\end{equation*}
This means that for every $y_{1}<y_{2}\in H_{0}^{1}(\om)$ with values in $(-\delta, \delta)$ a pointwise application of the mean value theorem gives for some $\theta$, with $\theta(x)\in (y_{1}(x), y_{2}(x))$
\begin{equation*}
	\int_{\om} (\mathcal{N}_{\epsilon}(y_{1})-\mathcal{N}_{\epsilon}(y_{2}))(y_{1}-y_{2})\,dx=\int_{\om} \mathcal{N}'_{\epsilon}(\theta) (y_{1}-y_{2})^{2}dx < - c \|y_{1}-y_{2}\|_{L^{2}(\om)}^{2}.
\end{equation*}
Hence if  $c>0$ turns out to be large, the absorption of the last term into $\frac{1}{(c_{\om}+1)^{2}}\|y_{1}-y_{2}\|_{H_{0}^{1}(\om)}^{2}$ is not possible. We note however that one can still prove existence of solutions for the PDEs with nonmonotone nonlinearity for instance by showing that the latter is equivalent to the Euler-Lagrange equation of an associated variational problem, see for instance \cite{DonHinPap20} or by using the theory of type $M$ operators as it is done in the next section. Nevertheless uniqueness can no longer be guaranteed.

Having a  (canonical) smoothing $\mathcal{N}_{\epsilon}$ of $\mathcal{N}$, that satisfies the properties  of Proposition \ref{prop:approx_smoothing} with the additional property that $\mathcal{N}_{\epsilon}(x,\cdot)$ is monotonically increasing for every $x\in \om$, could  be theoretically achieved in two ways: The first way would be to take advantage of the fact that any ReLU network $\mathcal{N}$ of arbitrary number of layers can be realized by a ReLU network of one hidden layer. Then one could use the canonical smoothing derived from convolution $\sigma_{\epsilon}:=\rho_{\epsilon}\ast \sigma$ that preserves monotonicity, see $(i)$ of Proposition \ref{prop:relu_smoothing}. Of course such an approach would not necessarily work in practice in the case  one wants to use a classical solver in order to solve a smooth version of \eqref{P}, since the one-hidden layer version of $\mathcal{N}$ cannot be easily derived. The second way, would be to consider  abandoning the canonical smoothing approach and smooth directly the multilayer network as $\mathcal{N}_{\epsilon}:=\rho_{\epsilon}\ast \mathcal{N}$. While such an approach preserves monotonicity, the computation of a convolution of the network could be computationally demanding and the resulting function  cannot necessarily be represented by a neural network. Hence, both approaches appear impractical.

Our discussion here should serve as a warning that using feasible canonical smoothing  approaches of $\mathcal{N}$ with the target of solving a smooth approximating problem to $\eqref{P}$ using standard algorithms could be problematic since  multiple solutions for the smoothed state equation might be introduced by this process. This provides a further motivation for  designing algorithms that directly solve the nonsmooth problem as we do in the following Sections \ref{sec:algorithm} and \ref{sec:numerics}.

\section{A descent algorithm for B-stationarity} \label{sec:algorithm}
In this section we introduce a descent algorithm for the ReLU network learning-informed optimal control problem \eqref{P} and discuss its convergence. We recall that $\mathcal{N}$ is assumed monotone in $y$ which gives rise to a unique solution of the learning-informed state equation.
\subsection{A descent algorithm}
We aim to compute local minimizers for \eqref{P} that satisfy certain stationarity conditions, as outlined in Theorem \ref{thm:review}. Here we are particularly interested in $B$-stationarity, i.e.,\ control-state pairs $(u,y)$ that satisfy the following variational inequality:
\begin{equation}\label{B_stationarity_VI}
\mathcal{J}'(u;h)= \langle y-g, S'(u;h)\rangle + \alpha \langle u,h\rangle \ge 0, \quad \text{for all } h\in T_{\mathcal{C}_{ad}}(u).
\end{equation}
For the ease of exposition, from now on we focus on the case where $\mathcal{C}_{ad}$ is of the form \eqref{Cad_box} and thus $T_{\mathcal{C}_{ad}}(u)$ can be written  as in \eqref{TCad_box}.

We proceed in terms of the reduced version of \eqref{P}, i.e., by considering the state as dependent on $u$, i.e., $y=S(u)$, which allows to eliminate the state as an independent variable. Then, given some $u\in \mathcal{C}_{ad}$, following  \cite{Hint_Suro} we consider the following auxiliary problem:
\begin{equation}\label{eq:aux_pro}
\operatorname{minimize} \quad \langle S(u)-g,S'(u;h)\rangle + \alpha \langle u,h\rangle \quad \text{ over }  h\in T_{\mathcal{C}_{ad}}(u).
\end{equation}
Note that according to the definition of $B$-stationarity \eqref{B_stationarity_VI},  it holds that  $h=0\in T_{\mathcal{C}_{ad}}(u)$ is a solution of \eqref{eq:aux_pro} if and only if $(u,S(u))$ is a $B$-stationary point. We point out that when $(u,S(u))$ is not $B$-stationary, then problem \eqref{eq:aux_pro} is not necessarily well-posed. As a remedy, we introduce the regularized version 
\begin{equation}\label{eq:reg_aux_pro}
\operatorname{minimize} \quad \frac{1}{2}q(h,h)+  \langle S(u)-g,S'(u;h)\rangle + \alpha \langle u,h\rangle \quad \text{ over }  h\in T_{\mathcal{C}_{ad}}(u),
\end{equation}
where $q: L^{2}(\om)\times L^{2}(\om)\to \RR$ is a symmetric  functional with $v\mapsto q(v,v)$ convex, differentiable (typically quadratic, hence the notation) and for every $v,v'\in L^{2}(\om)$ satisfying 
\begin{equation}\label{g_coercive_bounded}
q(v,v)\geq C_1\norm{v}^2_{L^2(\Omega)} \quad \text{ and }\quad q(v,v')\leq C_2\norm{v}_{L^2(\Omega)}\norm{v'}_{L^2(\Omega)}, 
\end{equation}
for some constants $C_{1}, C_{2}>0$.
Note that according to \cite[Lemma 2.1]{Hint_Suro} $h=0$ is  a solution of \eqref{eq:reg_aux_pro} if and only $h=0$ is a solution of \eqref{eq:aux_pro}.
Furthermore, the following proposition holds.
\begin{proposition}
	Let $u\in \mathcal{C}_{ad}$ be a feasible point for the reduced version of \eqref{P}. Then the following properties are satisfied:
	\begin{enumerate}
		\item[(1)] The problem \eqref{eq:reg_aux_pro} admits a solution $\bar{h}\in T_{\mathcal{C}_{ad}}$(u).
		\item[(2)] If $\bar{h}\neq 0$, then $\bar{h}$ is a descent direction for the reduced objective $\mathcal{J}$ associated with \eqref{P}.
		\item[(3)] If the directional derivative $S'(u;\cdot): L^{p}(\om)\to Y$ is bounded and linear, then $\bar{h}$ is unique.
	\end{enumerate}
\end{proposition}
\begin{proof}
The proof is essentially the same as the one of \cite[Proposition 2.3]{Hint_Suro}, with the only difference that $h$ is constrained to $T_{\mathcal{C}_{ad}}(u)$ instead of the whole $L^{p}(\om)$.
For the first assertion we only need to notice that  $T_{\mathcal{C}_{ad}}(u)$ is non-empty, convex and closed due to the assumption that $\mathcal{C}_{ad}$ is non-empty, convex and closed. Then existence of solutions follows from  the direct method of the calculus of variations. For the second one, notice that  since $u$ is feasible, it follows that $0\in T_{\mathcal{C}_{ad}}(u)$. Therefore the same argument as \cite[Proposition 2.3]{Hint_Suro} can be applied here. The third assertion follows from the strong convexity of the resulting problem.
\end{proof}
From this discussion it follows that for computing a descent direction for the reduced version of \eqref{P} at a non $B$-stationary point $u$, it suffices to solve \eqref{eq:reg_aux_pro}. 
Notice, however, that solving \eqref{eq:reg_aux_pro} is delicate whenever $S'(u; \cdot)$ is not bounded and linear. The latter is typically connected to active nonsmoothness of $\mathcal{N}$, that is when the set
\[\om_{\mathcal{N}}(u):= \{ x\in \om: \; \mathcal{N}(x,\cdot) \text{ is nondifferentiable at } y(x)=S(u)(x) \},\]
has a strictly positive Lebesgue measure (which we denote by $\mathfrak{m}$). In  such a situation we will consider a specific approximation of \eqref{eq:reg_aux_pro} as detailed below.
 Note that $\om_{\mathcal{N}}(u)$ is Lebesgue measurable since $\mathcal{N}$ is jointly continuous on $\om\times \RR$. We mention also that in the case where $\om_{\mathcal{N}}(u)$ has zero Lebesgue measure then \eqref{eq:reg_aux_pro} is a standard quadratic problem, presuming $q$ quadratic.

The specific approximation of \eqref{eq:reg_aux_pro} which we utilize in the nonsmooth case consists of a substitution of the nonlinear (and nonsmooth) map $S'(u;\cdot)$ by a differentiable approximation $\Pi_{\epsilon}(u;\cdot)$. More precisely, fixing an $\epsilon>0$, we define $d_{\epsilon}\in \Pi_{\epsilon}(u;h)$ where $d_{\epsilon}$ is a solution of the problem 
\begin{equation}\label{eq:smoothed_der}
\left \{
\begin{aligned}
-\Delta d_\epsilon + D_\epsilon(y;d_\epsilon)&=h , \;\; \text{ in }\Omega,\\
d_\epsilon &=0, \;\; \text{ on }\partial \Omega.
\end{aligned}
\right. 
\end{equation}
The crucial point here is that  $ D_\epsilon$ is the Nemytskii operator that corresponds to a function $\mathcal{D}_{\epsilon}$ which is   smooth with respect to the second variable but it does not correspond to the derivative $\mathcal{N}_{\epsilon}'$ of some smoothing $\mathcal{N}_{\epsilon}$ of $\mathcal{N}$.
In order  to define $\mathcal{D}_{\epsilon}$ we fix a canonical smoothing $(\sigma_{\epsilon})_{\epsilon>0}$ of the ReLU function such that also \eqref{sigma_uniform_values_stronger} holds. 
Then $\mathcal{D}_{\epsilon}$ is defined by simply substituting the ReLU (the max function) by $\sigma_{\epsilon}$ whenever this ReLU is applied to the direction $d$, but leaving the derivative of ReLU intact, wherever that appears in the recursive formulas \eqref{direct_deriv_N2}--\eqref{direct_deriv_N_ell}  for the directional derivative $\mathcal{N}_{x}'(y;d)$ of the ReLU network $\mathcal{N}$. Specifically, for $z:=(x,y)$,
\begin{align*}
(\mathcal{D}_{\epsilon}^{(2)})_{x}(y;d)&=W_{2}\cdot \left ( \mathbbm{1}_{(0,\infty)}(W_{1}z+b_{1}) W_{1}(:,n_{0})d 
+  \mathbbm{1}_{\{0\}}(W_{1}z+b_{1})  \sigma_{\epsilon}(W_{1}(:,n_{0})d) \right ),\\ 
(\mathcal{D}_{\epsilon}^{(\ell)})_{x}(y;d)&=W_{\ell} \cdot \left ( \mathbbm{1}_{(0,\infty)}(N^{(\ell-1)}(z)) (\mathcal{D}_{\epsilon}^{\ell-1})_{x}(y;d) 
+ \mathbbm{1}_{\{0\}} (N^{(\ell-1)}(z)) \sigma_{\epsilon}((\mathcal{D}_{\epsilon}^{(\ell-1)})_{x}(y;d))  \right ),\\
(\mathcal{D}_{\epsilon})_{x}(y;d)&=(\mathcal{D}_{\epsilon}^{(L)})_{x}(y;d),
\end{align*}
with $\ell=3, \ldots, L$, where $L$ is the number of layers of $\mathcal{N}$. It is easy to check that the regularity of $\mathcal{D}_{\epsilon}$ with respect to $d$ is dictated by the regularity of $\sigma_{\epsilon}$.  For the sake of clarity, we state the formulas of $\mathcal{D}_{\epsilon}$ for the case of one and two-hidden layer ReLU networks, where also for simplicity, there is no explicit dependence on $x$, i.e.,  $\mathcal{N}: \RR \to \RR$. For the one-hidden layer case we have for $W_{1}=(w_{1}^{1}, \ldots, w_{1}^{n_{1}})^{T}$, $W_{2}=(w_{2}^{1}, \ldots, w_{2}^{n_{1}})$, $b_{1}=(b_{1}, \ldots, b_{n_{1}})^{T}$, $b_{2}\in \RR$,
\begin{align*}
\mathcal{N}(y)&=b_{2}+ \sum_{i=1}^{n_{1}} w_{2}^{i} \max( w_{1}^{i}y+b_{1}^{i},0),\\
\mathcal{N}'(y;d)&=\sum_{i=1}^{n_{1}} w_{2}^{i} \left (  \mathbbm{1}_{(0,\infty)}(w_{1}^{i}y+b_{1}^{i}) w_{1}^{i} d +\mathbbm{1}_{\{0\}}(w_{1}^{i}y+b_{1}^{i}) \max(w_{1}^{i}d,0) \right ),\\
\mathcal{D}_{\epsilon}(y;d)&=\sum_{i=1}^{n_{1}} w_{2}^{i} \left (  \mathbbm{1}_{(0,\infty)}(w_{1}^{i}y+b_{1}^{i}) w_{1}^{i} d +\mathbbm{1}_{\{0\}}(w_{1}^{i}y+b_{1}^{i}) \sigma_{\epsilon}(w_{1}^{i}d) \right ).
\end{align*}
On the other hand for a two-hidden layer case we have, for $W_{1}=(w_{1}^{1}, \ldots, w_{1}^{n_{1}})^{T}$, $W_{2}=(w_{2}^{j,i})_{j,i}$, $i=1, \ldots, n_{1}$, $j=1, \ldots, n_{2}$, $W_{3}=(w_{3}^{1}, \ldots , w_{3}^{n_{2}})$, $b_{1}=(b_{1}, \ldots, b_{n_{1}})^{T}$, $b_{2}=(b_{1}, \ldots, b_{n_{2}})^{T}$, $b_{3}\in \RR$,
\begin{align*}
\mathcal{N}(y)&= b_{3} + \sum_{j=1}^{n_{2}} w_{3}^{j} \max \left (
\sum_{i=1}^{n_{1}} w_{2}^{j,i} \max(w_{1}^{k} y +b_{1}^{k}, 0) +b_{2}^{j}
  ,0\right ),\\
  \mathcal{N}'(y;d)
  &=\sum_{j=1}^{n_{2}} w_{3}^{j} \mathbbm{1}_{(0,\infty)}(v^{j})\left ( \sum_{i=1}^{n_{1}} w_{2}^{j,i} \left (  \mathbbm{1}_{(0,\infty)}(w_{1}^{i}y+b_{1}^{i}) w_{1}^{i} d +\mathbbm{1}_{\{0\}}(w_{1}^{i}y+b_{1}^{i}) \max(w_{1}^{i}d,0) \right ) \right)\\
  &\;\;+ \sum_{j=1}^{n_{2}} w_{3}^{j} \mathbbm{1}_{\{0\}}(v^{j}) \max\left( \sum_{i=1}^{n_{1}} w_{2}^{j,i} \left (  \mathbbm{1}_{(0,\infty)}(w_{1}^{i}y+b_{1}^{i}) w_{1}^{i} d +\mathbbm{1}_{\{0\}}(w_{1}^{i}y+b_{1}^{i}) \max(w_{1}^{i}d,0) \right )  ,0\right),
  \end{align*}
  \begin{align*}
  \mathcal{D}_{\epsilon}(y;d)&=\sum_{j=1}^{n_{2}} w_{3}^{j} \mathbbm{1}_{(0,\infty)}(v^{j})\left ( \sum_{i=1}^{n_{1}} w_{2}^{j,i} \left (  \mathbbm{1}_{(0,\infty)}(w_{1}^{i}y+b_{1}^{i}) w_{1}^{i} d +\mathbbm{1}_{\{0\}}(w_{1}^{i}y+b_{1}^{i}) \sigma_{\epsilon}(w_{1}^{i}d) \right ) \right)\\
  &\;\;+ \sum_{j=1}^{n_{2}} w_{3}^{j} \mathbbm{1}_{\{0\}}(v^{j}) \sigma_{\epsilon}\left( \sum_{i=1}^{n_{1}} w_{2}^{j,i} \left (  \mathbbm{1}_{(0,\infty)}(w_{1}^{i}y+b_{1}^{i}) w_{1}^{i} d +\mathbbm{1}_{\{0\}}(w_{1}^{i}y+b_{1}^{i}) \sigma_{\epsilon}(w_{1}^{i}d) \right ) \right),
\end{align*}
where $v^{j}=\sum_{i=1}^{n_{1}} w_{2}^{j,i} \max(w_{1}^{k} y +b_{1}^{k}, 0) +b_{2}^{j}$.\\

We have the following approximation result.

\begin{lemma}\label{lem:conv_nn}
Let $1\le p< \infty$, $\mathcal{N}:\om\times \RR\to \RR$ be a ReLU neural network,  $N: L^{p}(\om)\to L^{p}(\om)$ its corresponding Nemytskii operator, $u\in L^{p}(\om)$, $y=S(u)$ and $N'(y;\cdot): L^{p}(\om) \to L^{p}(\om)$ be the directional derivative of $N$. Then for the operator $L^{p}(\om)\ni d \mapsto D_\epsilon(y;d)$ it holds that
\begin{equation}\label{N_prime_D_eps}
\norm{N'(y;d)- D_\epsilon(y;d)}_{L^p(\Omega)}\leq C \epsilon \; \text{ for all }\; d\in L^p(\Omega),
\end{equation}
where $C>0$ is some constant independent of $\epsilon$. In particular $D_\epsilon(y;\cdot): L^{p}(\om) \to L^{p}(\om)$.
\end{lemma}
\begin{proof}
The proof is straightforward via  induction over the number of layers of $\mathcal{N}$, using \eqref{sigma_uniform_values_stronger}, and thus we  omit the details.
\end{proof}

Note that in particular for the functions $\mathcal{N}'(y(\cdot);\cdot), \mathcal{D}_{\epsilon}(y(\cdot);\cdot): \om \times \RR \to \RR$
we also have that there exists a constant $C>0$ such that for every $\epsilon>0$, $d\in \RR$ and for almost every $x\in\om$, 
\begin{equation}\label{N_prime_D_eps_2}
|\mathcal{N}'(y(x);d)-\mathcal{D}_{\epsilon}(y(x);d)|<C \epsilon.
\end{equation}
 Using the fact that $y\in L^{\infty}(\om)$, as well as $|\sigma_{\epsilon}'|\le 1$, it can also be deduced that
$\mathcal{D}_{\epsilon}(y(x);\cdot)$ is uniformly Lipschitz, i.e., there exists a $c>0$ such that for every $d_{1}, d_{2}\in \RR$, for every $\epsilon>0$ and almost every $x\in \om$
\begin{equation}\label{LipschitzD_eps}
|\mathcal{D}_{\epsilon}(y(x);d_{1})-\mathcal{D}_{\epsilon}(y(x);d_{2})|\le c |d_{1}-d_{2}|.
\end{equation}
In particular this also implies that
 there exist $a,b>0$ such that for every $d\in \RR$ and almost every $x\in \om$
\begin{equation}\label{atmostlinearD_eps}
|\mathcal{D}_{\epsilon}(y(x);d)|\le a+ b |d|.
\end{equation}

\begin{remark}\label{C_independent_of_y}
We note that  the constant $C>0$ in \eqref{N_prime_D_eps_2}, and hence also the one in \eqref{N_prime_D_eps}, can be considered to be independent of the state $y$ and as a result also independent of the corresponding control $u$. Indeed, observe that $C$ is dependent on the $L^{\infty}(\om)$-norm of $y$, but given the estimate \eqref{Holder_estimate} and the fact that every $u$ considered here belongs to the box constraint set $\mathcal{C}_{ad}$ of the form \eqref{Cad_box}, we have that the $L^{\infty}(\om)$-norm of $y$ is uniformly bounded.
\end{remark}

We note  that one cannot necessarily expect  the functions $\mathcal{D}_{\epsilon}(y(x);\cdot)$ to be  monotone, see the discussion in Section \ref{sec:smoothingReLU}. Hence the Browder-Minty theorem can no longer be applied, in order to get the existence of a unique solution for the regularized adjoint equation \eqref{eq:smoothed_der}. Nevertheless existence of solutions can be shown via applying the theory of \emph{type $M$} operators, see \cite{showalter2013monotone}. We recall that if $V$ is a reflexive Banach space, and $V^{\ast}$ is its dual, then an operator $\mathcal{A}: V \to V^{\ast}$ is called to be of type $M$ whenever it holds that if $d_{n}\rightharpoonup d$, $\mathcal{A}d_{n} \rightharpoonup h$ and $\limsup_{n} \langle \mathcal{A}d_{n}, d_{n} \rangle \le \langle h, d \rangle$ then it follows that $\mathcal{A}d=h$. The corresponding proposition follows next.

\begin{proposition}\label{prop:existence_Deps}
For every $h\in L^{2}(\om)$, the equation \eqref{eq:smoothed_der} admits a solution $d_{\epsilon}\in H_{0}^{1}(\om)$.
\end{proposition}
\begin{proof}
According to \cite[Corollary 2.2]{showalter2013monotone} it suffices to show that  $\mathcal{A}: H_{0}^{1}(\om)\to H^{-1}(\om)$ is type $M$, bounded and coercive where for every $d,v\in H_{0}^{1}(\om)$
\begin{equation}\label{mathcal_A}
\mathcal{A}d(v):= \langle \nabla d, \nabla v \rangle +  \langle D_{\epsilon}(y;d), v \rangle. 
\end{equation}
Note that the second term on the right-hand side of \eqref{mathcal_A} is well-defined due to \eqref{atmostlinearD_eps}. The first term of \eqref{mathcal_A} defines a hemicontinuous and monotone operator and hence it is of type $M$, \cite[Lemma 2.1]{showalter2013monotone}. Thus in order to show that $\mathcal{A}$ is of type $M$, according to \cite[Example 2.B]{showalter2013monotone} it suffices to show that the operator $\mathcal{B}:H_{0}^{1}(\om)\to H^{-1}(\om)$
\[\mathcal{B}d(v):=\langle D_{\epsilon}(y;d), v \rangle\]
is completely continuous, i.e., whenever $d_{n}\rightharpoonup d$ in $H_{0}^{1}(\om)$ it holds that $\mathcal{B}d_{n}\to \mathcal{B}d$ strongly in $H^{-1}(\om)$. Indeed from the compact embedding of $H_{0}^{1}(\om)$ into $L^{2}(\om)$ we have that $d_{n} \to d$ in $L^{2}(\om)$. Using \eqref{LipschitzD_eps} we estimate
\begin{equation}
\|D_{\epsilon}(y;d_{n})-D_{\epsilon}(y;d)\|_{L^{2}(\om)}\le c \|d_{n}-d\|_{L^{2}(\om)}
\end{equation}
and thus $D_{\epsilon}(y;d_{n})\to D_{\epsilon}(y;d)$ in $L^{2}(\om)$ which  implies that $\mathcal{B}d_{n}\to \mathcal{B}d$ strongly in $H^{-1}(\om)$. Finally, clearly $\mathcal{A}:H_{0}^{1}(\om)\to H^{-1}(\om)$ is a bounded operator, and also coercive.
Indeed, for the latter property,  we have for $d\in H_{0}^{1}(\om)$ that
\begin{align*}
\frac{\mathcal{A}d(d)}{\|d\|_{H_{0}^{1}(\om)}}
&\ge \frac{1}{(c_{\om}+1)^{2}} \|d\|_{H_{0}^{1}(\om)} + \frac{1}{\|d\|_{H_{0}^{1}(\om)}} \langle N'(y;d),d \rangle+  \frac{1}{\|d\|_{H_{0}^{1}(\om)}} \langle D_{\epsilon}(y;d)- N'(y;d),d \rangle\\
&\ge \frac{1}{(c_{\om}+1)^{2}} \|d\|_{H_{0}^{1}(\om)}+  \frac{1}{\|d\|_{H_{0}^{1}(\om)}} \underbrace{\langle N'(y;d)-N'(y;0),d-0 \rangle}_{\ge 0} - \frac{\tilde{C}\epsilon}{\|d\|_{H_{0}^{1}(\om)}} \|d\|_{L^{2}(\om)},
\end{align*}
for some constant $\tilde{C}>0$. Here, $c_{\om}$ is the Poincar\'e constant and 
we have used the fact that $\mathcal{N}'(y;d)$ is monotonically increasing with respect to $d$ and also \eqref{N_prime_D_eps_2}.

\end{proof}

Upon fixing an $\epsilon>0$, we  use a  solution of \eqref{eq:smoothed_der}, denoted by $d_{\epsilon}=d_{\epsilon}(h)\in \Pi_{\epsilon}(u;h)$,  to replace $S'(u;h)$ when $\om_{\mathcal{N}}(u)$ has positive Lebesgue measure.
In particular,  \eqref{eq:reg_aux_pro} is approximated by the  following problem:
\begin{equation}\label{eq:smo_reg_aux_pro}
\begin{aligned}
&\text{minimize }\quad  \frac{1}{2}q(h,h)+ \langle S(u)-g,d_{\epsilon}\rangle + \alpha \langle u,h \rangle\quad \text{ over } h\in L^{2}(\om),\; d_{\epsilon}\in H_{0}^{1}(\om) \\
&\text{subject to }\left \{
\begin{aligned}
-\Delta d_{\epsilon} + D_{\epsilon}(y;d_{\epsilon})&=h, \;\; \text{ in }\Omega,\\
d_{\epsilon}&=0, \;\; \text{ on }\partial \Omega,
\end{aligned}
\right. \quad \text{ and }\quad  h\in T_{\mathcal{C}_{ad}}(u).
\end{aligned}
\end{equation}

\begin{proposition}\label{prop:smo_reg_aux_pro}
The minimization problem \eqref{eq:smo_reg_aux_pro} has a solution.
\end{proposition}
\begin{proof}
The first claim is that there exist  constants $c_{1}, c_{2}>0$ independent of $h$ and small   $\epsilon>0$ such that the following estimate holds true
\begin{equation}\label{d_eps_estimate}
\|d_{\epsilon}\|_{H_{0}^{1}(\om)}\le c_{1}+ c_{2}\|h\|_{L^{2}(\om)},
\end{equation}
from which it straightforwardly follows that the objective in \eqref{eq:smo_reg_aux_pro} is bounded from below and coercive in $L^{2}(\om)$. In order to show \eqref{d_eps_estimate} we add and subtract $N'(y;d_{\epsilon})$ in \eqref{eq:smoothed_der} and test with $d_{\epsilon}$ getting
\begin{align*}
&\|\nabla d_{\epsilon}\|_{L^{2}(\om)}^{2} + \langle D_{\epsilon}(y;d_{\epsilon}) -  N'(y;d_{\epsilon}), d_{\epsilon}  \rangle + \underbrace{\langle N'(y;d_{\epsilon}), d_{\epsilon} \rangle}_{\ge 0} =  \langle h, d_{\epsilon} \rangle\\
\Rightarrow 
&\|\nabla d_{\epsilon}\|_{L^{2}(\om)}^{2} - \tilde{C}\epsilon \|d_{\epsilon}\|_{L^{2}(\om)} \le \|h\|_{L^{2}(\om)} \|d_{\epsilon}\|_{L^{2}(\om)}.
\end{align*}
By estimating the $H_{0}^{1}$  norm by the $L^{2}$  norm using the Poincar\'e inequality and by dividing by $\|d_{\epsilon}\|_{H_{0}^{1}(\om)}$ we have the result. Consider now two minimizing sequences $(h_{n})_{n\in\NN}$ and $(d_{\epsilon}^{n})_{n\in\NN}$. From the coercivity of the objective and from the estimate \eqref{d_eps_estimate} it follows that these are bounded in $L^{2}(\om)$ and $H_{0}^{1}(\om)$ respectively and hence there exist $h^{\ast}\in L^{2}(\om)$ and $d_{\epsilon}^{\ast}$ in $H_{0}^{1}(\om)$ such that $h_{n} \rightharpoonup h^{\ast}$ in $L^{2}(\om)$ and $d_{\epsilon}^{n} \rightharpoonup d_{\epsilon}^{\ast}$ in $H_{0}^{1}(\om)$. Since $T_{\mathcal{C}_{ad}}(u)$ is convex and $L^{2}$-strongly closed it follows that $h^{\ast}\in T_{\mathcal{C}_{ad}}(u)$. It remains to show that $(h^{\ast}, d_{\epsilon}^{\ast})$ is a feasible pair, i.e., it satisfies \eqref{eq:smoothed_der}. For this it suffices to show that $D_{\epsilon}(y;d_{\epsilon}^{n})\rightharpoonup D_{\epsilon}(y;d_{\epsilon}^{\ast})$ weakly in $L^{2}(\om)$, which follows similarly as in the proof of Proposition \ref{prop:existence_Deps}.
The proof is complete in view of the lower semicontinuity of the objective in \eqref{eq:smo_reg_aux_pro} with respect to the corresponding weak convergences.

\end{proof}

\vspace{1em}

In the remainder of this section, we show that for sufficiently small $\epsilon>0$, we are still able to find a descent direction by solving \eqref{eq:smo_reg_aux_pro} instead of \eqref{eq:reg_aux_pro}. We start with the following lemma.

\begin{lemma}\label{lem:conv_dd}
	Let $\mathcal{N}$ be a ReLU neural network, $u, h\in L^{p}(\om)$, $y=S(u)$, $\epsilon>0$, and  let $d=S'(u;h)$, $d_{\epsilon}\in \Pi_\epsilon(u;h)$ be defined as before. Then the following estimate holds: 
\begin{equation}\label{P_eps_S_prime_estimate}
\norm{d_{\epsilon}- S'(u;h)}_{H^1(\Omega)}\leq C \norm{N'(y;d_\epsilon)-D_\epsilon(y;d_\epsilon)}_{L^2(\Omega)}.
\end{equation}
 with a constant $C>0$ independent of $h$ and $\epsilon$. In particular in view of \eqref{N_prime_D_eps} the inequality 
 \begin{equation}\label{P_eps_S_prime_Ceps}
 \|d_{\epsilon}- S'(u;h)\|_{H^{1}(\om)}\le C\epsilon
 \end{equation} 
 holds for a generic constant $C>0$ still independent of $h$ and $\epsilon$.
 \end{lemma}
\begin{proof}
We have that $d_{\epsilon}, d$ satisfy
	\[- \Delta  d_\epsilon + D_\epsilon(y;d_\epsilon)=h \text{,  } - \Delta  d + N'(y;d)=h,\;\text{ in } \Omega, \quad \text{ and }\quad d_{\epsilon}=d =0 \;\text{ on  }\; \partial \Omega.\]
	It follows that  $e_\epsilon := d_\epsilon - d$, satisfies 
	\[- \Delta  e_\epsilon + N'(y;d_\epsilon)-N'(y;d)= N'(y;d_\epsilon)-D_\epsilon(y;d_\epsilon)\;\text{ in } \Omega,\quad \text{ and }\quad e^h_\epsilon =0 \;\text{ on  }\; \partial \Omega.\]
Identifying $N'(y;d_\epsilon)-N'(y;d)=\xi(d_{\epsilon} -d)$ for some $\xi \in L^\infty(\Omega)$, with $\xi\ge 0$ a mean value representation (cf. \cite[Proposition 3.1]{DonHinPapVoe22}), and using standard estimates for elliptic PDEs (e.g., \cite[Chapter 6, Theorem 2]{EvansPDEs}) we have the conclusion.
\end{proof}

\begin{remark}\label{still_C_independent_of_y}
The estimate in Lemma \ref{lem:conv_dd} is  uniform for every element of the set $\Pi_{\epsilon}(u;h)$ which is potentially a non-singleton.
We also note again that the constants $C>0$ in \eqref{P_eps_S_prime_estimate} and \eqref{P_eps_S_prime_Ceps} can also be considered to be independent of $y$ and $u$. This follows from Remark \ref{C_independent_of_y} and the fact  that the $L^{\infty}(\om)$-norm of $\xi$ above can be upper bounded independently of $y$ (and $\epsilon$), making the constant $C>0$ in the first estimate  \eqref{P_eps_S_prime_estimate} independent on $y$ (and $\epsilon$).
\end{remark}

Lemma \ref{lem:conv_dd} indicates that $\Pi_\epsilon(u;h)\to S'(u;h)$ in $H^1(\Omega)$ as $\epsilon \to  0$. We note that in order to rigorously state this convergence we would need to define a selection function that chooses a solution of \eqref{eq:smo_reg_aux_pro} for every $\epsilon>0$. While this can be done using the axiom of choice, or at least the axiom of countable choice, for a sequence $\epsilon_{n}\to 0$, we will refrain from using it whenever possible and constrain ourselves to estimates of the type \eqref{P_eps_S_prime_Ceps}.

The next proposition shows that for sufficiently small $\epsilon>0$, we can indeed compute a descent direction by solving \eqref{eq:smo_reg_aux_pro} instead of  \eqref{eq:reg_aux_pro}. 

\begin{proposition}\label{prop:h_descent}
	Let $u\in \mathcal{C}_{ad}$ be a feasible point for the reduced problem of \eqref{P} which is  not $B$-stationary. Then there exists  $\epsilon^{\ast}>0$, such that for $0<\epsilon<\epsilon^{\ast}$ a solution $h_{\epsilon}$ of problem \eqref{eq:smo_reg_aux_pro} is a descent direction for the reduced objective $\mathcal{J}$ of \eqref{P} at $u$ (in particular $h_{\epsilon}\ne 0$). 
\end{proposition}
\begin{proof}
	Our goal is to show that 
	 there exist $\epsilon^{\ast}>0$ such that, for all $\epsilon<\epsilon^{\ast}$, if $h_{\epsilon}$ solves \eqref{eq:smo_reg_aux_pro}, then
	\[ \langle S(u)-g,S'(u;h_\epsilon)\rangle + \alpha \langle u,h_\epsilon\rangle < 0. \]
Observe first that from the fact that $S'(u;0)=0$ and 	from \eqref{P_eps_S_prime_Ceps}, we have that there exists a constant $C>0$ independent of $\epsilon$ such that for every $d_{\epsilon}$ solving \eqref{eq:smoothed_der} for $h=0$, we have $\|d_{\epsilon}\|_{H^{1}(\om)}\le C\epsilon$. 
It follows that if
	 $(h_\epsilon, d_{\epsilon})$ is a solution of \eqref{eq:smo_reg_aux_pro}, then we have 
	\begin{equation}\label{eq:aux1}
\frac{1}{2}q(h_\epsilon,h_\epsilon)+ \langle S(u)-g, d_{\epsilon})\rangle + \alpha \langle u,h_\epsilon\rangle \leq C \epsilon .
	\end{equation}
again for a constant independent of 
$\epsilon>0$.
	Based on \eqref{eq:aux1}, we have
	\begin{align}\label{eq:descent}
	\langle S(u)-g,S'(u;h_\epsilon)\rangle + \alpha \langle u,h_\epsilon\rangle &\leq C\epsilon - \frac{1}{2}q(h_\epsilon,h_\epsilon) +  \langle S(u)-g,S'(u;h_\epsilon)-d_{\epsilon}\rangle .
	\end{align}
Now in view of the estimate \eqref{P_eps_S_prime_Ceps}, we have for a generic constant $C>0$ still independent of $\epsilon>0$ and $u$
	\begin{align}\label{eq:descent2}
	\langle S(u)-g,S'(u;h_\epsilon)\rangle + \alpha \langle u,h_\epsilon\rangle &\leq C\epsilon - \frac{1}{2}q(h_\epsilon,h_\epsilon). 
	\end{align}
In order to finish the proof it suffices to show that there exists $\epsilon^{\ast}>0$ and $M>0$ such that for every $\epsilon<\epsilon^{\ast}$	
\begin{equation}\label{q_M}
M\le q(h_{\epsilon}, h_{\epsilon}),
\end{equation}
or in view of the coercivity estimate in \eqref{g_coercive_bounded}, it suffices to show
\begin{equation}\label{q_M_norm}
M\le \|h_{\epsilon}\|_{L^{2}(\om)}^{2}.
\end{equation}
Then by potentially reducing $\epsilon^{\ast}$ further, the results follows.
Suppose towards contradiction that \eqref{q_M_norm} does not hold. Then there exists a sequence $\epsilon_{n}\to 0$ such that $\|h_{\epsilon_{n}}\|_{L^{2}(\om)}\to 0$, which implies that $h_{\epsilon_{n}}\to 0$ in $L^{2}(\om)$. Then from Lemma \ref{lem:h_eps} below we deduce that $\bar{h}=0$ is a minimizer of \eqref{eq:reg_aux_pro} which is a contradiction since we have assumed that $u$ is not $B$-stationary. 
\end{proof}

\begin{lemma}\label{lem:h_eps}
Let $u\in \mathcal{C}_{ad}$, $\epsilon_{n}\to 0$  and let $h_{\epsilon_{n}}^{\ast}$ be a  minimizer for the problem \eqref{eq:smo_reg_aux_pro} for every $n\in\NN$. Then there exists a subsequence $(h_{\epsilon_{n_{k}}}^{\ast})_{k\in\NN}$  and  a minimizer $h^{\ast}$ of \eqref{eq:reg_aux_pro} such that $h_{\epsilon_{n_{k}}}^{\ast}\rightharpoonup h^{\ast}$ in $L^{2}(\om)$ as $k\to\infty$.
\end{lemma}
\begin{proof}
We first claim that the sequence $(h_{\epsilon_{n}}^{\ast})_{n\in\NN}$ is bounded in $L^{2}(\om)$. This can be seen for instance 
from \eqref{P_eps_S_prime_Ceps} and the fact that $\frac{1}{2}q(\cdot, \cdot)+ \langle S(u)-g,S'(u;\cdot)\rangle$ is coercive.
 It follows that  there exists a subsequence $(h_{\epsilon_{n_{k}}}^{\ast})_{k\in\NN}$  and  $h^{\ast}\in L^{2}(\om)$  such that $h_{\epsilon_{n_{k}}}^{\ast}\rightharpoonup h^{\ast}$ in $L^{2}(\om)$ as $k\to\infty$. From the estimate \eqref{d_eps_estimate} we can assume that $d_{\epsilon_{n_{k}}}^{\ast}\rightharpoonup d^{\ast}$ in $H_{0}^{1}(\om)$ for some $d^{\ast}\in H_{0}^{1}(\om)$ where $d_{\epsilon_{n_{k}}}^{\ast}$ satisfies \eqref{eq:smoothed_der} for $h_{\epsilon_{n_{k}}}^{\ast}$ as right-hand side, also assuming that it has been selected using the axiom of countable choice. Note that we can easily check that $d^{\ast}=S'(u;h^{\ast})$, i.e., the pair $(h^{\ast}, d^{\ast})$ satisfies the unregularized adjoint equation \eqref{adjoint}. Indeed, this follows from the fact that $-\Delta d_{\epsilon_{n_{k}}}^{\ast}\rightharpoonup - \Delta d^{\ast}$ and $h_{\epsilon_{n_{k}}}^{\ast} \rightharpoonup h^{\ast}$ in $H^{-1}(\om)$ and from the fact that $D_{\epsilon_{n_{k}}}(y;d_{\epsilon_{n_{k}}}^{\ast})\to N'(y;d^{\ast})$ in $L^{2}(\om)$. The last convergence can be inferred from the estimate
 \[\|D_{\epsilon_{n_{k}}}(y;d_{\epsilon_{n_{k}}}^{\ast}) -N'(y;d^{\ast}) \|_{L^{2}(\om)}
 \le  \| D_{\epsilon_{n_{k}}}(y;d_{\epsilon_{n_{k}}}^{\ast}) -N'(y;d_{\epsilon_{n_{k}}}^{\ast}) \|_{L^{2}(\om)}
 + \|N'(y;d^{\ast}) - N'(y;d_{\epsilon_{n_{k}}}^{\ast}) \|_{L^{2}(\om)}
 \]
 in combination with \eqref{N_prime_D_eps}, the Lipschitz continuity of $N'(y;\cdot)$ and the fact that $d_{\epsilon_{n_{k}}}^{\ast}\to d^{\ast}$ in $L^{2}(\om)$.
 
 Using the minimizing property of $h_{\epsilon_{n_{k}}}^{\ast}$ and letting $G: L^{2}(\om)\times H_{0}^{1}(\om)\to \RR$ with 
 $G(h,d)= \frac{1}{2}q(h,h)+ \langle S(u)-g,d\rangle + \alpha \langle u,h\rangle + \mathcal{X}_{T_{\mathcal{C}_{ad}(u)}}(h)$ we have that
 \begin{equation}\label{G_mini}
 G(h_{\epsilon_{n_{k}}}^{\ast}, d_{\epsilon_{n_{k}}}^{\ast}) \le G(h_{\epsilon_{n_{k}}}, d_{\epsilon_{n_{k}}}),
 \end{equation}
 for all pairs $(h_{\epsilon_{n_{k}}}, d_{\epsilon_{n_{k}}})$ that satisfy \eqref{eq:smoothed_der} for $\epsilon:=\epsilon_{n_{k}}$. We now claim that for every pair $(h,d)$ satisfying \eqref{adjoint} there exists a pair sequence $(\bar{h}_{\epsilon_{n_{k}}}, \bar{d}_{\epsilon_{n_{k}}})$  that satisfies \eqref{eq:smoothed_der} for $\epsilon:=\epsilon_{n_{k}}$ for each index $k$ such that $\bar{h}_{\epsilon_{n_{k}}}\to h$ in $L^{2}(\om)$ and $\bar{d}_{\epsilon_{n_{k}}}\rightharpoonup d$ in $H_{0}^{1}(\om)$. Indeed we can set $\bar{h}_{\epsilon_{n_{k}}}:=h$ for all $k\in\NN$, and choose  $\bar{d}_{\epsilon_{n_{k}}}\in H_{0}^{1}(\om)$ a solution of
\[-\Delta \bar{d}_{\epsilon_{n_{k}}} + D_{\epsilon_{n_{k}}}(y; \bar{d}_{\epsilon_{n_{k}}})=h.\] 
 Similarly as before we can check that  $\bar{d}_{\epsilon_{n_{k}}}\rightharpoonup d$ in $H_{0}^{1}(\om)$ where $d=S'(u;h)$. By employing the inequality \eqref{G_mini} and taking limits on both sides we have
 \begin{equation*}
 G(h^{\ast},d^{\ast})=\liminf_{k\to\infty}  G(h_{\epsilon_{n_{k}}}^{\ast}, d_{\epsilon_{n_{k}}}^{\ast})\le \lim_{k\to\infty}  G(\bar{h}_{\epsilon_{n_{k}}}, \bar{d}_{\epsilon_{n_{k}}})= G(h,d).
 \end{equation*}
 Since $(h,d)$ was a arbitrary pair  satisfying \eqref{adjoint}, the result follows.
\end{proof}

\begin{remark}\label{rem:M_depend_u}
We note that since the value of the constant $M>0$ in \eqref{q_M_norm} potentially depends on $u\in \mathcal{C}_{ad}$, it cannot be guaranteed that  $\epsilon^{\ast}>0$ can be chosen to have a common fixed value for all $u\in \mathcal{C}_{ad}$.
\end{remark}

Details on how we solve \eqref{eq:smo_reg_aux_pro} in practice are provided below in Section \ref{solving_subproblem}.
Once \eqref{eq:smo_reg_aux_pro} is solved, and a descent direction $h$ is identified, we perform an Armijo line search  in order to compute a step length that sufficiently decreases the reduced objective $\mathcal{J}$. For the sake of completeness we outline this in Algorithm \ref{alg:line_search}, which assumes that we have already computed  $u_{k}, u_{k}+h\in \mathcal{C}_{ad}$ at the $k$-th iteration of the main algorithm.
\begin{algorithm}[t]
\begin{flushleft}
Input: $h\in T_{\mathcal{C}_{ad}}(u_k)$, $\tau_0=\tau >0$, $c\in (0,1)$, $0<\eta\ll 1$, $\nu \in (0,1)$, $i=0$.\\
While
\begin{equation}\label{eq:ls_cost}
\mathcal{J}(u_k+\tau_i h)> \mathcal{J}(u_k)+ \nu \tau_i \mathcal{J}'(u_k; h) \quad \text{ and } \quad  \tau_{i} >\eta
\end{equation}\\
Set: $\tau_{i+1}= c\tau_i$, $i=i+1$\\
 end while
\end{flushleft}\caption{Armijo line search}\label{alg:line_search}
\end{algorithm}
Here $\eta>0$ is some   parameter that prevents the step size from becoming too small.

Note that the directional derivative of the reduced objective $\mathcal{J}'(u_k; h)$ in \eqref{B_stationarity_VI},  can be evaluated using standard adjoint calculus. The corresponding involved PDEs (the state equation in \eqref{P} and the adjoint equation \eqref{adjoint}) are solved numerically via a (semismooth) Newton algorithm. 

In practice,  the decrease of the step length $\tau$  in Algorithm \ref{alg:line_search} may be faster than the decrease of the magnitude of the descent direction $h$. This may result in an insufficient decrease of  the cost functional $\mathcal{J}$, particularly when the iterates approach some nonstationary point where the (reduced) objective is nonsmooth. In such a case we perform a \emph{robustification step} similar to \cite[Algorithm 4]{Hint_Suro}. That is, we resort to a smoothed optimal control problem 
in order  to compute a new control  $u_k$, and then compute a new descent direction based on this $u_{k}$.
 In particular, we solve the following problem
\begin{equation}\label{eq:smooth_cost}
\begin{aligned}
&\text{minimize }\quad \frac{1}{2} \|y-g\|_{L^{2}(\Omega)}^{2} + \frac{\alpha}{2} \|u\|_{L^{2}(\om)}^{2},\quad \text{ over } (y,u)\in H_{0}^{1}(\om) \times L^{2}(\om),\\
&\text{subject to }\left \{
\begin{aligned}
-\Delta y + \mathcal{N}_{\delta}(\cdot, y)&=u, \;\; \text{ in }\Omega\\
y&=0, \;\; \text{ on }\partial \Omega
\end{aligned}
\right., \quad \text{ and }\quad  u\in \mathcal{C}_{ad}.
\end{aligned}
\end{equation}
where $\mathcal{N}_\delta$ is a (canonically) smoothed version of the network $\mathcal{N}$. Note that since problem \eqref{eq:smooth_cost} is merely a helpful tool in the overall algorithm (in practice the robustification step is rarely activated - see next section), and not the final problem to be solved, the potential nonuniqueness of its solutions is not a point of concern. 
The numerical solver for this smooth problem can be found for instance in \cite{DonHinPap20}. After every robustification step, we decrease the parameter $\delta$ by a factor $\tilde{c}\in (0,1)$. 

 We state now in Algorithm \ref{alg:non_smooth_sqp} the overall descent algorithm which is based on the strategy of sequentially minimizing the cost function in \eqref{eq:smo_reg_aux_pro} in order to obtain descent directions.
 A few initial remarks on Algorithm \ref{alg:non_smooth_sqp} are in order. Note that if $u_{k}$ is not a $B$-stationary point, then the internal loop which is triggered in Step 3, in the case where $h_{k}$ is not a descent direction, is finite. This is indeed guaranteed in view of Proposition \ref{prop:h_descent}. The extra update $\epsilon\to c_{1}\epsilon$ in Step 4 after every successful Armijo line search, ensures that the parameter $\epsilon$ goes to zero along the iterations.

\begin{algorithm}[t]
\begin{flushleft}
		Input: $u_0\in \mathcal{C}_{ad}$, $\eta>0$, $\epsilon=\epsilon_0>0$, $\delta=\delta_0>0$, $1\geq \tau >\tau_{\min}>0$,  and  $c,\tilde{c},c_{1},c_{2}, \nu \in (0,1)$. \\[0.5em]
		
		Obtain $y_{0}=S(u_{0})$ by solving the state equation in \eqref{P} using a semismooth Newton method.
		
		Perform the following iteration for $k=0,1,2,\ldots$:\\
		\begin{itemize}
			\item[Step 1:] 
			Solve problem \eqref{eq:lin_reg_aux_pro} at $u=u_k$ in order to get initial values for $h_k$ 			 (details in Section \ref{subsec:algo}). 
			
			If $\mathfrak{m}(\Omega_{\mathcal{N}}(u_k))= 0$, go directly to Step 4.\\
			Else go to Step 2.
			
			\item[Step 2:] Solve the subproblem \eqref{eq:smo_reg_aux_pro}
			 (details  in Sections \ref{solving_subproblem} and \ref{subsec:algo})
			   and update $h_k$ accordingly.

			\item[Step 3:] Check if $h_{k}$ is a descent direction, i.e., whether 
			\[\langle S(u_{k})-g,S'(u_{k};h_k)\rangle + \alpha \langle u,h_k\rangle < 0.\]
			If this is not satisfied, update  $\epsilon\rightarrow c_{1}\epsilon$, and return to Step 2.
			\item[Step 4:] Perform the Armijo line search in Algorithm \ref{alg:line_search} with parameters $c$, $\nu$ and $\eta:=\min\set{\tau_{\min}, \tilde{c} \norm{h_k}_{L^{2}(\om)}}$ to obtain a step length $\tau\in (0,1]$, and then update $\epsilon\to c_{1}\epsilon$.	
			If $\tau<\eta$,  stop the line search, perform the robustification step by solving \eqref{eq:smooth_cost} to obtain a new $u_k$,  update $\delta\rightarrow c_{2} \delta$, let $k=k+1$ and return to Step 1.
			
			\item[Step 5:] Set $u_{k+1}=u_k+\tau h_k$, and compute $S(u_{k+1})$ using again a semismooth Newton method by solving the state equation in \eqref{P}. Let $k=k+1$. 
		\end{itemize}
		
		\caption{Proposed algorithm for optimal control of ReLU-network-informed PDEs}
		\label{alg:non_smooth_sqp}\end{flushleft}
\end{algorithm}

 We mention already here that in order to get an initial value for $h_{k}$ in Step 1, which is used as initialization for Step 2, we solve the following problem  
 \begin{equation}\label{eq:lin_reg_aux_pro}
 \operatorname{minimize} \quad \frac{1}{2}q(h,h)+ \langle S(u)-g,\Pi_{0}(u;h) \rangle + \alpha \langle u,h \rangle \quad \text{ over }  h\in T_{\mathcal{C}_{ad}}.
 \end{equation}
Here, for $h\in L^2(\Omega)$, $\Pi_{0}(u;h)\in H_0^1(\Omega)$ denotes a solution of the following linear equation:
 \begin{equation}\label{eq:app_pde}
\left \{
\begin{aligned}
-\Delta d + D_0(y)d&=h , \;\; \text{ in }\Omega,\\
d &=0, \;\; \text{ on }\partial \Omega,
\end{aligned}
\right. 
\end{equation}
 where $D_0$ is the function that results  by formally setting the derivatives of the ReLU functions at zero to be zero, recall formula \eqref{relu_prime}.
Note that in the case $\mathfrak{m}(\Omega_{\mathcal{N}}(u))= 0$,  \eqref{eq:lin_reg_aux_pro} is equivalent to  \eqref{eq:reg_aux_pro}. We point the reader to Remark \ref{remark_D0} below regarding potential (but rare) complications which might be caused by $D_0$ in \eqref{eq:app_pde}.
We also note that if $q(h,h):=\norm{h}^2_{L^2(\Omega)}$, then Algorithm \ref{alg:non_smooth_sqp} will perform exactly like a (sub-) gradient descent method, which can be slow in terms of convergence rates.
In order to accelerate the algorithm, in the numerical examples
  we use the quadratic functional
  \begin{equation}\label{quadratic_functional}
  q(h,h)= \langle \Pi_{0}(u;h) ,\Pi_{0}(u; h)\rangle +\alpha \langle h,h \rangle,
  \end{equation}
  and we denote its derivative at $h$ by $Qh$. 
We also note  that if the network function $\mathcal{N}(x,y)$ is smooth with respect to $y$, then the proposed algorithm with the above quadratic functional  is an SQP (Sequential Quadratic Programming) type method. 
\begin{remark}\label{remark_D0}
As it was pointed out in \cite{relu_chainrule}, even though $D_{0}$ is almost equal to the gradient of $\mathcal{N}$ and in particular it is an almost everywhere positive function, its values at the nondifferentiability points of $\mathcal{N}$ could lie strictly below the Clark subdifferential of $\mathcal{N}$  at these points. For example, if $\mathcal{N}: \RR \to \RR$, then for every $y\in\RR$, it holds that $\partial \mathcal{N}(y)=[\underline{\partial}\mathcal{N}(y),  \overline{\partial}\mathcal{N}(y)]$, where $\underline{\partial}\mathcal{N}:=\min\{\mathcal{N}'_{-}, \mathcal{N}'_{+}\}$, $\overline{\partial}\mathcal{N}:=\max\{\mathcal{N}'_{-}, \mathcal{N}'_{+}\}$ with $\mathcal{N}'_{\pm}$ denoting the left- and right-sided derivatives. While due to $\mathcal{N}$ being increasing we have $\underline{\partial}\mathcal{N}(y)>0$ for every $y\in\RR$, it could be the case that for some $y_{0}\in\RR$ it holds $D_{0}(y_{0})<0<\underline{\partial}\mathcal{N}(y_0)$ and as a result if the function $y$ in \eqref{eq:app_pde} attains the value $y_{0}$ at a set of positive measure the existence  of that equation could be at stake. Since however \eqref{eq:app_pde} is only used to get some initial values for $h_{k}$, in practice, we can  restrict ourselves to a nonnegative approximation by setting the negative values of $D_{0}$ to zero.
\end{remark}

\subsection{Solving problem \eqref{eq:smo_reg_aux_pro}}\label{solving_subproblem}
We continue by providing some details on solving \eqref{eq:smo_reg_aux_pro} in Step 2 of Algorithm \ref{alg:non_smooth_sqp}. From now on we assume that the state equation in \eqref{eq:smo_reg_aux_pro} admits a unique solution $d_\epsilon=d_\epsilon(h_\epsilon)$. This is indeed the case when, e.g., for sufficiently small $\rho>0$, $\langle D_\epsilon(y;d),d\rangle\geq -\rho\|d\|_{L^2(\Omega)}^2$ for all $d\in H_0^1(\Omega)$. Then the first-order optimality condition for an optimal $h_\epsilon$ reads
\begin{equation}\label{eq:var_ineq}
\langle Q h_\epsilon + p_\epsilon(u;h_\epsilon)+\alpha u, h\rangle \geq 0 \quad \text{ for all }  h\in T_{\mathcal{C}_{ad}}(u).
\end{equation}
Note that in \eqref{eq:var_ineq}, $p_\epsilon(u;h_\epsilon)$ represents the directional derivative of the second term in the objective of \eqref{eq:smo_reg_aux_pro}.  

Since we have assumed box constraints on the control variable (compare \eqref{Cad_box}), and in view of \cite[Theorem 2.29]{Tro10}, formally  the above variational inequality can be equivalently characterized by a system of equations as follows:
\begin{equation}\label{eq:KKT_sub}
\left\{
\begin{aligned}
&-\Delta d_\epsilon + D_\epsilon(y;d_\epsilon)-h_\epsilon=0\; \text{ in } \Omega,\;\text{ and}\; d_\epsilon =0 \text{ on } \partial \Omega,\\
&-\Delta p_\epsilon + \partial_{d}(D_\epsilon(y;d_\epsilon)) p_\epsilon=y-g\; \text{ in } \Omega,\;\text{ and}\; p_\epsilon =0 \text{ on } \partial \Omega,\\
&Q h_\epsilon+ p_\epsilon  +\mu=-\alpha u,\\
&\mu -\max(0,\mu+\lambda(u+h_\epsilon-u_b))-\min(0,\mu+\lambda(u+h_\epsilon-u_a))=0,
\end{aligned}\right.
\end{equation}
where $\lambda>0$ is a constant which is typically set equal to the cost of the control, i.e., $\lambda=\alpha$.
 The first equation in  \eqref{eq:KKT_sub} is simply \eqref{eq:smoothed_der}, 
  while the second one is the adjoint equation that provides a way to calculate the directional derivative in \eqref{eq:var_ineq}.   The third equation represents the first-order stationarity condition of \eqref{eq:smo_reg_aux_pro} with $\mu$ being a slack variable, while the fourth one is used to enforce the box constraint $u+h_\epsilon\in \mathcal{C}_{ad}$, complementarity $\mu(u+h_\epsilon-u_a)(u+h_\epsilon-u_b)=0$ a.e. in $\Omega$, as well as $\mu\geq 0$ a.e. on $\{u+h_\epsilon=u_b\}$ and $\mu\leq 0$ a.e on $\{u+h_\epsilon=u_a\}$. Under suitable assumptions, the nonlinear and nonsmooth system \eqref{eq:KKT_sub} can be solved efficiently via  a primal-dual active-set algorithm (PDAS) for which we provide the details in Section \ref{subsec:algo} below.
  
The subtle point of the system \eqref{eq:KKT_sub} is that additional conditions are required for the existence of solutions for the second equation since the $L^{\infty}$-function $\partial_{d}(D_\epsilon(y;d_\epsilon))$ might be negative on a large set and hence the corresponding PDE operator would not be coercive. This is due to the potential nonmonotonicity of $D_{\epsilon}(y;\cdot)$. Below we provide a sufficient condition which guarantees existence of solutions and, as a consequence the constraint qualification of \cite{Zowe1979} is satisfied. Thus, \eqref{eq:KKT_sub} indeed represents the Karush-Kuhn-Tucker (KKT) system for \eqref{eq:smo_reg_aux_pro}. Note that as we show in Proposition \ref{pro:algo_convergence}, essentially unless an early stopping occurs, it holds that $\|h_{\epsilon_{k}}\|_{L^{2}(\om)}\to 0$ along the iterations $k$ of Algorithm \ref{alg:non_smooth_sqp}. In view of \eqref{eq:smoothed_der}, this implies that $\|d_{\epsilon_{k}}\|_{H_{0}^{1}(\om)}\to 0$ as well. The condition below leverages  this fact.
  
  \begin{lemma}\label{lem:existence_second}
Let $y\in Y\cap C^{0,a}(\overline{\om})$ be a solution of  the state equation in \eqref{P} and let $d_{\epsilon}\in H_{0}^{1}(\om)$ be a solution of the first equation in \eqref{eq:KKT_sub} such that the estimate $\|d_{\epsilon}\|_{H_{0}^{1}(\om)}\to 0$ as $\epsilon\to 0$. 
Suppose that there exists $\delta>0$, possibly dependent on $y$, such that for small enough $\epsilon>0$, the set
\begin{equation}\label{badset}
U=\{x\in\om: \mathcal{D}_{\epsilon}(y(x); \cdot) \text{ is monotone increasing in }(-\delta, \delta)\}
\end{equation}
has a full Lebesgue measure. Then for small enough $\epsilon>0$ the second equation in \eqref{eq:KKT_sub} has a solution $p_{\epsilon}\in H_{0}^{1}(\om)$.
\end{lemma}
\begin{proof}
It suffices to show that   
\begin{equation}\label{Linfty_eps}
 \mathfrak{m} (\{x\in\om: \partial_{d} D_{\epsilon}(y(x);d_{\epsilon}(x))<0\} )\to 0 \quad \text{as }\epsilon \to 0.
\end{equation}
Indeed, if this holds then the operator $\mathcal{A}_{\epsilon}:H_{0}^{1}(\om)\to H^{-1}(\om)$, where for $d,v\in H_{0}^{1}(\om)$, $\mathcal{A}_{\epsilon}d(v):= \langle \nabla d, \nabla v \rangle +  \langle \partial_{d} D_{\epsilon}(y;d_{\epsilon})d, v \rangle$ is coercive for small enough $\epsilon>0$ and we can proceed as in Proposition \ref{prop:existence_Deps}.

Since we have $\|d_{\epsilon}\|_{H_{0}^{1}(\om)}\to 0$ as $\epsilon\to 0$, using the Chebyshev inequality, it follows that 
\begin{equation}\label{leb_to_zero}
\mathfrak{m}\left (\{x\in\om:\,|d_{\epsilon}(x)|\ge \delta\} \right )\to 0 \quad \text{as }\epsilon \to 0.
\end{equation}
We then have the partition
\begin{align*}
(\{x\in\om: \partial_{d} D_{\epsilon}(y(x);d_{\epsilon}(x))<0\}
&= 
(\{x\in\om: \partial_{d} D_{\epsilon}(y(x);d_{\epsilon}(x))<0\}\cap\{x\in\om:\,|d_{\epsilon}(x)|< \delta\}\\ 
&\quad\cup(\{x\in\om: \partial_{d} D_{\epsilon}(y(x);d_{\epsilon}(x))<0\}\cap\{x\in\om:\,|d_{\epsilon}(x)|\ge \delta\}. 
\end{align*}
The first set in the partition above is a subset of $\om\setminus U$ so it has zero Lebesgue measure, while the measure of the second set goes to zero as $\epsilon\to 0$ in view of \eqref{leb_to_zero}. That shows \eqref{Linfty_eps}.
\end{proof}  
  
Lemma  \ref{lem:existence_second} indicates that, in order to have existence of solutions for the second equation in \eqref{eq:KKT_sub},  it suffices to impose some condition that guarantees that the smoothed function $\mathcal{D}_{\epsilon}(y(x);\cdot)$ will not be decreasing in an area around zero; in a large set or in set of full measure as it is done here. This is the main region of interest since $\|d_{\epsilon_{k}}\|_{L^{2}(\om)}\to 0$ and thus all its values will be essentially concentrated around that area.

\subsection{Convergence analysis}

In this section, we provide information about  the quality of the limits of the sequence of controls $(u_{k})_{k\in \NN}$ and pertinent states $(y_{k})_{k\in\NN}$ generated by Algorithm \ref{alg:non_smooth_sqp}. We start with a result regarding the convergence of the sequence of descent directions $(h_{k})_{k\in\NN}$.

\begin{proposition}\label{pro:algo_convergence}
Let $(u_k)_{k\in\NN}$ be a sequence of controls generated by Algorithm \ref{alg:non_smooth_sqp}. 
 If for every $k\in\NN$, $u_{k}$ is not a $B$-stationary point and the robustification step is activated only finitely many times,  then $\|h_k\|_{L^{2}(\om)}\to 0$ as $k\to \infty$.

\end{proposition}
\begin{proof}
Note that since the robustification step is activated finitely many times only, we have that the sequence $(\mathcal{J}(u_k))_{k\in\NN}$ is eventually strictly decreasing. Since all its elements are positive, it follows that there exists $\mathcal{J}^*\geq 0$ such that $\mathcal{J}(u_k)\to \mathcal{J}^*$. Assume without loss of generality, that for all but finitely many iterates we have $\mathfrak{m}(\om_{\mathcal{N}}(u_{k}))>0$.
Note that from the Armijo line search we  have for large enough $k$, and a constant $C>0$ independent of $\epsilon$ and $u_{k}$ (see also \eqref{eq:descent2})
\begin{align*}
\mathcal{J}(u_{k+1})-\mathcal{J}(u_k)
&\le \nu \tau \mathcal{J}'(u_{k};h_{k})
\le \nu \tau C\epsilon-\frac{\nu\tau}{2}q(h_{k}, h_{k})\\
&\le \nu C\epsilon -\frac{\nu C_{2}}{2}\min(\tau_{min} \|h_{k}\|_{L^{2}(\om)}^{2},  \tilde{c}\|h_{k}\|_{L^{2}(\om)}^{3}),
\end{align*}
where we used the fact that $1>\tau>\min(\tau_{min}, \tilde{c}\|h_{k}\|_{L^{2}(\om)})$ and the  estimate \eqref{g_coercive_bounded}.
Since $0>\mathcal{J}(u_{k+1})-\mathcal{J}(u_k)\to 0$ as $k\to \infty$ and the fact that $\epsilon$ is also going to zero along the iterates, see the remarks after Algorithm \ref{alg:non_smooth_sqp}, it follows that  $\norm{h_k}_{L^2(\Omega)}\to 0$. Lastly if $\mathfrak{m}(\om_{\mathcal{N}}(u_{k}))=0$ for infinitely many $k$'s then, along that subsequence, still denoted by $(u_k)_{k\in \N}$, we have 
	\begin{align*}
	\mathcal{J}(u_{k+1})-\mathcal{J}(u_k)
	\le \nu \tau \mathcal{J}'(u_{k};h_{k})
	\le -\frac{\nu\tau}{2}q(h_{k}, h_{k})
	\le -\frac{\nu C_{2}}{2}\min(\tau_{min} \|h_{k}\|_{L^{2}(\om)}^{2},  \tilde{c}\|h_{k}\|_{L^{2}(\om)}^{3}),
	\end{align*}
see \eqref{eq:lin_reg_aux_pro} and \eqref{eq:app_pde}. This concludes the proof.
\end{proof}

The next theorem provides more details about the iterates of Algorithm \ref{alg:non_smooth_sqp}. In fact, depending on properties with respect to robustification and the nonsmooth behavior of $\mathcal{N}$, along specific subsequences limit points satisfying different types of stationarity are obtained, respectively.
\begin{theorem}\label{thm:algo_convergence}
Let $\lambda =\alpha$ in the KKT system \eqref{eq:KKT_sub}.
Let $(u_k)_{k\in\NN}$ be a sequence of controls generated by Algorithm \ref{alg:non_smooth_sqp}, with $(y_k)_{k\in \NN}$ the corresponding states. 
	Then the following hold true:
\begin{itemize}
\item[(1)] Suppose the algorithm returns $h_{k_0}=0$ after finitely many iterations, and $u_{k_0}$ and $y_{k_0}$ are the corresponding control and state, respectively. If $\mathfrak{m}(\Omega_{\mathcal{N}}(u_{k_0}))=0$, then the algorithm returns a B-stationary point; otherwise the following conditions are satisfied:
\begin{equation}\label{eq:Finite_stationary}
\begin{aligned}
- \Delta y_{k_0} +\mathcal{N}(\cdot,y_{k_0})  -u_{k_0}=0\;  & \text{ in } \Omega ,\quad y_{k_0}=0\;   \text{ on } \partial \Omega ,\\
-  \Delta p_{k_0} +   \chi_\epsilon p_{k_0} - y_{k_0}= -g\;  & \text{ in } \Omega ,\quad
p_{k_0} =0\;   \text{ on } \partial \Omega,\\
(p_{k_0}+  \alpha u_{k_0}, h)\geq 0\; & \text{ for all } \quad h\in T_{\mathcal{C}_{ad}}(u_{k_0}).
\end{aligned}  
\end{equation}
where $\chi_\epsilon  = \partial_{d} D_\epsilon(y_{k_0};d_\epsilon)$, and $d_\epsilon$ solves the PDE \[-\Delta d + D_\epsilon(y_{k_0};d)=0\; \text{ in } \Omega,\;\; d =0 \text{ on } \partial \Omega.\]
\item[(2)] When the robustification step is activated only finitely many times, the following two cases need to be distinguished:
  \begin{itemize}
  \item[(i)] Along a subsequence where $\mathfrak{m}(\Omega_{\mathcal{N}}(u_{k_l}))=0$ for all $l\in\mathbb{N}$, there exists a further subsequence still denoted by $(u_{k_l},y_{k_l})$, so that $u_{k_l}\to  u^*$ in $L^2(\Omega)$, and $u^*\in\mathcal{C}_{ad}$ satisfies 
  \begin{equation}\label{eq:smooth_stationary}
 \mathcal{J}^\circ (u^*;h)\geq 0 \;\text{   for all } \; h \in T_{\mathcal{C}_{ad}}(u^*),
  \end{equation}
  where $ \mathcal{J}^\circ(u^*;h)$ is the Clarke directional derivative of $\mathcal{J}(\cdot)$ at $u^*$ in the direction $h$, i.e. $\mathcal{J}^\circ (u^*;h)=\sup_{\chi\in \partial \mathcal{J}(u^*) }\langle\chi,h\rangle$.
  \item[(ii)] Along a subsequence where $\mathfrak{m}(\Omega_{\mathcal{N}}(u_{k_l}))>0$ for all $l\in\mathbb{N}$, there exists a further subsequence still denoted by  $(u_{k_l},y_{k_l})$, so that  $u_{k_l}\to  u^*$ in $L^2(\Omega)$, and $u^*\in\mathcal{C}_{ad}$ satisfies the weak stationarity condition.
\end{itemize}
\item[(3)] When the robustification step is activated for infinitely many times, then there exists a subsequence so that the algorithm converges to a C-stationary point along that subsequence.
\end{itemize}
\end{theorem}
\begin{proof}
	We prove each of the statement here.
	\begin{itemize}[leftmargin=18pt]
\item[(1)] The first statement on the smooth case is due to the setting of the algorithm. In fact, when $\mathfrak{m}(\Omega_{\mathcal{N}}(u_{k_0}))= 0$, then we have $\Pi_0(u_{k_0};h)=S'(u_{k_0};h)$. When $0$ is a minimizer of \eqref{eq:aux_pro}, then for every $h \in T_{\mathcal{C}_{ad}}(u_{k_0})$, $\mathcal{J}'(u_{k_0};h)\geq \mathcal{J}'(u_{k_0};0)= 0$ due to the property of the minimizer.
When the nonsmooth part has positive measure, $S'(u_{k_0};h)$ is replaced by the smooth approximation $\Pi_\epsilon(u_{k_0};h)$ for some fixed $\epsilon>0$. The conclusion is drawn by rewriting the KKT system in \eqref{eq:KKT_sub} where the equivalence between the third variational inequality in \eqref{eq:Finite_stationary} and the third and the fourth equations in \eqref{eq:KKT_sub} is considered, a proof of which can be found in \cite[Theorem 2.29]{Tro10}.
\item[(2)]  We turn to the first assertion in the second statement. Notice that for a bounded sequence $(u_k)_{k\in \N}\subset \mathcal{C}_{ad}\subset L^2(\Omega)$, we can extract a weakly convergent subsequence denoted by $(u_{k_l})$, and $u_{k_l} \rightharpoonup u^*\in \mathcal{C}_{ad}$.
Let $y_{k_l}$, $p_{k_l}$ be the solutions of the state equation and the adjoint equation corresponding to $u_{k_l}$, respectively, and $y^*$, $p^*$ be the solutions corresponding to $u^*$.
Using standard regularity results on solutions of elliptical PDEs, we have $y_{k_l}\in H^1_0(\Omega)$ and $p_{k_l}\in H^1_0(\Omega)$ for all $l\in\mathbb{N}$, and $(y_{k_l})$ and $(p_{k_l})$ are uniformly bounded in $H^1_0(\Omega)$, respectively.
Using the compact embedding of $H^1_0(\Omega)$ into $L^2(\Omega)$, we conclude that $y_{k_l}\to y^*$ and $p_{k_l}\to p^*$ both in the $L^2(\Omega)$ norm topology.
Referring to the fourth equation in the KKT system \eqref{eq:KKT_sub} for each $u_{k_l}$ in the subsequence, we derive also that $u_{k_l}\to u^*$ strongly in $L^2(\Omega)$ if we choose $\lambda=\alpha>0$. This is because of $Qh_{k_l}\to 0$ (as well as $h_{k_l}\to 0$)  and the relation $Qh_{k_l}+ p_{k_l}= -\mu_{k_l}-\alpha u_{k_l}\to p^*$ in $L^2(\Omega)$, and the connection given by the fourth equation in \eqref{eq:KKT_sub} when $\lambda =\alpha$, i.e.,
\[ \mu_{k_l}=\max(0,\mu_{k_l} +\alpha(u_{k_l}+h_{k_l}-u_b)+\min(0,\mu_{k_l} +\alpha(u_{k_l}+h_{k_l}-u_a)))
\] 
which ensures that $\mu_{k_l}\to \mu^*$ in $L^2(\Omega)$, and subsequently $u_{k_l}\to u^*$ in $L^2(\Omega)$.
Since $y_{k_l}\in L^\infty(\Omega)$ and $N(\cdot)$ is Lipschitz,  $N'(y_{k_l})\in L^\infty(\Omega)$ are uniformly bounded. Using the Banach-Alaoglu theorem, we have $N'(y_{k_l}) \overset{*}{\rightharpoonup}\zeta $ for some $\zeta\in L^\infty(\Omega)$.
Using the definition of the Clarke subgradient, we have $\zeta\in \partial N(y^*)$ by upper semicontinuity of $\partial N(\cdot)$; see, e.g., \cite{aubin2009set}.

Using the above convergence properties, we arrive at the system 
\begin{equation}\label{eq:smooth_case}
\begin{aligned}
- \Delta y^* +\mathcal{N}(\cdot,y^*)  -u^*=0\;  & \text{ in } \Omega ,\quad y^*=0\;   \text{ on } \partial \Omega ,\\
-  \Delta p^* +   \zeta p^* - y^*= -g\;  & \text{ in } \Omega ,\quad
p^* =0\;   \text{ on } \partial \Omega,\\
\langle p^*+  \alpha u^*, h\rangle\geq 0\; & \text{ for all } \quad h\in T_{\mathcal{C}_{ad}}(u^*),
\end{aligned}  
\end{equation}
where the same argument for the third variational inequality holds as in Case $(1)$. 
By the definition of Clarke's generalized directional derivative, we then conclude that 
\[ \mathcal{J}^\circ (u^*;h)\geq \langle p^*+  \alpha u^*, h\rangle\geq 0 \;\text{   for all } \; h \in T_{\mathcal{C}_{ad}}(u^*). \]
%\\

For assertion $(ii)$, we use the same argument as in $(i)$ to have  $y_{k_l}\to  y^*$, $p_{k_l}\to  p^*$, and $u_{k_l}\to  u^*$ in $L^2(\Omega)$. Recall that $h_k\to 0$ in $L^2(\Omega)$ in the KKT system in \eqref{eq:KKT_sub}. Now we show that $(\partial_{d} D_{\epsilon_k}(y_k;d_{\epsilon_k}))_{k\in\mathbb{N}}$ is a bounded sequence in $L^\infty(\Omega)$. Note that $N(\cdot)$ is Lipschitz continuous and $y_k\in H^1(\Omega)\cap L^\infty(\Omega)$, and $D_{\epsilon_k}(y_k;\cdot)$ is $C^1$ smooth and therefore Lipschitz with respect to the second variable, from which we have that $\partial_d D_{\epsilon_k}(\cdot,\cdot)$ is uniformly bounded with respect to both variables, i.e., for all $k\in\mathbb{N}$ we have $\abs{\partial_d D_{\epsilon_k}(y_k,d_{\epsilon_k})}\leq M$. Thus, we have $\partial_d D_{\epsilon_k}(y_k,d_{\epsilon_k})\in L^\infty (\Omega)$ for all $k\in\mathbb{N}$. Now using the Banach-Alaoglu theorem, we conclude that there exists a weakly star convergent sub-sequence of $\partial_d D_{\epsilon_{k_l}}(y_{k_l},d_{\epsilon_{k_l}})$, i.e., there exists $\zeta\in L^\infty(\Omega)$ such that $\partial_d D_{\epsilon_{k_l}}(y_{k_l},d_{\epsilon_{k_l}}) \overset{*}{\rightharpoonup} \zeta$  (still denoted using the same indices).
Passing to the limit in the system \eqref{eq:KKT_sub} with respect to this subsequence, yields the conclusion.

\item[(3)]  For the third statement, we take the subsequence whose elements correspond to the control and state variables for activated robustification. This results in a sequence of optimal control problems with respect to the regularized PDEs in \eqref{eq:smooth_cost}.
Since in the $l^{th}$ robustification step, $\delta_{l+1}=\tilde{c}\delta_l$  for some $\tilde{c}\in (0,1)$ we infer $\delta_{l} \to 0$ as $l\to \infty$. This yields a C-stationary point in the limit as $l\to\infty$. For the associated analytical details on the convergence of the smoothed optimal control problems as $\delta_l\to 0$, we refer to the paper \cite{DonHinPapVoe22}.
	\end{itemize}
\end{proof}

We note that Case (1) of Theorem \ref{thm:algo_convergence} rarely occurs in practice and yields a desirable $B$-stationary point if $\mathcal{N}$ is differentiable at $u_{k_0}$; otherwise an approximate version of a $C$-stationary point is reached. Case (2) either yields a form of $C$-stationary point in (i), or an element satisying weaker conditions in (ii). The latter case produces the least favorable limit point in terms of stationarity. Finally, Case (3) provides a point satisfying $C$-stationarity conditions, which are weaker than $B$-stationarity conditions.

\subsection{Practical aspects concerning Algorithm \ref{alg:non_smooth_sqp}} We recall that for the sake of presentation we confine ourselves to the case where $\mathcal{C}_{ad}$ is given by box constraints; see \eqref{Cad_box}. We point out that such box constraints are relevant in numerous applications in PDE constrained optimization.

In order to account for possible violations of the control constraints in the practical numerical realization (e.g.\ due to inexact solves), we use the following merit function for the line search algorithm
\begin{equation}\label{eq:merit}
E_k(\tau):=\mathcal{J}(u_k+\tau h)+\kappa \Psi(u_k + \tau h),
\end{equation}
where 
\[ \Psi(u):= \norm{\max(0,u-u_b)}_{L^2(\Omega)}+\norm{\min(0,u-u_a)}_{L^2(\Omega)},\]
evaluates the violation of the box constraint.
Here $\kappa>0$ is the parameter from Algorithm \ref{alg:non_smooth_sqp}. The above merit function replaces the objective function $\mathcal{J}$ in \eqref{eq:ls_cost} in Algorithm \ref{alg:line_search}. Thus, we need to guarantee a descent direction for \eqref{eq:merit}. In our setting the latter is connected to a practical stopping rule for terminating the utilized solver for \eqref{eq:KKT_sub}. Notice that if $u_k\in \mathcal{C}_{ad}$, and the subproblem \eqref{eq:smo_reg_aux_pro} in particular the constraint has been settled with satisfactory accuracy, we shall have $u_k+h\in \mathcal{C}_{ad}$ as well. Then $(\Psi(u_k + h)- \Psi(u_k))=0$ and the standard Armijo line search is applied. This is often the case when a primal-dual active-set (PDAS) method (see, e.g., \cite{HIK}) is applied to box constraints as we will explain in detail in the next section. In case one aims at only approximately satisfying the constraint along the iterates, i.e. $(\Psi(u_k + h)- \Psi(u_k))>0$, a similar termination condition for the solver of the sub-problem \eqref{eq:smo_reg_aux_pro} as in \cite[Algorithm 1, (4.62)]{DonHinPap20} can be applied. It consists of the following inequalities:
\begin{equation}\label{eq:stopping_rule}
	\begin{aligned}
	& \mathcal{J}^\prime(u_k; h) + \kappa (\Psi(u_k + h)-\Psi(u_k))
	\leq - \xi  q(h,h), \\
	&\text{and}\;\quad \Psi(u_k + h) \leq (1- \xi)\Psi(u_k), 
	\end{aligned}\quad	\text{  for some } \quad \xi\in (0,1).
\end{equation}
The first inequality above guarantees a descent direction for the merit functional in every iteration. Whereas the second condition enforces uniform decay of the constraint violation along the iterations. Observe that if $\Psi(u_k)=0$, then $\Psi(u_l)=0$ for all $l>k$. The underlying assumption here is that the solver for \eqref{eq:KKT_sub} is able to achieve sufficiently accurate solutions.
	
Notice also that in Algorithm \ref{alg:non_smooth_sqp}, we require $\tau_0<\tau\leq 1$,
and observe further that by solving the system \eqref{eq:KKT_sub} exactly we obtain a direction $h$ with $u_k+h\in \mathcal{C}_{ad}$. Hence if $u_{k}\in \mathcal{C}_{ad}$ and the solution for \eqref{eq:KKT_sub} is accurate, then
this implies that  $u_k+\tau h\in \mathcal{C}_{ad}$ for all $\tau\in (0,1]$ by convexity of $\mathcal{C}_{ad}$. Consequently all the iterates are feasible, and the merit functional \eqref{eq:merit} is equivalent to the reduced functional provided that all the systems are solved exactly. Indeed, in our experiments, we use the PDAS algorithm which can compute highly accurate solutions for \eqref{eq:KKT_sub}.

\subsection{Details on the PDAS Algorithm}
\label{subsec:algo}

In the following, we provide some details on the implementation of PDAS in Algorithm \ref{alg:non_smooth_sqp}, as it is employed in two different steps. First we utilize PDAS to solve the KKT system of \eqref{eq:lin_reg_aux_pro} for initialization, which is:
\begin{equation}\label{eq:linear_kkt}
\begin{aligned}
(K^{-1}+\alpha \text{Id})h  +p_0 +\alpha u +\mu =0,\\
\mu -\max(0,\mu+\lambda(u+h-u_b))-\min(0,\mu+\lambda(u+h-u_a))=0,
\end{aligned}
\end{equation}
with $\lambda>0$ fixed. In practice we typically set $\lambda:=\alpha>0$. Here $Qh:=(K^{-1}+\alpha \text{Id})h$, with $Q:L^2(\Omega)\to L^2(\Omega)$ linear and continuous, i.e., $Q\in\mathcal{L}(L^2(\Omega))$, is the derivative of $\frac{1}{2}q(h,h)$ and  $K^{-1}\in \mathcal{L}(L^{2}(\Omega),B)$ for some $B\subset L^2(\Omega)$ with
\[
K\in\mathcal{L}(B,L^{2}(\Omega)),\quad K t=(-\Delta + (D_0(y))^*(-\Delta + (D_0(y))t\text{ for } t\in B.
\]
In fact, for given $h\in L^2(\Omega)$ and sufficiently smooth $\Omega$, $Kt=h$ is realized via finding $(s,t)\in(H_0^1(\Omega)\cap H^2(\Omega))^2$ such that
$$
-\Delta s + D_0(y)s=h\quad\text{and}\quad -\Delta t + D_0(y)t=s.
$$
Thus, $B=H_0^1(\Omega)\cap H^2(\Omega)$.

Computationally, \eqref{eq:linear_kkt} is realized as follows: we first introduce an auxiliary variable $t\in B$ and $K^{-1}h=t$. Then, in every iteration of PDAS, we solve the linear system \eqref{eq:complimentary} below. For this purpose let $\mathcal{A}_+$ denote an estimate for the upper active set $\{x\in\Omega:u(x)+h^*(x)=u_b(x)\}$, or short $\{u+h^*=u_b\}$, at the solution $h^*\in L^2(\Omega)$ of \eqref{eq:lin_reg_aux_pro} and analogously for $\mathcal{A}_-$ and the lower active set $\{u+h^*=u_a\}$, with $\mathcal{A}_+\cap\mathcal{A}_-=\emptyset$. Further, $\mathcal{I}=\Omega\setminus \mathcal{A}$, with $\mathcal{A}:=\mathcal{A}_+\cup\mathcal{A}_-$, is an estimate of the inactive set $\{u_a<u+h^*<u_b\}$.

The resulting linear system reads
\begin{equation}\label{eq:complimentary}
\begin{aligned}
Kt-h=0,\\	
t+\alpha h +\mathcal{J}_0'(u) + \mu=0,\\
h|_{\mathcal{A}_+}=u_b|_{\mathcal{A}_+}-u|_{\mathcal{A}_+}, \quad h|_{\mathcal{A}_-}=u_a|_{\mathcal{A}_-}-u|_{\mathcal{A}_-},\quad\mu|_\mathcal{I}=0,
\end{aligned}
\end{equation}
where $\mathcal{J}_0'(u)=p_0+\alpha u$, and $v|_{\mathcal{S}}$ denotes the restriction of a function $v:\Omega\to\mathbb{R}$ to a set 
$\mathcal{S}\subset\Omega$. Notice that only $h|_{\mathcal{I}}$ and $\mu|_{\mathcal{A}}$ are our desired unknown variables now.

Next, let $E_{\mathcal{I}}$ denote the extension-by-zero operator from $\mathcal{I}$ to $\Omega$. Then, $E_{\mathcal{I}}^*$ is the restriction operator from $\Omega$ to $\mathcal{I}$. The operators $E_{\mathcal{A}_+}$, $E_{\mathcal{A}_+}^*$, $E_{\mathcal{A}_-}$, $E_{\mathcal{A}_-}^*$ and $E_{\mathcal{A}}$, $E_{\mathcal{A}}^*$ are defined analogously. Note that, for instance, $\mu|_{\mathcal{I}}=E_\mathcal{I}^*\mu$. For convenience, below we will use both notations for restriction operators. With these definitions and noting from the second and third equation in \eqref{eq:complimentary} that
$$
h|_\mathcal{I}=-\alpha^{-1}(t|_\mathcal{I}+\mathcal{J}_0'(u)|_\mathcal{I}),
$$
and
$$
\mu|_\mathcal{A}=-t|_\mathcal{A}-\alpha E^*_{\mathcal{A}}(E_{\mathcal{A}_+}(u_b-u)|_{\mathcal{A}_+}+
E_{\mathcal{A}_-}(u_a-u)|_{\mathcal{A}_-})-\mathcal{J}_0'(u)|_\mathcal{A},
$$
we can reduce the system in \eqref{eq:complimentary} to solving 
\begin{equation}\label{eq:mh:5}
(K+\alpha^{-1}E_\mathcal{I}E_\mathcal{I}^*)t=E_{\mathcal{A}_+}(u_b-u)|_{\mathcal{A}_+}+
E_{\mathcal{A}_-}(u_a-u)|_{\mathcal{A}_-}-\alpha^{-1}E_\mathcal{I}\mathcal{J}_0'(u)|_\mathcal{I}.
\end{equation}
for $t\in B$. Backward substitution then yields $h|_\mathcal{I}$ and $\mu|_\mathcal{A}$.

Utilizing the above considerations, PDAS solves \eqref{eq:linear_kkt} iteratively by estimating the active and inactive sets and solving the associated linear system of the type \eqref{eq:mh:5} in every iteration. 
In this context, the active set estimation works as follows: Assume that a current iterate $(h^l,\mu^l)\in L^2(\Omega)^2$, $l\in\mathbb{N}$, is available. Then the next active and inactive set estimates are determined by
	\begin{align*}
	\mathcal{A}_+^{l+1}&:=\{\mu^l+\lambda(u+h^l-u_b)>0\},  &
		\mathcal{A}_-^{l+1}&:=\{\mu^l+\lambda(u+h^l-u_a)<0\}, \\
	\mathcal{A}^{l+1}&:=\mathcal{A}_+^{l+1}\cup\mathcal{A}_-^{l+1}, &
		\mathcal{I}^{l+1}&:=\Omega\setminus\mathcal{A}^{l+1}.
	\end{align*}
These sets are then used in \eqref{eq:mh:5}, respectively \eqref{eq:complimentary}, to obtain $(h^{l+1},\mu^{l+1})$. Unless some stopping rule is satisfied, PDAS returns to the next set estimation. We refer to \cite{HIK} for more details on PDAS including convergence considerations. The choice of numerical solvers for \eqref{eq:mh:5} may depend on the size of the system after discretization. In our situation, the standard Matlab backslash is sufficient already.	In our tests below, the PDAS iterations are terminated if the $L^2(\Omega)$-norm residual of the second equation in \eqref{eq:linear_kkt} drops below $10^{-16}$ or a maximum of $50$ iterations is reached.

The second application of PDAS is connected to numerically solving the nonlinear system in \eqref{eq:KKT_sub}.  
 Our strategy here is to decouple the system  \eqref{eq:KKT_sub} into the following two subsystems:
\begin{equation}\label{eq:KKT_sub1}\left\{
\begin{aligned}
&-\Delta d_\epsilon + D_\epsilon(y;d_\epsilon)-h=0\; \text{ in } \Omega,\;\text{ and}\; d_\epsilon =0 \text{ on } \partial \Omega,\\
&-\Delta p_\epsilon + \partial_{d_\epsilon}(D_\epsilon(y;d_\epsilon)) p_\epsilon=y-g\; \text{ in } \Omega,\;\text{ and}\; p_\epsilon =0 \text{ on } \partial \Omega,
\end{aligned}\right.
\end{equation}
and
\begin{equation}\label{eq:KKT_sub2}\left\{
\begin{aligned}
&Q h+ p_\epsilon  +\mu=-\alpha u,\\
&\mu -\max(0,\mu+\lambda(u+h-u_b))-\min(0,\mu+\lambda(u+h-u_a))=0.
\end{aligned}\right.
\end{equation}
Then, in our implementation of Step 2 of Algorithm \ref{alg:non_smooth_sqp}, while the $L^2(\Omega)$ residual norm of the system \eqref{eq:KKT_sub} is larger than $10^{-16}$, or the iteration count is smaller than $50$,
we use a consecutive and iterative way to implement the following:
\begin{itemize}[leftmargin=20pt]
\item[(i)] First we run Newton algorithm for \eqref{eq:KKT_sub1} to get an update to $d_\epsilon$ and $p_\epsilon$ for a fixed $h$. The algorithm can be initialized using the solution from its last round, and zeros for the first round;

\item[(ii)] Using the newly computed $d_\epsilon$ and $p_\epsilon$ from (i), we apply PDAS to \eqref{eq:KKT_sub2} to obtain updates of $\mu$ and $h$. The PDAS step is similar to the one we have described above. Only the terms pertinent to the new quantities in \eqref{eq:KKT_sub1} and \eqref{eq:KKT_sub2} are adapted. Especially, now $Q$ is associated with the functional $q_\epsilon(h,h)=\langle\Pi_{\epsilon}(u,\cdot)h,\Pi_{\epsilon}(u,\cdot)h\rangle+ \alpha \langle h,h\rangle$. However, we note here that in our experiments, for each PDAS iteration for \eqref{eq:KKT_sub2}, we found that using simply the quadratic functional $q$ from the initialization step above gives almost the same convergence behavior than using $q_\epsilon$ connected to $\Pi_{\epsilon}$. 
\end{itemize}

\section{Numerical results}\label{sec:numerics}
In this section, we demonstrate the practical performance of our proposed algorithm for solving optimal control problems with nonsmooth partial differential equations which contain ReLU network components.
\paragraph{\textbf{Parameter setting of Algorithm \ref{alg:non_smooth_sqp}}}
In our algorithm, we set $\eta=10^{-16}$, $\tau_{min}=10^{-16}$, $\epsilon_0=\delta_0=10^{-1}$, $c=0.6$, $c_1=c_2=0.1$, $\beta=1.1$, $\tilde{c}=0.5$. The parameter $\nu$ will be set depending on the value of $\alpha$ and the respective example.
In all the tested examples, we use  finite differences  for the PDE discretization, and in particular the standard five-point stencil for  the discrete Laplacian. The algorithm is terminated if $\norm{h}_{L^2(\Omega)}\leq 10^{-16}$.
For solving both the state equation and the adjoint equation, we use a (semismooth) Newton method \cite{HIK}, with the stopping rule on checking the $H^{-1}(\Omega)$-norm of the residual. Specifically, if the residual norm is smaller than $10^{-16}$ or the number of iterations is bigger than $50$, then we stop the Newton solver.
Numerical calculations were performed on a laptop with Intel Core i7-10850H CPU and 64GB memory using Matlab R2020b.

\subsection{Application to PDE with single max-function} 
Here we first show the result of our algorithm when applied to an example presented in \cite{christof}.
We choose $\Omega= (0,1)\times (0,1)$ to be the unit square, and design the exact solution and its adjoint state of the optimal control problem to be
\begin{equation}\label{eq:exact_sol}
y=p= \left\{ 
\begin{aligned}
 \big(\big(x_1-\frac{1}{2}\big)^4 +\frac{1}{2}  \big(x_1-\frac{1}{2}\big)^3\big)\sin (\pi x_2) &\quad x_1<\frac{1}{2},\\
 0 &\quad x_1\geq \frac{1}{2}.
\end{aligned}\right.
\end{equation}
No active control constraint is considered in this example for simplicity.
The state equation is given by the following second-order semilinear elliptic PDE:
\[ -\Delta y +\max(0,y)=u+f \text{   in  } \Omega, \quad \text{ and } y=0 \text{ on } \partial\Omega, \]
where, given $y$ and $p$, the optimal control $u$ and the given function $f$ can also be explicitly calculated using the KKT condition of the optimal control problem. Note that introducing a given function $f$ into the PDE does neither change the analysis nor the algorithm.
Both $y$ and $p$ are twice continuously differentiable and have the value zero on the right half of $\Omega$. Therefore, the nonsmoothness of the $\max$-function in the state equation at the solution appears on a set of positive measure in this example. This renders the control-to-state map nonsmooth at the solution $(u,y)$. We test our algorithm by using different discretization sizes $\mathrm{dx}$ (uniform in both dimensions) and with respect to variants of the control cost $\alpha$.
Particularly, in all the numerical tests provided in this paper,
	we consider the following $C^2$-smooth approximation of the max-function in $\mathcal{D}_\epsilon$:
	\begin{equation}\label{eq:max_reg}
	\sigma_\epsilon(t)=\left\{ \begin{array}{ll}
	t - \frac{\epsilon}{2}, & \text{ if } t \geq \epsilon,\\
	\frac{t^3}{\epsilon^2} -\frac{t^4}{2\epsilon^3}, & \text{ if } t \in (0,\epsilon),\\
	0, & \text{ if } t \leq 0,
	\end{array} \right.
	\end{equation}
for given $\epsilon>0$.
The numerical results are reported in Table \ref{tab:max_example}.
\renewcommand{\arraystretch}{1.2}
\begin{table}[!ht]
	\begin{center}
		\resizebox{0.75\textwidth}{!}{
			\begin{tabular}{ |l|lll|lll|}
				\hline
				&\multicolumn{3}{c|}{$\alpha=10^{-1}$, $\nu=0.9$  }&  \multicolumn{3}{c|}{$\alpha=10^{-2}$, $\nu=0.9$  }   \\ \hline
				Mesh size	&	Cost	& $\norm{u-u_h}/\norm{u}$	&  $\norm{y-y_h}/\norm{y}$ 	& Cost		& $\norm{u-u_h}/\norm{u}$	&  $\norm{y-y_h}/\norm{y}$  \\ \hline
				dx=1/16	&	{$0.0261 $} 	& 	{$ 0.0506$}  &	{$0.0234$} 	&{$0.0263$} 	& 	{$ 0.044$}  &	{$0.2021$}  \\ \hline
				dx=1/32	&	{$0.0329  $} 	& 	{$ 0.013$}  &	{$ 0.0058$} 	&	{$0.0331  $} 	& {$ 0.0116$}  &	{$ 0.052$} 	  \\ \hline
				dx=1/64	&	{$0.0368  $} 	& 	{$ 0.0029$}  &	{$ 0.0012$} 	&{$0.037  $} 	& 	{$0.003$}  &	{$ 0.0131$} 	 \\ \hline	
				dx=1/128  &	{$0.0389  $} 	& 	{$ 6.07\times10^{-4}$}  &	{$ 1.03\times10^{-4}$} 	&{$0.039  $} 	& 	{$7.1\times10^{-4}$}  &	{$ 0.0031$} 	 \\ \hline	
				&\multicolumn{3}{c|}{$\alpha=10^{-3}$, $\nu=0.9$  }&  \multicolumn{3}{c|}{$\alpha=10^{-4}$, $\nu=0.9$  }   \\ \hline
		Mesh size	&	Cost	& $\norm{u-u_h}/\norm{u}$	&  $\norm{y-y_h}/\norm{y}$ 	& Cost		& $\norm{u-u_h}/\norm{u}$	&  $\norm{y-y_h}/\norm{y}$  \\ \hline
				dx=1/16	&	{$0.0281 $} 	& 	{$ 0.0216$}  &	{$0.7755$} 	&{$0.0455$} 	& 	{$ 0.0057$}  &	{$1.3518$}  \\ \hline
				dx=1/32	&	{$0.0349  $} 	& 	{$ 0.0057$}  &	{$ 0.2003$} 	&	{$0.0523  $} 	& {$ 0.0015$}  &	{$ 0.3501$} 	  \\ \hline
				dx=1/64	&	{$0.0387  $} 	& 	{$ 0.0015$}  &	{$ 0.0511$} 	&{$0.0562  $} 	& 	{$ 3.94\times10^{-4}$}  &	{$ 0.0896$} 	 \\ \hline
				dx=1/128	&	{$0.0408  $} 	& 	{$ 3.73\times10^{-4}$}  &	{$ 0.0128$} 	&{$0.0582  $} 	& 	{$1.0\times10^{-4}$}  &	{$ 0.0226$} 	 \\ \hline	
				&\multicolumn{3}{c|}{$\alpha=10^{-5}$, $\nu=0.9$  }&  \multicolumn{3}{c|}{$\alpha=10^{-6}$, $\nu=0.9$  }   \\ \hline
		Mesh size	&	Cost	& $\norm{u-u_h}/\norm{u}$	&  $\norm{y-y_h}/\norm{y}$ 	& Cost		& $\norm{u-u_h}/\norm{u}$	&  $\norm{y-y_h}/\norm{y}$  \\ \hline
	dx=1/16	&	{$0.2194 $} 	& 	{$0.0011 $}  &	{$1.592$} 	&{$1.9592$} 	& 	{$ 2.35\times10^{-4} $}  &	{$1.6832$}  \\ \hline
	dx=1/32	&	{$0.2266  $} 	& 	{$ 3.08\times10^{-4}$}  &	{$ 0.4136$} 	&	{$1.9702  $} 	& {$ 8.32\times10^{-5}$}  &	{$ 0.4421$} 	  \\ \hline
	dx=1/64	&	{$0.2305  $} 	& 	{$ 8.06\times10^{-5}$}  &	{$ 0.1061$} 	&{$1.9744  $} 	& 	{$ 2.31\times10^{-5}$}  &	{$ 0.1159$} 	 \\ \hline	
	dx=1/128	&	{$0.2326  $} 	& 	{$ 2.06\times10^{-5}$}  &	{$ 0.027$} 	&{$1.9765  $} 	& 	{$ 6.06\times10^{-6}$}  &	{$ 0.0319$} 	 \\ \hline	
					&\multicolumn{3}{c|}{$\alpha=10^{-7}$, $\nu=0.9$  }&  \multicolumn{3}{c|}{$\alpha=10^{-8}$, $\nu=0.9$  }   \\ \hline
	Mesh size	&	Cost	& $\norm{u-u_h}/\norm{u}$	&  $\norm{y-y_h}/\norm{y}$ 	& Cost		& $\norm{u-u_h}/\norm{u}$	&  $\norm{y-y_h}/\norm{y}$  \\ \hline
	dx=1/16	&	{$19.3566 $} 	& 	{$3.27\times10^{-5} $}  &	{$1.7302$} 	&{$193.3311$} 	& 	{$ 3.4782\times10^{-6} $}  &	{$1.8797$}  \\ \hline
	dx=1/32	&	{$19.4059  $} 	& 	{$ 2.03\times10^{-5}$}  &	{$ 0.4713$} 	&	{$193.7626  $} 	& {$ 3.06\times10^{-6}$}  &	{$ 0.6281$} 	  \\ \hline
	dx=1/64	&	{$19.4128  $} 	& 	{$ 7.43\times10^{-6}$}  &	{$ 0.1330$} 	&{$193.7970  $} 	& 	{$ 2.09\times10^{-6}$}  &	{$ 0.3067$} 	 \\ \hline	
	dx=1/128	&	{$19.4151  $} 	& 	{$ 2.09\times10^{-6}$}  &	{$ 0.0468$} 	&{$193.8011  $} 	& 	{$ 7.65\times10^{-7}$}  &	{$ 0.2340$} 	 \\ \hline					
		\end{tabular}}
		\vspace{1em}
		\caption{Convergence performance of the proposed algorithm (for the single-max function problem) in terms of mesh size $\mathrm{dx}$ and the regularization parameter $\alpha$. The exact solution $y$ is given in \eqref{eq:exact_sol}, and $u=p/\alpha$ can be informed through the KKT system as we assumed that the constraint is not active.}
		\label{tab:max_example}
	\end{center}
\end{table}
We observe here that our algorithm can achieve quadratic convergence rates with respect to the mesh size $\mathrm{dx}$ as in \cite{christof}. When $\alpha$ becomes smaller, the convergence of the state variable becomes harder. It was reported in \cite{christof} that the semismooth Newton type method used there, achieved no convergence when $\alpha =10^{-6}$ or smaller. However, our method is capable of preserving the quadratic convergence when $\alpha= 10^{-6}$. For the case $\alpha <10^{-6}$, as provided in the last two groups in Table \ref{tab:max_example}, quadratic convergence rate can be observed in the case of $\alpha=10^{-7}$, and a suboptimal convergence rate appears when $\alpha =10^{-8}$. This shows that the proposed method is more robust for ill-conditioned problems.

In the following test examples, we also show that our proposed algorithm copes well with semilinear PDEs whose nonlinearities are general ReLU network functions. In this sense, our proposed method can be considered a genuine nonsmooth solver for such type of problems.

\subsection{Application to general multilayer ReLU network PDEs}
We consider  a ReLU neural network function $\mathcal{N}:\R \to \R$, with two-hidden-layers:
\begin{equation*}
\mathcal{N}(y)=\sum_{l=1}^L w_2^l\sigma\left(\sum_{k=1}^K w_1^{l,k} \sigma(w_0^k y + b_0^k) +b_1^l\right) + b_2.
\end{equation*}

\begin{figure}[htbp]
    \centering
\begin{minipage}[c]{\textwidth}
\begin{center}
	\includegraphics[width=0.4\textwidth]{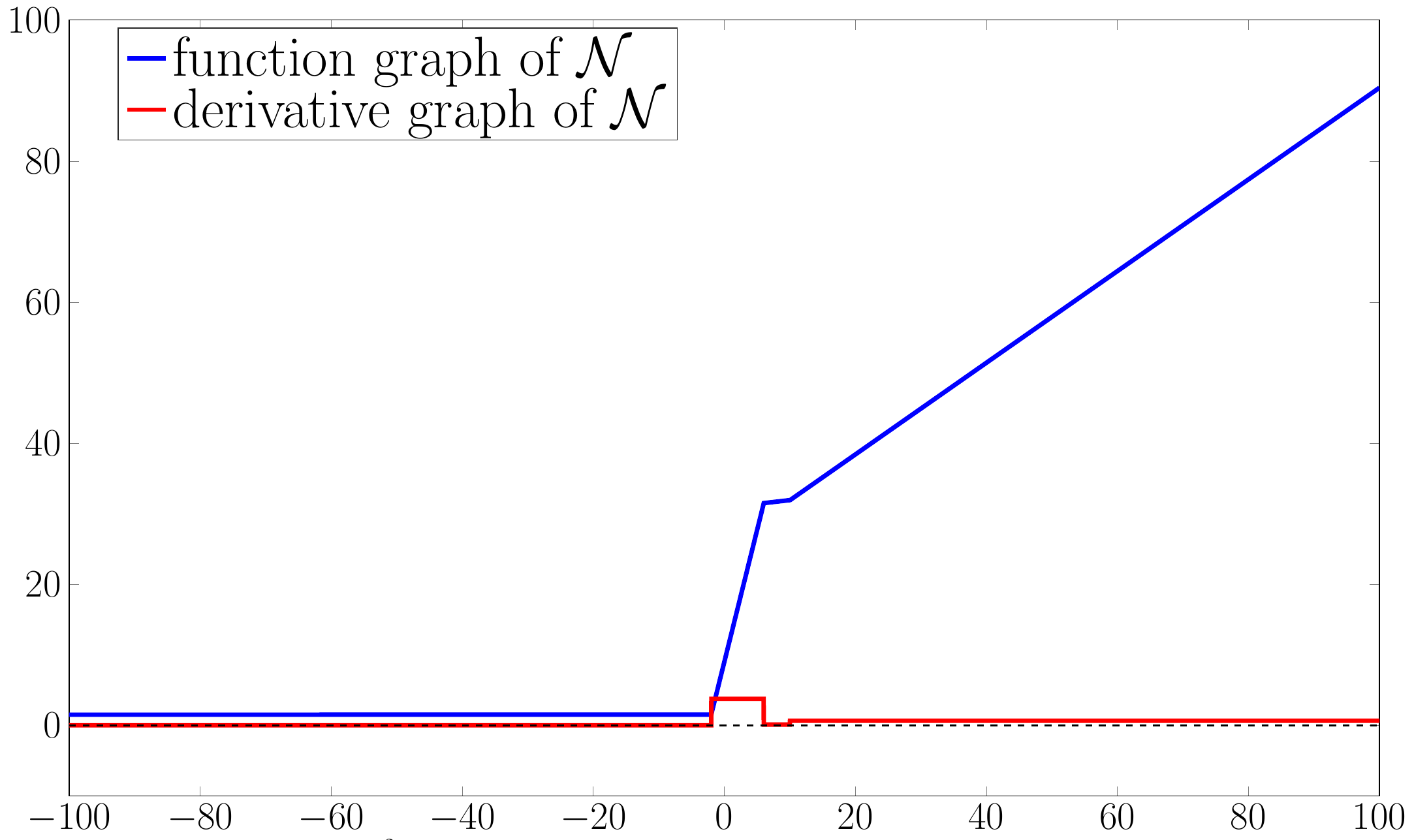}
	\includegraphics[width=0.4\textwidth]{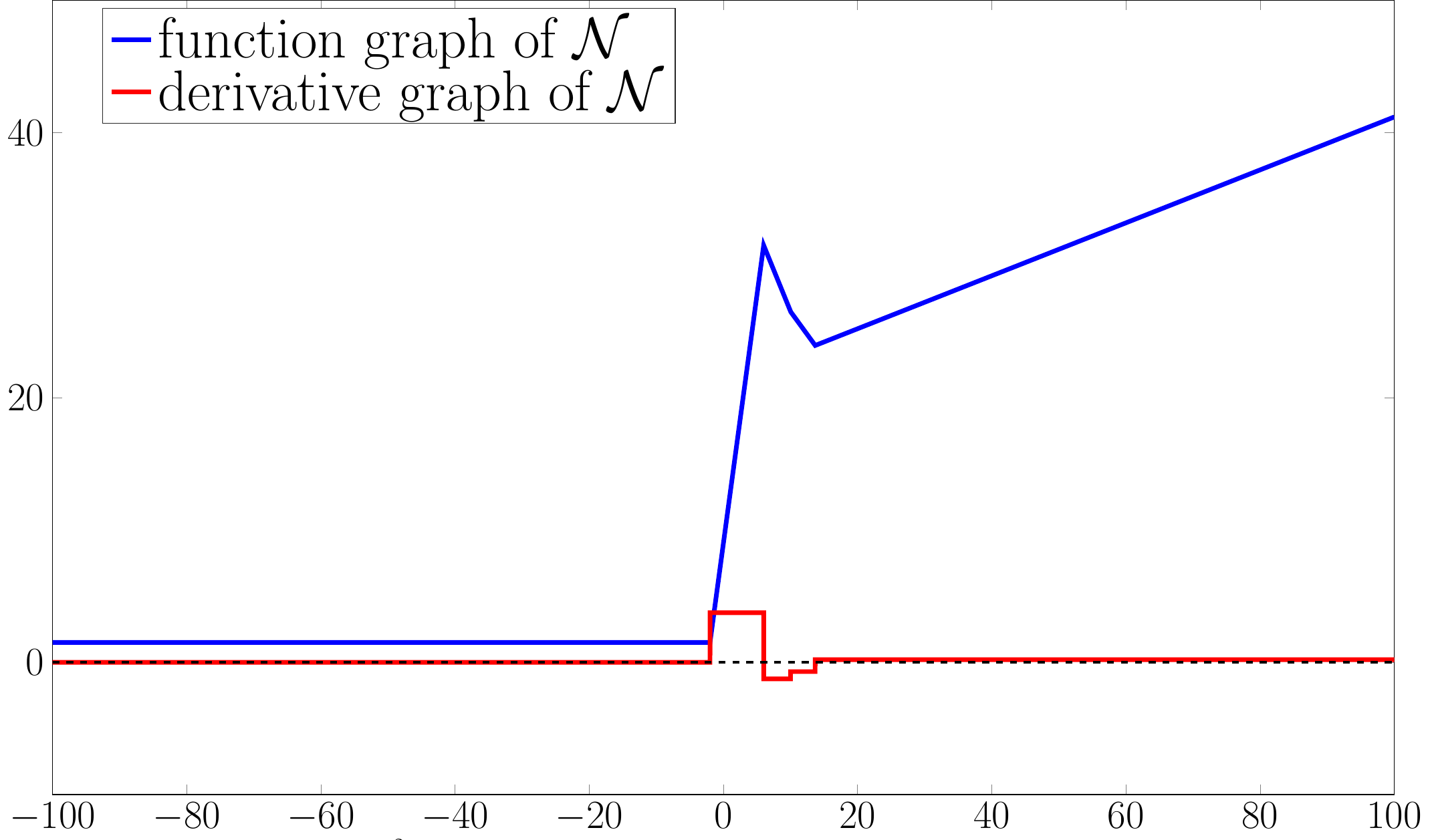}
\end{center}	
\end{minipage}\hfill
\vspace{1em}
 \begin{minipage}[c]{.6\textwidth}
 \centering
\resizebox{0.9\textwidth}{!}{
			\begin{tabular}{ |lll|lll|lll|}
				\hline
				\multicolumn{6}{|c|}{Weight parameters } 	& \multicolumn{3}{c|}{Bias  }    	\\ \hline
				{$w_0^1$}	& {$w_0^2$}	& {$w_0^3$} &\multicolumn{3}{c|}{} 	& {$b_0^1$  } & {$b_0^2$  }  & {$b_0^3$  } 	   \\ \hline
				$5$ & $0.1$  &  $10 $ & $-$ & $ -$ & $  -$ & $ 10$ &  $ -1$ & $   -60$      \\ \hline
				{$w_1^{1,1}$} 	& {$w_1^{1,2}$} 	& {$w_1^{1,3}$} 	&  {$w_1^{2,1}$} 	&  {$w_1^{2,2}$}	&  {$w_1^{2,3}$}	& $b_1^1$    &  $b_1^2$ &   \\ \hline
				$0.3$ & $2$  &  $-0.16$ & $  0.1 $ & $1$ &  $  -0.03\; (-0.12)$ & $ 0$ &  $  1$ & $   -$    \\ \hline
				$2$ & $-$  &  $-$ & $  1.5$ &$ -$ &  $  -$ &  $ 0$ &  $  -$ & $   -$  	 \\ \hline
		\end{tabular}}
\end{minipage}
\caption{Graphs and corresponding weights for the monotone and nonmonotone ReLU networks. The only difference   is their $w^{2,3}$-value ($-0.03$ vs $-0.12$).}
\label{fig:network_function1}
\end{figure}

Our  results here address optimal control problems for nonsmooth semilinear elliptic PDEs with both, monotone and nonmonotone network functions, respectively, as shown in Figure \ref{fig:network_function1}. 
For the sake of providing quantitative observations, we generate synthetic data by fixing the solution of the PDE. The data is generated from the function $ g_0=200\sin(\pi x)\sin(\pi y)$ and the control $u_0=\min(u_b,\max(u_a,-\Delta g_0+\mathcal{N}(g_0)))$ for $u_a=-1000$, $u_b=1000$, giving rise to a state $y_0$. 
 In this example, we choose $\Omega= (0,2)\times (0,2)$.
Then the function $g$ in the objective of \eqref{P} is computed numerically via the KKT-system for $(u_0,y_0)$.
Both test examples, respectively containing monotone and nonmonotone network functions, are generated in this way. 
We stress that the optimal control of PDEs involving ReLU neural network components, as proposed and studied in this paper, is a new feature in the literature, and our proposed algorithm is specific for the optimization with these type of PDE constraints. For this reason we refrain from comparing our algorithm with other (less tailored) methods for this set of examples.

Our numerical results are summarized in Table \ref{tab:results}. Here we collect three cases of discretization sizes with respect to varying cost parameter $\alpha$.
In all cases, we observe that $\mathfrak{m}(\Omega_{\mathcal{N}})$ is not zero at the solution rendering the control-to-state map genuinely nonsmooth. As a consequence, Step 2 in Algorithm \ref{alg:non_smooth_sqp} is always active.
In Table \ref{tab:results}, 'Cost' denotes the value of the objective functional of the optimal control problem at the final iterate, and 'Iterates' shows the number of outer iterations in Algorithm \ref{alg:non_smooth_sqp}. From the results reported in Table \ref{tab:results} we find that in both cases, monotone and nonmonotone, the algorithm exhibits a robust behavior across the scales of $\mathrm{dx}$ and $\alpha$. Specifically, the almost constant iteration count for varying $\mathrm{dx}$ can be associated with mesh-independent convergence of the algorithm. Moreover, in all cases highly accurate solutions could be obtained.
\begin{table}[!ht]
	\begin{center}
		\resizebox{0.8\textwidth}{!}{
			\begin{tabular}{ |l|llll|llll|}
				\hline
				& 	\multicolumn{4}{c|}{Monotone case} 	& \multicolumn{4}{c|}{Nonmonotone case}    \\ \hline
				&\multicolumn{4}{c|}{$\alpha=10^{-2}$ $\nu=0.7$  }&  \multicolumn{4}{c|}{$\alpha=10^{-2}$ $\nu=0.7$  }   \\ \hline
				Mesh size	&	Cost	& $\norm{h}$ & Iterates	& CPU time  & Cost	& $\norm{h}$ & Iterates	& CPU time  \\ \hline
				dx=1/16	&	{$2453.4 $} 	& 	{$ 2.0\times 10^{-27}$} &  34  &	{$0.5 $s} 	&{$2505.8$} &		{$ 2.0\times 10^{-27}$}  & 33 &	{$0.4 $s}  \\ \hline
				dx=1/32	&	{$2444.1  $} 	& 	{$ 3.4\times10^{-27}$}  &31	& 	{$1.5 $s} 	&	{$2496.1 $} 	& 	{$ 4.9\times10^{-27}$}  & 31	& 	{$1.5 $s} 	 	  \\ \hline
				dx=1/64	&	{$2441.6  $} 	& 	{$ 1.4\times 10^{-28}$}  &  34	&	{$23.9$s}  & 	{$2493.7 $} 	&  	{$ 3.3\times 10^{-27}$}  &	35 & 	{$39.0$s}    \\ \hline
				&\multicolumn{4}{c|}{$\alpha=10^{-10}$ $\nu=0.7$  }&  \multicolumn{4}{c|}{$\alpha=10^{-10}$ $\nu=0.7$  }   \\ \hline
				Mesh size	&	Cost	& $\norm{h}$	& Iterates & CPU time  	& Cost	& $\norm{h}$ & Iterates	& CPU time  \\ \hline
				dx=1/16	&	{$1.4531 \times 10^{-4} $} 	& 	{$ 7.9\times10^{-24}$}  &34	& 	{$1.3$s} 	&	{$1.4477\times10^{-4 }$} 	& 	{$ 4.2\times 10^{-24}$}  &34	& 	{$1.3  $s} 	  \\ \hline
				dx=1/32	&	{$1.4535 \times 10^{-4}    $} 	& 	{$ 5.4\times10^{-23}$}  &34	& 	{$4.7 $s} 	&	{$1.4474\times10^{-4}  $} 	& 	{$ 1.0\times10^{-22}$}  &34	& 	{$22.2 $s}   \\ \hline
				dx=1/64	&	{$ 1.4535 \times 10^{-4}  $} 	& 	{$ 3.9\times 10^{-21}$}  &34& 	{$23.4$s}  &{$ 1.4476\times10^{-4} $} 	& 	{$ 4.7\times 10^{-21}$}  &	34& 	{$23.8$s}      \\ \hline
				&\multicolumn{4}{c|}{$\alpha=10^{-16}$ $\nu=0.7$  }&  \multicolumn{4}{c|}{$\alpha=10^{-16}$ $\nu=0.7$  }    \\ \hline
				Mesh size	&	Cost	& $\norm{h}$	&Iterates & CPU time  	& Cost	& $\norm{h}$ & Iterates	& CPU time  \\ \hline
				dx=1/16	&	{$1.4531\times10^{-10}  $} 	& 	{$ 4.9\times10^{-17}$}  &55	& 	{$14.1$s} 	& {$1.4477\times10^{-10}  $} 	& 	{$ 8.4\times10^{-17}$}  & 54	& 	{$13.5$s} 	  \\ \hline
				dx=1/32	&	{$1.4535\times10^{-10}  $} 	& 	{$ 8.5\times10^{-17}$}  &55	& 	{$52.2 $s} 	&	{$1.4474\times10^{-10}  $} 	& 	{$ 8.1\times10^{-17}$}  & 55	& 	{$51.6 $s}	  \\ \hline
				dx=1/64	&	{$1.4535\times 10^{-10} $} 	& 	{$  5.9\times 10^{-17}$}  & 55	& 	{$273.6$s}  & {$1.4476 \times10^{-10} $} 	& 	{$ 5.9\times 10^{-17}$}  & 55	& 	{$ 283.3 $s}    \\ \hline
		\end{tabular}}
		\vspace{1em}
		\caption{Convergence performance of the proposed algorithm for monotone and nonmonotone ReLU neural network functions.}	
		\label{tab:results}
	\end{center}
\end{table}
In our computations, we also tested the algorithm in the extreme case of $\alpha =10^{-16}$, which however exhibits still a similar performance as for the last set of examples for both monotone  and  nonmonotone functions. The only difference for this case is that we used $\lambda=10^{-6}$ in \eqref{eq:linear_kkt} rather than $\lambda=\alpha$ as in the other cases.

\section{Conclusion}\label{sec:conclusion}
In this paper, we have studied numerical aspects of optimal control problems with ReLU-network-informed PDEs.
 It was firstly shown that a canonical smoothing  of a ReLU network, though practically very plausible, cannot always preserve its monotonicity, something that could imply lack of uniqueness of solutions for the corresponding  ReLU-network-informed PDEs. Therefore traditional numerical approaches relying on such smooth approximations may encounter difficulties in the solution process.
This motivates us to propose a genuine nonsmooth algorithm which respects the specific structure of ReLU networks in the PDEs.
 The proposed approach does not smoothen the state equation itself, but it rather approximates the derivative of the control-to-state map via smoothing of the max-function appearing at the directional derivatives.
Such approximations were proven to converge strongly to the original directional derivative of the nonsmooth operator in a vanishing smoothing regime. Moreover, this approximation process allows to identify descent directions of the reduced optimal control problem with respect to the nonsmooth PDEs at a given control iterate.
In our numerical tests, the proposed algorithm performs more robust in a benchmark  optimal control problem when compared to  recent nonsmooth algorithms designed  specifically for the optimal control of PDEs with a single max-function. In addition, our algorithm also works well for optimal control of semilinear elliptic PDEs with deeper ReLU network functions, which have a more general nonsmooth structure when compared to a single max-function.

\subsection*{Acknowledgments}

This work is supported by the Deutsche Forschungsgemeinschaft (DFG, German Research Foundation) under Germany's Excellence Strategy -- The Berlin Mathematics Research Center MATH+ (EXC-2046/1, project ID: 390685689).
The work of GD is supported by an NSFC grant, No. 12001194.
The work of MH is partially supported by the DFG SPP 1962, project-145r. KP would like to thank Amal Alphonse for useful discussions.

\bibliographystyle{plain}
\bibliography{kostasbib}
\end{document}